\pdfoutput=1
\documentclass[11pt,letterpaper]{amsart}

\usepackage{graphicx}
\usepackage[english]{babel}
\usepackage[T1]{fontenc}
\usepackage[latin1]{inputenc}
\usepackage{amsfonts}
\usepackage{amssymb}
\usepackage{amsthm}
\usepackage{amsmath}
\usepackage{tikz}
\usepackage{float}
\usepackage{enumerate}
\usepackage{accents}
\usepackage{mathtools}
\usepackage{mathrsfs}
\usepackage{comment}
\usepackage{afterpage}
\usepackage{bbm}
\usepackage{dsfont}
\usepackage{stmaryrd}
\usepackage{hyperref}
\usepackage{mleftright}
\usepackage{subfig}
\usepackage{soul}
\usepackage{tikz-cd} 
\usepackage{cleveref}

\numberwithin{equation}{section}

\theoremstyle{plain}
\newtheorem{thm}{Theorem}[section]
\newtheorem*{thm*}{Theorem}
\newtheorem{prop}[thm]{Proposition}
\newtheorem*{prop*}{Proposition}
\newtheorem{cor}[thm]{Corollary}
\newtheorem*{cor*}{Corollary}
\newtheorem{lem}[thm]{Lemma}

\newtheorem{thmintro}{Theorem}

\theoremstyle{definition}
\newtheorem{defn}[thm]{Definition}
\newtheorem*{defn*}{Definition}
\newtheorem{ex}[thm]{Example}
\newtheorem{rmk}[thm]{Remark}
\newtheorem*{rmk*}{Remarks}

\newtheorem*{conj*}{Conjecture}
\newtheorem*{quest*}{Question}

\newtheorem{nota}[thm]{Notation}

\newtheoremstyle{blue-environment}{}{}{}{}{\color{blue}\bfseries}{.}{ }{}
\theoremstyle{blue-environment}

\newcommand{\acts}{\curvearrowright}
\newcommand{\ra}{\rightarrow}

\newcommand{\sq}{\subseteq}

\newcommand{\x}{\times}

\newcommand{\id}{\mathrm{id}}

\newcommand{\mbb}{\mathbb}
\newcommand{\mc}{\mathcal}
\newcommand{\mf}{\mathfrak}
\newcommand{\mscr}{\mathscr}

\newcommand{\R}{\mathbb{R}}
\newcommand{\Z}{\mathbb{Z}}
\newcommand{\N}{\mathbb{N}}
\newcommand{\Q}{\mathbb{Q}}

\newcommand{\s}{\sigma}
\newcommand{\eps}{\epsilon}
\newcommand{\Om}{\Omega}
\newcommand{\om}{\omega}
\newcommand{\g}{\gamma}
\newcommand{\G}{\Gamma}

\newcommand{\Out}{{\rm Out}}
\newcommand{\Aut}{{\rm Aut}}

\newcommand{\X}{\mc{X}}

\DeclareMathOperator{\lk}{lk}
\DeclareMathOperator{\St}{st}
\DeclareMathOperator{\Fix}{Fix}
\DeclareMathOperator{\Min}{Min}

\DeclareMathOperator{\ar}{ar}
\DeclareMathOperator{\Gr}{Gr}
\DeclareMathOperator{\rk}{rk}
\newcommand{\utimes}{\mathbin{\rotatebox[origin=c]{-90}{$\ltimes$}}}

\begin{document}

\title{Growth of automorphisms of virtually special groups} 

\author[E.\,Fioravanti]{Elia Fioravanti}\address{Institute of Algebra and Geometry, Karlsruhe Institute of Technology}\email{elia.fioravanti@kit.edu} 
\thanks{The author is supported by Emmy Noether grant 515507199 of the DFG}

\begin{abstract}
    We study the speed of growth of iterates of outer automorphisms of virtually special groups, in the Haglund--Wise sense. We show that each automorphism grows either polynomially or exponentially, and that its stretch factor is an algebraic integer. For coarse-median preserving automorphisms, we show that there are only finitely many growth rates and we construct an analogue of the Nielsen--Thurston decomposition of surface homeomorphisms.

    These results are new already for right-angled Artin groups. However, even in this particular case, the proof requires studying automorphisms of arbitrary special groups in an essential way.

    As results of independent interest, special groups are accessible over centralisers, and they have a canonical JSJ decomposition over centralisers. For any virtually special group $G$, the outer automorphism group $\Out(G)$ is boundary amenable, satisfies the Tits alternative, and has finite virtual cohomological dimension.
\end{abstract}

\maketitle

\section{Introduction}

Given an infinite, finitely generated group $G$, a fundamental problem is describing the structure of its outer automorphism group $\Out(G)$. From a topological perspective, this corresponds to understanding self homotopy equivalences of any classifying space of $G$. Unfortunately, the properties of $G$ are rarely reflected into those of $\Out(G)$ in a straightforward way, and there are virtually no general tools to study the structure of automorphisms.

Some of the most elementary groups --- such as free and surface groups --- give rise to some of the most interesting and intricate outer automorphism groups --- such as $\Out(F_n)$ and mapping class groups --- whose study has occupied a central place in low-dimensional topology and geometric group theory in the past decades. At the same time, even harmless-looking groups such as the Baumslag--Solitar group $BS(2,4)$ can behave wildly: its outer automorphism group is not finitely generated \cite{Collins-Levin}.

The only general class of groups for which we have a near-complete understanding of automorphisms is arguably that of \emph{Gromov-hyperbolic} groups \cite{Gromov-hyp}. Breakthroughs of Rips and Sela showed that every hyperbolic group $G$ has a canonical JSJ decomposition \cite{RS97}, 
and that $\Out(G)$ is finitely generated and completely encoded in this decomposition \cite{Rips-Sela,Sela-hypJSJ}, leading to the solution of the isomorphism problem in this case \cite{Sela-isom,Dahmani-Guirardel}. Regrettably, our understanding of automorphisms for general families of groups has not evolved qualitatively since Rips and Sela's work in the 90s. Many of their results have been extended to \emph{relatively hyperbolic} groups \cite{Drutu-Sapir08,BS08,Groves-AGT,Dahmani-Groves,GL-relhyp}, though always relying on somewhat similar techniques that require hyperbolicity in fundamental ways.

The main goal of this article is to move past the restrictions of negative curvature and develop a general theory of automorphisms for a broad and natural family of \emph{non-positively curved} groups. We have found that the ideal setting for such a theory is provided by the (compact) special groups of Haglund and Wise \cite{HW08}. This is the most important class of \emph{cubulated} groups --- marrying a wealth of examples with a powerful toolkit to study them --- so our choice was also motivated by recent questions of Rips on the structure of automorphisms of cubulated groups \cite[p.\,826]{Sela-new1}.

New approaches to study automorphisms of non-hyperbolic groups were recently proposed by Groves and Hull \cite{Groves-Hull} and by Sela \cite{Sela-new1,Sela-new2}, both relying on (weakly) acylindrical actions on hyperbolic spaces. What is different in our own approach is that we completely give up on acylindricity, rather embracing the fact that we will have to work with non-small $\R$--trees and non-acylindrical graphs of groups, and resolving by other means the issues that this causes. To some extent, this choice is a forced one: there exist special groups $G$ such that $\Out(G)$ is infinite and such that $G$ has no small actions on $\R$--trees \cite[p.\,166]{Fio10a}, so not all automorphisms of $G$ are ``seen'' by its (strongly) acylindrical actions on hyperbolic spaces. 

Proving that $\Out(G)$ is finitely generated for all special groups $G$ remains out of reach for the moment (although the techniques we develop here bring this question much closer to what can be reasonably tackled). Instead, we focus on the problem of analysing how fast the length of an element $g\in G$ can grow when we apply iterates of an outer automorphism $\phi\in\Out(G)$. 

Even in the very particular case when $G$ is a right-angled Artin group (RAAG), very little seems to be known on growth, despite the fact that several refined aspects of automorphisms of RAAGs are well-understood \cite{Day,Day-Wade,BCV}, including finite generation of their automorphism groups \cite{Laurence,Servatius}. In fact, our approach to analyse growth of automorphisms of RAAGs requires studying automorphisms of \emph{arbitrary} special groups in a key way: it relies on a hierarchical construction inspired by Nielsen--Thurston decompositions and \cite{Sela-Duke}, by which we reduce the study of an automorphism of a RAAG to the study of a ``simpler'' automorphism of an (a priori) general special group.

We now make precise the notion of ``length'' of a group element $g\in G$, and ``growth'' of an outer automorphism $\phi\in\Out(G)$. For the former, it is natural to consider the \emph{conjugacy length} $\|g\|$, i.e.\ the minimum word length of an element in the conjugacy class of $g$, fixing once and for all a finite generating set of $G$, whose choice will play no role. If $G$ is the fundamental group of a compact Riemannian manifold, then $\|g\|$ is roughly equal to the length of a shortest closed geodesic in the free homotopy class determined by $g$. We then define the \emph{growth rate} of $g$ under $\phi$ as the equivalence class of the sequence $n\mapsto\|\phi^n(g)\|$ up to bi-Lipschitz equivalence (see \Cref{sub:growth}). More crudely, one can define the \emph{stretch factor} of $\phi$ as
\[ {\rm str}(\phi):=\sup_{g\in G}\,\limsup_{n\ra+\infty}\,\|\phi^n(g)\|^{1/n} . \]

These notions are best understood in three fundamental examples: free abelian groups $\Z^m$, closed surface groups $\pi_1(S)$, and free groups $F_m$. Growth of elements of ${\rm GL}_m(\Z)$ is easily described in terms of the Jordan decomposition, while growth of elements of the mapping class group of $S$ is encoded in their Nielsen--Thurston decomposition \cite{Thurston-BAMS}. Analysing growth of automorphisms of free groups $F_m$ proved to be a more complex problem, which required refined techniques inspired by train tracks on surfaces \cite{BH92,BFH1,BFH2,Bridson-Groves} and was finally solved by Levitt \cite{Levitt-GAFA}. In all three examples, each outer automorphism admits only finitely many growth rates as the element $g$ varies in $G$, and each growth rate is bi-Lipschitz equivalent to a sequence $n\mapsto n^p\lambda^n$ for some $p\in\N$ and $\lambda\geq 1$. The number $\lambda$ happens to be an algebraic integer and a weak Perron number\footnote{An \emph{algebraic integer} is a root of a monic polynomial with integer coefficients. It is a \emph{weak Perron number} if its modulus is not smaller than that of any of its Galois conjugates.}.

Automorphisms of Gromov-hyperbolic groups (and toral relatively hyperbolic groups) have a similar growth behaviour, as was announced by \cite{CHHL}. At the same time, almost nothing is known on growth of automorphisms of non-hyperbolic groups and all classical techniques to approach this problem --- mainly train tracks and JSJ decompositions --- are known to fail or run into serious issues when one abandons the world of (relatively) hyperbolic groups. Even restricting to the rather tame world of RAAGs, almost no growth information seems to be available. 

What is more, Coulon recently constructed finitely generated groups with automorphisms of ``intermediate growth'' \cite{Coulon}: they grow faster than any polynomial and slower than any exponential. It remains unknown if these exotic behaviours can occur for automorphisms of finitely presented groups, but there seems to be no reason to expect the contrary.

Our first result is that the top growth rate of automorphisms of special groups is rather well-behaved. In particular, this applies to all automorphisms of RAAGs. We refer to \Cref{thm:tame} for a more precise statement.

\begin{thmintro}\label{thmintro:general_aut}
    Let $G$ be a virtually special group and let $\phi\in\Out(G)$.
    \begin{enumerate}
        \item The stretch factor ${\rm str}(\phi)$ is a weak Perron number.
        \item If ${\rm str}(\phi)=1$, then $\phi$ grows at most polynomially.
        \item If ${\rm str}(\phi)>1$, then ${\rm str}(\phi)$ is realised on some $H\leq G$ that is a surface group, a free product, or a group with a free abelian direct factor.
    \end{enumerate}
\end{thmintro}

Part~(3) roughly means that the only source of exponential growth are pseudo-Anosovs on compact surfaces, fully irreducible automorphisms of free products, and skewing automorphisms of direct products.

\Cref{thmintro:general_aut} does \emph{not} describe the growth rates of the elements of $G$ under $\phi$ up to bi-Lipschitz equivalence. However, if $G$ is $1$--ended and ${\rm str}(\phi)$ is not realised on any maximal direct product inside of $G$, then we show that the top growth rate of $\phi$ is purely exponential. Namely, for every $g\in G$, the sequence $n\mapsto\|\phi^n(g)\|$ is either bi-Lipschitz equivalent to $n\mapsto\lambda^n$ for $\lambda={\rm str}(\phi)$, or exponentially slower (see part~(3) of \Cref{prop:pure_above_osing}).

We obtain a more complete description of growth rates under stronger assumptions on the automorphism. Specifically, we need $\phi$ to preserve the coarse median structure on $G$ induced by a convex-cocompact embedding into a RAAG. Automorphisms of RAAGs are coarse-median preserving precisely when they are \emph{untwisted} \cite{Fio10a}, a class of automorphisms that is well-studied and important in its own right \cite{CSV}. In addition, \emph{all} automorphisms of right-angled Coxeter groups are coarse-median preserving, and so are all automorphisms of Gromov-hyperbolic groups. See \cite[Theorem~D]{FLS} for further examples.

\begin{thmintro}\label{thmintro:cmp_aut}
    Let $G$ be special. Let $\phi\in\Out(G)$ be coarse-median preserving. 
    \begin{enumerate}
        \setlength\itemsep{.25em}
        \item There exist finitely many weak Perron numbers $\lambda_1,\dots,\lambda_m>1$ and an integer $P\in\N$ such that the following holds. For every $g\in G$, the sequence $n\mapsto\|\phi^n(g)\|$ is either bi-Lipschitz equivalent to $n\mapsto n^p\lambda_i^n$ for some integer $0\leq p\leq P$ and some index $1\leq i\leq m$, or it is eventually bounded by the sequence $n\mapsto n^P$.
        \item For any growth rate $\mf{o}=[n\mapsto n^p\lambda_i^n]$, let $\mc{K}(\mf{o})$ be the family of subgroups of $G$ all of whose elements grow at speed at most $\mf{o}$ under $\phi$. Then there are only finitely many $G$--conjugacy classes of maximal subgroups in $\mc{K}(\mf{o})$, and each of these is quasi-convex\footnote{Quasi-convexity of a subgroup $K\leq G$ is meant in the coarse-median sense and implies that $K$ is undistorted in $G$ and itself special. Equivalently, $K$ is convex-cocompact in $G$.} in $G$.
    \end{enumerate}
\end{thmintro}
 
\Cref{thmintro:cmp_aut} extends to many groups that are just \emph{virtually} special, including all right-angled Coxeter groups; see \Cref{rmk:fi_cmp_fix}. Part~(2) of the theorem can be regarded as an analogue of the Nielsen--Thurston decomposition associated to a surface homeomorphism. \Cref{thmintro:cmp_aut}(2) is easily seen to fail for \emph{general} automorphisms of special groups \cite[Example~3.4]{Fio11a}. 

We should clarify that we do \emph{not} show that sub-polynomial growth rates are \emph{exactly} polynomial in \Cref{thmintro:cmp_aut} and, therefore, we leave open the rather unlikely possibility that there are infinitely many of the latter. 

Part of the proofs of Theorems~\ref{thmintro:general_aut} and~\ref{thmintro:cmp_aut} is based on procedures that relate the group $\Out(G)$ to the groups $\Out(P)$, where $P\leq G$ are lower-complexity special groups (\Cref{prop:product_embedding}). One important consequence is the following. When $G$ is a RAAG, this was shown in \cite{BGH22,CV-BLMS,Horbez14} (see \cite{Ham-exact,Kida,CV86,BFH1,BFH2} when $G$ is free or surface).

\begin{thmintro}\label{thmintro:exact}
    For any virtually special group $G$:
    \begin{enumerate}
        \item $\Out(G)$ is boundary amenable;
        \item $\Out(G)$ is virtually torsion-free with finite cohomological dimension;
        \item $\Out(G)$ satisfies the Tits alternative and contains no Baumslag--So\-litar subgroups ${\rm BS}(m,n)$ with $|m|\neq|n|$.
    \end{enumerate}
\end{thmintro}

These are some of the rare properties known to hold for all outer automorphism groups of virtually special groups; the only other one seems to be residual finiteness \cite{AMS}. Boundary amenability also implies that $\Out(G)$ satisfies the Novikov conjecture on higher signatures \cite{Higson,Baum-Connes-Higson}.

\smallskip
{\bf On the proof of Theorems~\ref{thmintro:general_aut} and~\ref{thmintro:cmp_aut}.}
The main difficulty to overcome is that classical JSJ theory is not well-suited to describing automorphisms of non-hyperbolic groups. Any finitely presented group has a JSJ decomposition over cyclic subgroups \cite{RS97}, and more generally over slender subgroups \cite{DS99,DS00,FP06,GL-JSJ}, but there are two key differences.

First, JSJ decompositions of hyperbolic groups are \emph{canonical}: they are graph-of-groups splittings preserved by all automorphisms of $G$ \cite{Sela-hypJSJ,Bow-JSJ}. Canonical JSJs are available for more general groups $G$ (the \emph{trees of cylinders} of \cite{GL-cyl}), but always under the assumption that abelian subgroups of $G$ do not interact with each other in complicated ways (the typical example is that of a relatively hyperbolic group). By contrast, many special groups (e.g.\ freely indecomposable RAAGs) are \emph{thick} in the sense of \cite{Behrstock-Drutu-Mosher}, which is a strong negation of this kind of property.

Secondly, cyclic splittings of hyperbolic groups encode all automorphisms: if $G$ is a torsion-free, $1$--ended hyperbolic group that does not split over $\Z$, then $\Out(G)$ is finite \cite{BF-stable}. This property again fails badly for special groups: there are RAAGs $A_{\G}$ that do not split over any abelian subgroups, but still have infinite $\Out(A_{\G})$ \cite[Figure~1]{Fio10a}.

One fundamental clue as to how to overcome these two issues comes from our previous work in \cite{Fio10e}: if $G$ is special and $\Out(G)$ is infinite, then $G$ splits over a \emph{centraliser} or a co-abelian subgroup thereof. Centralisers are not slender or small, and they have a complicated intersection pattern in general, but they also form a rather rigid collection of subgroups of $G$.

One of the goals of this article --- and its most important new tool --- is the construction of a \emph{canonical} JSJ decomposition over centralisers, for any special group $G$ (\Cref{thm:JSJ+}). To describe this, let $\mc{S}(G)$ denote the collection of maximal subgroups of $G$ that virtually split as direct products; we refer to these as the \emph{singular subgroups} of $G$ (\Cref{sub:singular}).

\begin{thmintro}\label{thmintro:JSJ}
    If $G$ is special and $1$--ended, then $G$ has an $\Aut(G)$--invariant graph-of-groups splitting (possibly a single vertex) such that:
    \begin{enumerate}
        \item each edge group is either a centraliser, or a cyclic subgroup;
        \item each subgroup in $\mc{S}(G)$ is conjugate into a vertex group;
        \item each vertex group is of one of the following two kinds:
            \begin{enumerate}
                \item[(a)] a quadratically hanging subgroup with trivial fibre;
                \item[(b)] a quasi-convex subgroup of $G$ that is elliptic in all splittings of $G$ over centralisers, relative to $\mc{S}(G)$.
            \end{enumerate}
    \end{enumerate}
\end{thmintro}

We think of type~(b) vertex groups $V\leq G$ as the ``rigid'' ones. Unlike for hyperbolic groups, the (virtual) restriction map $\Out(G)\ra\Out(V)$ can have infinite image even when $V$ is rigid. What is important is that any automorphism of a rigid group $V$ has its top growth rate realised on a special group of lower ``complexity'' (usually, a singular subgroup of $V$), which enables us to study growth by induction on complexity. The base step amounts to surface groups and free products, which can be dealt with using train tracks. Modulo technicalities, this is the core of the proof of \Cref{thmintro:general_aut}.

The enhanced JSJ decomposition in \Cref{thmintro:JSJ} is constructed in two steps. First, one finds an $\Out(G)$--invariant \emph{deformation space} of splittings over centralisers (in the sense of \cite{Forester,GL-JSJ}), using an accessibility result. Then, one finds an $\Out(G)$--invariant splitting in the deformation space, using a new construction of canonical splittings (\Cref{thm:invariant_splittings}); the latter is inspired by \cite{GL-cyl}, but can be applied under weaker hypotheses, allowing complicated intersection patterns of abelian subgroups. The accessibility result needed for the first step does not follow from classical forms of accessibility \cite{Dunwoody-acc,BF-complexity,Sela-acyl-acc}, as we are interested in non-acylindrical splittings over non-small subgroups. We instead prove the following:

\begin{thmintro}\label{thmintro:acc}
    Any special group $G$ admits an integer $N=N(G)$ such that any reduced graph-of-groups splitting of $G$ over centralisers has $\leq N$ edges.
\end{thmintro}

\Cref{thmintro:cmp_aut} and its complete description of growth rates for coarse-median preserving automorphisms require more work. The strategy is to consider the maximal subgroups of $G$ whose elements grow at below-top speed under $\phi\in\Out(G)$. If these are special and invariant under a power of $\phi$, then one can restrict $\phi$ to these subgroups, apply \Cref{thmintro:general_aut} and repeat. If this procedure eventually terminates, one obtains a description of all growth rates of $\phi$. A similar strategy was used in \cite{Sela-Duke} to construct analogues of Nielsen--Thurston decompositions for automorphisms of free groups.

Unfortunately, these maximal ``slow'' subgroups of $G$ are not even finitely generated for a general $\phi\in\Out(G)$ \cite[Example~3.4]{Fio11a}. This is where the coarse-median preserving hypothesis comes in: it guarantees that the maximal slow subgroups are quasi-convex in $G$. More precisely, given $G$ special and $\phi\in\Out(G)$ of infinite order, we can use the Bestvina--Paulin construction \cite{Bes88,Pau91} to produce $\R$--trees $G\acts T$ whose arc-stabilisers are (possibly infinitely generated) co-abelian subgroups of centralisers \cite{Fio10e}. When $\phi$ is coarse-median preserving, arc-stabilisers are actually genuine centralisers, which implies that the point-stabilisers of $T$ are quasi-convex (using \Cref{thmintro:acc} and \cite{Fio11b}). The maximal slow subgroups of $G$ are finite intersections of such point-stabilisers, and hence well-behaved.

Even if one is only interested in untwisted automorphisms of RAAGs, the proof of \Cref{thmintro:cmp_aut} forces us to immediately abandon this world: the maximal slow subgroups have no reason to be RAAGs in general (\Cref{ex:fix_not_raag}). One then has to understand automorphisms of more general special groups in order to prove \Cref{thmintro:cmp_aut} in this case. This should further motivate the level of generality pursued in this article: it serves the concrete purpose of understanding automorphisms of very explicit groups such as RAAGs.

Finally, we mention that versions of Rips--Sela theory for actions of special groups on higher-dimensional median spaces are available \cite{CRK1,CRK2}, though we managed to only rely on actions on $\R$--trees for this article.

\smallskip
{\bf Structure of paper.}
\Cref{sect:accessibility} is concerned with accessibility over centralisers; \Cref{thmintro:acc} is proved there as \Cref{thm:accessible}. \Cref{sect:JSJ} is devoted to the construction of the enhanced JSJ decomposition discussed in \Cref{thmintro:JSJ} (see \Cref{thm:JSJ+}). Then, \Cref{sect:exact} uses the JSJ decomposition to prove \Cref{thmintro:exact} (\Cref{cor:exact}). Finally, \Cref{sect:tameness} contains the proof of \Cref{thmintro:general_aut} (\Cref{thm:tame}), and \Cref{sect:cmp} that of \Cref{thmintro:cmp_aut} (\Cref{thm:cmp_main}).

\smallskip
{\bf Acknowledgements.} 
I am indebted to Ric Wade for suggesting that the proof of boundary amenability for $\Out(G)$ might also show finiteness of vcd, and to the referee for their careful reading and many insightful comments. I thank Adrien Abgrall, Martin Bridson, Daniel Groves, Vincent Guirardel, Camille Horbez, Ashot Minasyan, Sam Shepherd and Zlil Sela for useful suggestions, and the authors of \cite{CHHL} for sharing with me a preliminary version of their work after the present article was completed.


\section{Preliminaries}\label{sect:prelims}

\subsection{Special groups}\label{sub:special_prelims}

We say that a group $G$ is \emph{special}
if it is the fundamental group of a compact special cube complex. Equivalently, $G$ is a convex-cocompact subgroup of a RAAG $A_{\G}$, with respect to the action on the universal cover of the Salvetti complex $\X_{\G}$ \cite{HW08}; this means that $G$ leaves invariant a convex subcomplex of $\X_{\G}$ and acts cocompactly on it. 

\begin{rmk}\label{rmk:extending_to_virtually_special}
    Though Theorems~\ref{thmintro:general_aut} and~\ref{thmintro:exact} are stated for \emph{virtually} special groups, we will mostly work with actual special groups in the main body of the paper. One can reduce to this case by considering a finite-index, characteristic, special subgroup $G_0$ of any virtually special group $G$. There is then a uniform $k\geq 1$ such that each $g\in G$ has $g^k\in G_0$. However, some care is required when comparing the growth rates of $g$ and $g^k$, as we might have $|g|\gg |g^k|$ for elements of a general virtually special group. 
    The extension of Theorem~\ref{thmintro:general_aut} to \emph{virtually} special groups is taken care of in \Cref{lem:fi_fix}, and we discuss extensions of \Cref{thmintro:cmp_aut} in \Cref{rmk:fi_cmp_fix}.
\end{rmk}

Throughout the article, we fix a convex-cocompact embedding $\iota\colon G\hookrightarrow A_{\G}$. It is important to remember that many notions will not be intrinsic to $G$, and will rather depend on the choice of $\G$ and $\iota$. In particular, such notions will not be preserved by automorphisms of $G$ (and this will not matter). Our results always hold for \emph{all} choices of the convex-cocompact embedding $\iota$.

Occasionally, we will consider a \emph{coarse median structure} on $G$. This is an equivalence class of maps $\mu\colon G^3\ra G$ pairwise at bounded distance from each other with respect to a word metric (see \cite{Bow13,Fio10a}). Our coarse median structures will always be induced by the median operator on $\X_{\G}$, pulled back via the chosen convex-cocompact embedding $\iota\colon G\hookrightarrow A_{\G}$. Note that different choices of $\iota$ will induce different coarse median structures on $G$ (see \cite{FLS} for examples). An element of $\Aut(G)$ is \emph{coarse-median preserving} if it preserves the chosen coarse median structure on $G$; such automorphisms form a subgroup containing all inner automorphisms. In particular, it makes sense to speak of coarse-median preserving elements of $\Out(G)$.

An element $g\in G\setminus\{1\}$ is called \emph{contracting} if its centraliser $Z_G(g)$ is cyclic. With respect to any geometric $G$--action on a CAT(0) space, these are precisely the elements acting as rank--$1$ isometries \cite{Bestvina-Fujiwara,Charney-Sultan,Sisto-Z}.

We now discuss four classes of subgroups of any special group $G$ that play an important role in the paper. They relate to each other as follows:
\begin{center}
    \begin{tikzcd}
    \text{centraliser} \arrow[Rightarrow,rd,shorten >=.1ex, shorten <=.4ex] & & \\[-5ex]
    & \text{$G$--semi-parabolic} \arrow[Rightarrow,r] & \text{convex-cocompact} . \\[-6ex]
    \text{$G$--parabolic} \arrow[Rightarrow,ru] & &
    \end{tikzcd}
\end{center}
With the exception of centralisers, these classes of subgroups depend on the choice of the embedding $G\hookrightarrow A_{\G}$ and they are not $\Aut(G)$--invariant.

\subsubsection{Centralisers}\label{subsub:centralisers}
We write $Z_G(A):=\{g\in G\mid ga=ag,\ \forall a\in A\}$ for all subsets $A\sq G$.
A subgroup $H\leq G$ is a \emph{centraliser} if $H=Z_G(A)$ for some $A\sq G$; equivalently, we have $H=Z_GZ_G(H)$. We denote by $\mc{Z}(G)$ the family of all centralisers in $G$. Note that $\mc{Z}(G)$ is $\Aut(G)$--invariant and closed under intersections.

A subgroup $H\leq G$ is \emph{root-closed} if, whenever $g^n\in H$ for some $g\in G$ and $n\geq 2$, we actually have $g\in H$. Centralisers in special groups are always root-closed, as this is true in RAAGs \cite[Section~III]{Servatius}. Other properties of centralisers in special groups follow from the fact that they are convex-cocompact and $G$--semi-parabolic, two classes that we discuss below.

\subsubsection{$G$--parabolic subgroups}\label{subsub:G-parabolics}
A subgroup of a RAAG $A_{\G}$ is \emph{parabolic} if it is conjugate to $A_{\Delta}$ for some $\Delta\sq\G$. Consequently, if $G\leq A_{\G}$ is convex-cocompact, we say that a subgroup $H\leq G$ is \emph{$G$--parabolic} if $H=G\cap P$ for a parabolic subgroup $P\leq A_{\G}$. We emphasise that the trivial group $\{1\}$ and the whole group $G$ are $G$--parabolic subgroups.

With a notational abuse, we denote by $\mc{P}(G)$ the family of all $G$--parabolic subgroups of $G$, keeping in mind that this family also depends on our chosen embedding $G\hookrightarrow A_{\G}$. In particular, $\mc{P}(G)$ is not $\Aut(G)$--invariant. Nevertheless, $\mc{P}(G)$ is closed under intersections and all $G$--parabolic subgroups are root-closed. The main utility of $G$--parabolic subgroups comes from the following observation; see \cite[Corollary~3.21]{Fio10e} for a proof.

\begin{lem}\label{lem:parabolics_cofinite}
    The family $\mc{P}(G)$ is finite up to $G$--conjugacy.
\end{lem}

\subsubsection{Convex-cocompact subgroups}\label{subsub:cc}

A subgroup $H\leq G$ is \emph{convex-cocompact}, with respect to the chosen embedding $\iota\colon G\hookrightarrow A_{\G}$, if the universal cover of the Salvetti complex $\X_{\G}$ admits an $H$--invariant convex subcomplex on which $H$ acts cocompactly. Equivalently, $H$ acts cocompactly on its essential core within $\X_{\G}$ (in the sense of \cite{CS11}) or, again, $H$ is quasi-convex with respect to the coarse-median structure on $G$ induced by $\iota$; see \cite[Lemma~3.2]{Fio10a}.

The following results collect basic properties of convex-cocompactness.

\begin{lem}\label{lem:cc_basics}
    Let $G\leq A_{\G}$ and $H,K\leq G$ all be convex-cocompact.
    \begin{enumerate}
        \item The intersection $H\cap K$ is convex-cocompact.
        \item The normaliser $N_G(H)$ is convex-cocompact and virtually splits as $H\x P$ for some $P\in\mc{P}(G)$.
        \item If $g\in G$ satisfies $gHg^{-1}\leq H$, then $gHg^{-1}=H$.
        \item We have $\{g\in G \mid gHg^{-1}\leq K\}=K\cdot F\cdot N_G(H)$ with $F\sq G$ finite.
        \item The set $\{H\cap gKg^{-1}\mid g\in G\}$ is finite up to $H$--conjugacy.
        \item If $H\leq K$ and $K=K_1\x\dots\x K_m$ with all $K_i$ convex-cocompact, then the product $(H\cap K_1)\x\dots\x (H\cap K_m)$ has finite index in $H$.
    \end{enumerate}
\end{lem}
\begin{proof}
    Parts~(1),~(2) and~(4) are Lemmas~2.8,~3.22(1) and~2.9 in \cite{Fio10e}, respectively. Part~(5) is \cite[Lemma~4.2(3)]{Fio11b}. Part~(3) is immediate from the fact that $H$ is separable in $G$ \cite[Corollary~7.9]{HW08}, and part~(6) is clear if we think of convex-co\-com\-pact\-ness as coarse-median quasi-convexity.
\end{proof}

\begin{lem}\label{lem:parabolic_normaliser2}
Let $G\leq A_{\G}$ be convex-cocompact.
    \begin{enumerate}
        \item All subgroups in the families $\mc{Z}(G)$ and $\mc{P}(G)$ are convex-cocompact.
        \item For every $P\in\mc{P}(G)$, the normaliser $N_G(P)$ lies in $\mc{P}(G)$.
    \end{enumerate}
\end{lem}
\begin{proof}
    In view of \Cref{lem:cc_basics}(1), it suffices to prove part~(1) for $G=A_{\G}$. Convex-cocompactness is then immediate for parabolics, and a consequence of Servatius' Centralizer Theorem in \cite[Section~III]{Servatius} for centralisers.

    Regarding part~(2), consider an element $P\in\mc{P}(G)$. Let $\Delta\sq\G$ be minimal such that there exists $x\in A_{\G}$ with $P=G\cap xA_{\Delta}x^{-1}$. Up to conjugating the whole $G$ by $x$, we can assume that $P=G\cap A_{\Delta}$. Setting $\Delta^{\perp}:=\bigcap_{x\in\G}\lk(x)$, we have $N_{A_{\G}}(A_{\Delta})=A_{\Delta}\x A_{\Delta^{\perp}}$ by \cite[Proposition~3.13]{Antolin-Minasyan}. The subgroup $\mc{N}:=G\cap N_{A_{\G}}(A_{\Delta})$ is then $G$--parabolic and contained in $N_G(P)$.

    Now, \Cref{lem:cc_basics}(2) implies that $N_G(P)$ virtually splits as $P\x Q$ with $Q\in\mc{P}(G)$. Servatius' Centralizer theorem shows that $Q\leq G\cap A_{\Delta^{\perp}}$, 
    and so we have $P\x Q\leq\mc{N}$. In conclusion, $\mc{N}$ has finite index in $N_G(P)$ and, since $\mc{N}$ is root-closed, we must have $\mc{N}=N_G(P)$.
\end{proof}

Given an action $H\acts X$, we denote by $\Fix(H)$ the subset of fixed points.

\begin{lem}\label{lem:cc_edges}
    Let $G\leq A_{\G}$ be convex-cocompact. Let $G\acts T$ be a minimal action on a simplicial tree with convex-cocompact edge-stabilisers. 
    \begin{enumerate}
        \item Vertex-stabilisers of $T$ are convex-cocompact as well.
        \item If $H\leq G$ is convex-cocompact, then $N_G(H)\acts\Fix(H)$ is cocompact.
    \end{enumerate}
\end{lem}
\begin{proof}
    Part~(1) is \cite[Proposition~4.1]{Fio11b}. Regarding part~(2), since $G$ is finitely generated and acts minimally, the action $G\acts T$ is cocompact. Therefore, it suffices to show that, for each edge $e\sq T$, the normaliser $N_G(H)$ acts cofinitely on the set of edges of $\Fix(H)$ that lie in the $G$--orbit of $e$. Consider the set $\{g\in G\mid ge\sq\Fix(H)\}$. If $E$ is the stabiliser of $e$, this set coincides with $\{g\in G\mid g^{-1}Hg\leq E\}$. Since $H$ and $E$ are convex-cocompact, the latter set is a product $N_G(H)\cdot F\cdot E$ for a finite subset $F\sq G$, by \Cref{lem:cc_basics}(4). This means that $N_G(H)$ acts with $|F|$ orbits on the set of $H$--fixed edges in the $G$--orbit of $e$, as we wanted.
\end{proof}

The following is a related observation.

\begin{rmk}\label{rmk:root-closed_vertices}
    Let $H$ be a group and let $H\acts T$ be a minimal action on a tree. If edge-stabilisers are root-closed in $H$, then vertex-stabilisers are root-closed as well. Indeed, consider $g\in H$ and $n\geq 2$ such that $g^n$ fixes a vertex $x\in T$. It follows that $g$ is elliptic, so it fixes a vertex $y\in T$. If $y\neq x$, let $e\sq T$ be the edge incident to $x$ in the direction of $y$. Since $g^n$ fixes both $x$ and $y$, it must fix $e$. Now, the fact that $H_e$ is root-closed implies that $g\in H_e\leq H_x$, as required.
\end{rmk}

\subsubsection{$G$--semi-parabolic subgroups}\label{subsub:G-semi-parabolics}

Let $G\leq A_{\G}$ be convex-cocompact.

Centralisers are the most important class of subgroups of $G$ for this article: they will provide the edge groups for most of our splittings of $G$. A useful feature of centralisers is that they are close to being $G$--parabolic; their failure is only due to a (virtual) abelian factor. At the same time, the class $\mc{Z}(G)$ has a considerable weakness: if $H\leq G$ is convex-cocompact and $Z\in\mc{Z}(G)$, the intersection $H\cap Z$ will not normally lie in $\mc{Z}(H)$. 

To circumvent this, it is often useful to work with a larger class of subgroups, which have the same structure as centralisers and additionally satisfy the above stability property. We say that a subgroup is \emph{co-abelian} if it is normal with abelian quotient.

\begin{defn}\label{defn:semi-parabolic}
    A subgroup $Q\leq G$ is \emph{$G$--semi-parabolic} if it is convex-cocompact, root-closed, and has a co-abelian $G$--parabolic subgroup $R\lhd Q$.
\end{defn}

All centralisers are $G$--semi-parabolic: for centralisers of single elements, this follows from Servatius' Centralizer Theorem; it then suffices to note that intersections of semi-parabolic subgroups are again semi-parabolic.

The co-abelian $G$--parabolic subgroup $R\lhd Q$ in \Cref{defn:semi-parabolic} is not unique in general, but it does have the following useful property.

\begin{lem}\label{lem:weak_parabolic part}
    Let $Q\leq G$ be $G$--semi-parabolic, and let $R\lhd Q$ be co-abelian and $G$--parabolic. If we have $gRg^{-1}\leq Q$ for some $g\in G$, then $gRg^{-1}=R$.
\end{lem}
\begin{proof}
    As in the proof of \Cref{lem:parabolic_normaliser2}(2), we can assume that $R=G\cap A_{\Delta}$ for some $\Delta\sq\G$ such that we have $N_G(R)\leq N_{A_{\G}}(A_{\Delta})=A_{\Delta}\x A_{\Delta^{\perp}}$. Since we also have $Q\leq N_G(R)$, \Cref{lem:cc_basics}(6) shows that $Q$ virtually splits as $R\x A$, where $R=Q\cap A_{\Delta}$ and we set $A:=Q\cap A_{\Delta^{\perp}}$. 

    Now, consider an element $g\in G$ with $gRg^{-1}\leq Q$. The intersection $\mc{I}:=gRg^{-1}\cap (R\x A)$ must be entirely contained in $R$ (this is clear considering the cyclically reduced form of elements of $A_{\G}$). At the same time, since $\mc{I}$ has finite index in in $gRg^{-1}$ and since $R$ is root-closed, we obtain $gRg^{-1}\leq R$. \Cref{lem:cc_basics}(3) finally yields $gRg^{-1}=R$ as desired.
\end{proof}

\begin{rmk}\label{rmk:semi-parabolic_chains}
    There is a uniform bound on the length of chains of $G$--semi-parabolic subgroups $Q_1\lneq\dots\lneq Q_n$, which depends only on $G$ and its embedding in a RAAG. For instance, this is a particular case of \cite[Lemma~3.36]{Fio10e}. Using \Cref{lem:cc_basics}, the definition of ``$G$--semi-parabolic'' used in \cite{Fio10e} is easily seen to be equivalent to the above one.
\end{rmk}

\subsection{Terminology on $\R$--trees and splittings}\label{sub:arboreal_terminology}

Given a finitely generated group $G$ and a non-elliptic action on an $\R$--tree $G\acts T$, we denote by $\Min(G,T)$ the $G$--minimal subtree of $T$, i.e.\ the smallest invariant subtree. If $U\sq T$ is a subset, we write $G_U:=\{g\in G\mid gU=U\}$ for its $G$--stabiliser. An \emph{arc} is a subset homeomorphic to $[0,1]$. If $\beta\sq T$ is an arc, then $G_{\beta}$ coincides with the \emph{pointwise} stabiliser of $\beta$ (in all our cases of interest, $G_{\beta}$ is root-closed in $G$ and so there are no arc-inversions). On the other hand, if $\alpha$ is a line, then $G_{\alpha}$ might translate along it. A subtree $\tau\sq T$ is \emph{stable} if all its arcs have the same stabiliser. The action $G\acts T$ is \emph{BF--stable} (after \cite{BF-stable} and \cite{Guir-Fourier}) if every arc contains a stable sub-arc.

Throughout the article, we implicitly assume that actions on simplicial trees are without inversions. A \emph{splitting} of a group $G$ is a minimal $G$--action on a simplicial tree $T$ such that $T$ has at least one edge. Following \cite{GL-JSJ}, we speak of an \emph{$(\mc{A},\mc{H})$--splitting}, where $\mc{A}$ and $\mc{H}$ are families of subgroups of $G$, if the $G$--stabiliser of each edge of $T$ lies in $\mc{A}$ and if each subgroup in $\mc{H}$ fixes a point of $T$. We will also say that $T$ is a tree \emph{over} $\mc{A}$ and \emph{relative} to $\mc{H}$, with the same meaning. Importantly, we are interested in families $\mc{A}$ that are {\bf not} closed under taking subgroups, for instance $\mc{A}=\mc{Z}(G)$. This is an important difference to \cite{GL-JSJ} and standard JSJ theory, where it is often an implicit standing assumption. When we wish to allow the possibility that $T$ is a single vertex, or that $G$ acts non-minimally, then we simply speak of an \emph{$(\mc{A},\mc{H})$--tree}, still expecting $T$ to be simplicial.

A subgroup $H\leq G$ is \emph{$(\mc{A},\mc{H})$--rigid} in $G$ if it is elliptic in all $(\mc{A},\mc{H})$--splittings of $G$. A torsion-free group $G$ is \emph{$1$--ended relative} to $\mc{H}$ if $G$ is $(\{\{1\}\},\mc{H})$--rigid in itself; when $\mc{H}=\emptyset$, this is the usual notion of $1$--endedness \cite{Stallings}. A vertex group $Q$ of a $G$--tree $T$ is \emph{quadratically hanging (QH)} relative to $\mc{H}$ if we can represent $Q$ as the fundamental group of a compact hyperbolic surface $\Sigma$ such that all incident edge groups of $T$, as well as all subgroups of $Q$ contained in elements of $\mc{H}$, are \emph{peripheral} (i.e.\ conjugate into the fundamental groups of the components of $\partial\Sigma$).

A $G$--tree $T$ is \emph{invariant} under a subgroup $\mc{O}\leq\Out(G)$ if, for each automorphism $\varphi\in\Aut(G)$ with outer class in $\mc{O}$, there exists $\Phi\in\Aut(T)$ such that $\Phi(gx)=\varphi(g)\Phi(x)$ for all $g\in G$ and $x\in T$. When $G$ has trivial centre, this simply means that the action $G\acts T$ extends to an action $\tilde{\mc{O}}\acts T$, for the preimage $\tilde{\mc{O}}\leq\Aut(G)$ of the group $\mc{O}\leq\Out(G)$.

\subsection{Growth rates}\label{sub:growth}

When we speak of a ``growth rate'' in this article, we refer to the following general notion of how fast a sequence can grow.

\subsubsection{Abstract growth rates}
If $a,b\colon\N\ra\R_{>0}$ are two sequences, we write $a\preceq b$ if there exists a constant $C$ such that $a_n\leq Cb_n$ for all $n\geq 0$. We say that $a$ and $b$ are \emph{equivalent}, written $a\sim b$, if we have both $a\preceq b$ and $b\preceq a$. We denote by $[1]$ the equivalence class of bounded sequences.

\begin{defn}
    A \emph{growth rate} is a $\sim$--equivalence class $[a_n]$ of sequences in $(\R_{>0})^{\N}$ with $[a_n]\succeq [1]$. We denote by $(\mf{G},\preceq)$ the set of all growth rates with the poset structure induced by the relation $\preceq$.
\end{defn}

It will often be convenient to work modulo a non-principal ultrafilter $\om\sq 2^{\N}$ and the weaker ordering on sequences that it induces. We refer to \cite[Chapter~10]{DK} for generalities on ultrafilters. Given sequences $a,b\colon\N\ra\R_{>0}$, we write $a\preceq_{\om} b$ if we have $\lim_{\om} a_n/b_n<+\infty$. We say that $a$ and $b$ are \emph{$\om$--equivalent}, written $a\sim_{\om} b$, if we have both $a\preceq_{\om} b$ and $b\preceq_{\om} a$. 

\begin{defn}
    An \emph{$\om$--growth rate} is a $\sim_{\om}$--equivalence class $[a_n]$ of sequences in $(\R_{>0})^{\N}$ with $[a_n]\succeq_{\om} [1]$. We denote by $(\mf{G}_{\om},\preceq_{\om})$ the set of all $\om$--growth rates with the total order induced by the relation $\preceq_{\om}$.
\end{defn}

Each $\sim_{\om}$--equivalence class of sequences is a union of (uncountably many) $\sim$--equivalence classes. Thus, there is a natural order-preserving quotient map $(\mf{G},\preceq)\ra (\mf{G}_{\om},\preceq_{\om})$. We emphasise that, while $(\mf{G},\preceq)$ is only partially ordered, all its quotients $(\mf{G}_{\om},\preceq_{\om})$ are totally ordered.

\begin{rmk}\label{rmk:fail->fail_mod_omega}
    We have $[a_n]\preceq [b_n]$ if and only if the inequality $[a_n]\preceq_{\om}[b_n]$ holds for all non-principal ultrafilters $\om$.
\end{rmk}

We will often simply write $a_n\preceq b_n$, $a_n\sim b_n$, $a_n\preceq_{\om} b_n$, omitting square brackets when this streamlines notation without causing ambiguities. 

\subsubsection{Lengths on groups}\label{subsub:lengths_on_groups}
Let $G$ be a group with a finite generating set $S$. We denote by $|\cdot|_S$ and $\|\cdot\|_S$ the \emph{word} and \emph{conjugacy length} on $G$ associated with $S$, as defined in the Introduction. If $T$ is a different finite generating set, then $|\cdot|_T$ and $\|\cdot\|_T$ are bi-Lipschitz equivalent to $|\cdot|_S$ and $\|\cdot\|_S$.

If $G\acts X$ is an isometric action on a metric space, we can consider its \emph{displacement parameter}, and the \emph{translation length} of an element:
\begin{align*}
    \tau_X^S&:= \inf_{x\in X} \max_{s\in S} d(sx,x) , & \ell_X(g)&:=\inf_{x\in X} d(x,gx) .
\end{align*}
If $T$ is another finite generating set, there is a constant $C=C(S,T)$ such that $\frac{1}{C}\tau_X^S\leq \tau_X^T\leq C\tau_X^S$ for all $G$--spaces $X$. The following is straightforward. 

\begin{lem}\label{lem:length_generalities}
    Consider $\varphi\in\Aut(G)$ and an isometric action $G\acts X$.
    \begin{enumerate}
        \item The map $\varphi\colon G\ra G$ is bi-Lipschitz with respect to both $|\cdot|_S$ and $\|\cdot\|_S$.
        \item For every $g\in G$, we have $\ell_X(g)\leq \tau_X^S\|g\|_S$.
        \item If $X$ is geodesic and $G\acts X$ is free, proper and cocompact, then there is a constant $c=c(S,X)>0$ such that $\ell_X(g)\geq c\|g\|_S$ for all $g\in G$. 
    \end{enumerate}
\end{lem}

\begin{rmk}\label{rmk:conjlength_vs_undistortion}
    Let $H\leq G$ be generated by a finite subset $U\sq H$. If $H$ is undistorted, then the word lengths $|\cdot|_U$ and $|\cdot|_S$ are bi-Lipschitz equivalent on $H$. However, undistortion alone does not suffice to conclude that the conjugacy lengths $\|\cdot\|_U$ and $\|\cdot\|_S$ are bi-Lipschitz equivalent on $H$.

    Still, if $H\leq G\leq A_{\G}$ are convex-cocompact, then $\|\cdot\|_U$ and $\|\cdot\|_S$ are bi-Lipschitz equivalent. Indeed, there are convex subcomplexes $C_H\sq C_G\sq\X_{\G}$ that are invariant and cocompact, respectively, for the $H$-- and $G$--action. \Cref{lem:length_generalities} shows that $\|\cdot\|_U$ and $\|\cdot\|_S$ are bi-Lipschitz to $\ell(\cdot,C_H)$ and $\ell(\cdot,C_G)$, respectively. Finally, we have $\ell(h,C_H)=\ell(h,C_G)$ for $h\in H$, since the nearest-point projection $C_G\ra C_H$ is $1$--Lipschitz and $H$--equivariant.
\end{rmk}

\subsubsection{Growth of automorphisms}\label{subsub:growth_of_automorphisms}
Consider now a finitely generated group $G$, an automorphism $\varphi\in\Aut(G)$ and its outer class $\phi\in\Out(G)$. Let $|\cdot|$ and $\|\cdot\|$ be the word and conjugacy lengths on $G$ with respect to some finite generating set, whose choice will play no role.

\begin{defn}
    The \emph{growth rate} of an element $g\in G\setminus\{1\}$ under $\varphi$ is the $\sim$--equivalence class of the sequence $n\mapsto|\varphi^n(g)|$ in $\mf{G}$. Similarly, the \emph{growth rate} of $g$ under $\phi$ is\footnote{Here $\phi^n(g)$ is not a well-defined element of $G$, but it is a well-defined conjugacy class.} the equivalence class of $n\mapsto\|\phi^n(g)\|\in\mf{G}$. 
\end{defn}

A different choice of generating set for $G$ only alters $|\cdot|$ and $\|\cdot\|$ through a bi-Lipschitz equivalence, so it leads to the exact same growth rates within $\mf{G}$. Since $\|\cdot\|\leq |\cdot|$, we always have $\|\phi^n(g)\| \preceq |\varphi^n(g)|$. The equivalence $|\varphi^n(g)|\sim [1]$ holds if and only if a power of $\varphi$ fixes $g$, and $|\phi^n(g)\|\sim[1]$ holds if and only if a power of $\phi$ preserves the conjugacy class of $g$. Thus, we artificially define the growth rate of the identity of $G$ to be $[1]$.

We denote by $\mc{G}(\varphi)\sq\mf{G}$ and $\mf{g}(\phi)\sq\mf{G}$ (or $\mc{G}(G,\varphi)$ and $\mf{g}(G,\phi)$ in case of ambiguity) the sets of all growth rates of the automorphism, as $g$ varies in $G$. There are also versions of the above concepts modulo any non-principal ultrafilter $\om$. We thus denote by $\mc{G}_{\om}(\varphi)$ and $\mf{g}_{\om}(\phi)$ the sets of $\om$--growth rates of $\varphi$ and $\phi$, that is, the images of $\mc{G}(\varphi)$ and $\mf{g}(\phi)$ under the projection $\mf{G}\ra\mf{G}_{\om}$. In keeping with the above conventions, we will generally denote ($\om$--)growth rates of automorphisms by the letter $\mc{O}$, and ($\om$--)growth rates of outer automorphisms by the letter $\mf{o}$.

There are two additional important elements of $\mf{G}$ that we can associate with $\varphi$ and $\phi$. In some sense, they play the role of a ``maximum'' for the sets $\mc{G}(\varphi)$ and $\mf{g}(\phi)$, but it is important to stress that, a priori, they do {\bf not} lie in $\mc{G}(\varphi)$ or $\mf{g}(\phi)$. Fixing any finite generating set $S\sq G$, we write:
\begin{align*}
    \overline{\mc{O}}_{\rm top}(\varphi)&:=\big[\s^S(\varphi^n)\big],  \qquad\text{where} \quad \s^S(\varphi):=\max_{s\in S}|\varphi(s)| , \\
    \overline{\mf{o}}_{\rm top}(\phi)&:=\big[\tau^S(\phi^n)\big],  \qquad\text{where} \quad \tau^S(\phi):=\min_{x\in G}\max_{s\in S}|x\varphi(s)x^{-1}| .
\end{align*}
Note that $\overline{\mc{O}}_{\rm top}(\varphi)$ and $\overline{\mf{o}}_{\rm top}(\phi)$ are again completely independent of the choice of $S$. The notation $\tau^S(\phi)$ is consistent with the one introduced in \Cref{subsub:lengths_on_groups}; it is the displacement parameter for the action of $G$ on the Cayley graph used to define the word length $|\cdot|$, pre-composed with $\varphi$. 

\begin{lem}\label{lem:o_bar_is_faster}
    We have $\mc{O}\preceq\overline{\mc{O}}_{\rm top}(\varphi)$ for all $\mc{O}\in\mc{G}(\varphi)$. Similarly, we have $\mf{o}\preceq\overline{\mf{o}}_{\rm top}(\phi)$ for all $\mf{o}\in\mf{g}(\phi)$. 
\end{lem}
\begin{proof}
    For all $g\in G$ and $n\in\N$, we have $|\varphi^n(g)|\leq |g|_S\cdot\s^S(\varphi^n)$ and $\|\phi^n(g)\|\leq \|g\|_S\cdot\tau^S(\phi^n)$. The lemma immediately follows.
\end{proof}

Despite the previous lemma, it is not at all clear whether the sets $\mc{G}(\varphi)$ and $\mf{g}(\phi)$ always have a $\preceq$--maximum, even for automorphisms of special groups. I still do not know if this is the case in the setting of \Cref{thmintro:general_aut}. This is true under the assumptions of \Cref{thmintro:cmp_aut}, but even in that situation it will only become clear near the end of the article (\Cref{thm:cmp_main}(1)).

What is instead clear, even at this preliminary stage, is that automorphisms of special groups always have a top $\om$--growth rate, for any non-principal ultrafilter $\om$.
Denote by $\mc{O}_{\rm top}^{\om}(\varphi)$ and $\mf{o}_{\rm top}^{\om}(\phi)$ the projections of $\overline{\mc{O}}_{\rm top}(\varphi)$ and $\overline{\mf{o}}_{\rm top}(\phi)$ to $\mf{G}_{\om}$. These do lie in $\mc{G}_{\om}(\varphi)$ and $\mf{g}_{\om}(\phi)$:

\begin{lem}\label{lem:top_exists}
    Let $G,\varphi,\phi$ be as above, and let $\om$ be a non-principal ultrafilter.
    \begin{enumerate}
        \item We always have $\mc{O}^{\om}_{\rm top}(\varphi)=\max\mc{G}_{\om}(\varphi)$.
        \item If $G$ is special, then $\mf{o}_{\rm top}^{\om}(\phi)=\max\mf{g}_{\om}(\phi)$.
    \end{enumerate}
\end{lem}
\begin{proof}
By \Cref{lem:o_bar_is_faster}, the $\om$--growth rates $\mc{O}_{\rm top}^{\om}(\varphi)$ and $\mf{o}_{\rm top}^{\om}(\phi)$ are $\om$--faster than any $\om$--growth rate in the sets $\mc{G}_{\om}(\varphi)$ and $\mf{g}_{\om}(\phi)$, respectively. Thus, we only need to find two elements $g,g'\in G$ with $|\varphi^n(g)|\sim_{\om}\mc{O}^{\om}_{\rm top}(\varphi)$ and $\|\phi^n(g')\|\sim_{\om}\mf{o}^{\om}_{\rm top}(\phi)$. Since $S$ is finite, there is an element $s_0\in S$ such that the equality $\s^S(\varphi^n)=|\varphi^n(s_0)|$ holds for $\om$--all $n$; we can then take $g:=s_0$. As to the existence of $g'$, it requires the Bestvina--Paulin construction and some weak properties of asymptotic cones of special groups; we explain this a couple of pages below, in \Cref{prop:degeneration_basics}.
\end{proof}

\begin{rmk}\label{rmk:dominating_otop}
    Let $G$ be special. Given some $\mf{o}\in\mf{G}$, we have $\mf{o}\succeq\mf{o}'$ for all $\mf{o}'\in\mf{g}(\phi)$ if and only if $\mf{o}\succeq\overline{\mf{o}}_{\rm top}(\phi)$. Similarly, we have $\mf{o}\succeq\mc{O}$ for all $\mc{O}\in\mc{G}(\varphi)$ if and only if $\mf{o}\succeq\overline{\mc{O}}_{\rm top}(\varphi)$. In view of \Cref{rmk:fail->fail_mod_omega}, both observations follow from the combination of \Cref{lem:o_bar_is_faster} and \Cref{lem:top_exists}.
\end{rmk}

\subsubsection{Operations on growth rates}\label{subsub:operations_growth_rates}
Given two growth rates $[a_n],[b_n]\in\mf{G}$, the sum $[a_n]+[b_n]:=[a_n+b_n]$ is a well-defined growth rate. The sum of two $\om$--growth rates is also a well-defined $\om$--growth rate. For $k\in\N$, we define $k\ast[a_n]:=[a_{kn}]$, that is, the equivalence class of the sequence $n\mapsto a_{kn}$. Similarly, we set $\frac{1}{k}\ast[a_n]:=[a_{\lfloor\frac{n}{k}\rfloor}]$. The operation $\ast$ is only well-defined on $\mf{G}$; it does not descend to an operation on $\mf{G}_{\om}$.

For a finitely generated group $G$ and $\psi\in\Out(G)$, we have the identities $\overline{\mf{o}}_{\rm top}(\psi^k)\sim k\ast\overline{\mf{o}}_{\rm top}(\psi)$ and $\overline{\mf{o}}_{\rm top}(\psi)\sim\tfrac{1}{k}\ast\overline{\mf{o}}_{\rm top}(\psi^k)$ for each $k\geq 1$. (The second formula uses \Cref{lem:length_generalities}(1).) Analogous identities hold for elements of $\Aut(G)$ and $\overline{\mc{O}}_{\rm top}(\cdot)$.

\subsubsection{Beat and controlled subgroups}\label{sub:beat}
Given two growth rates $[a_n],[b_n]\in\mf{G}$, we write $[a_n]\prec[b_n]$ if we have $[a_n]\preceq[b_n]$ and $[a_n]\not\sim[b_n]$. Equivalently, we have $\limsup_n a_n/b_n<+\infty$ and $\liminf_n a_n/b_n=0$. 

The following terminology will be useful in Sections~\ref{sect:tameness} and~\ref{sect:cmp}. Let $G$ be finitely generated and consider $\phi\in\Out(G)$. Given a growth rate $\mf{o}\in\mf{G}$, we say that an element $g\in G$ is \emph{$\mf{o}$--beat} (for $\phi$) if we have the strict inequality $\|\phi^n(g)\|\prec\mf{o}$. We say that $g$ is \emph{$\mf{o}$--controlled} if we have the weak inequality $\|\phi^n(g)\|\preceq\mf{o}$. A subgroup $H\leq G$ is \emph{$\mf{o}$--beat} or \emph{$\mf{o}$--controlled} if all elements of $H$ have the respective property; for this we always compute conjugacy lengths within $G$. We denote by $\mc{B}(\mf{o},\phi)$ and $\mc{K}(\mf{o},\phi)$ (or simply $\mc{B}(\mf{o})$ and $\mc{K}(\mf{o})$) the families of $\mf{o}$--beat and $\mf{o}$--controlled subgroups of $G$.

Given a non-principal ultrafilter $\om$, we can also consider the elements $g\in G$ that are $\mf{o}$--beat or $\mf{o}$--controlled \emph{modulo $\om$}, that is, those satisfying $\|\phi^n(g)\|\prec_{\om}\mf{o}$ or $\|\phi^n(g)\|\preceq_{\om}\mf{o}$. We consequently define the families $\mc{B}^{\om}(\mf{o})$ and $\mc{K}^{\om}(\mf{o})$ of subgroups that are $\mf{o}$--beat or $\mf{o}$--controlled modulo $\om$.

\begin{rmk}\label{rmk:beat_invariant}
    For any $\phi\in\Out(G)$ and $\mf{o}\in\mf{G}$, the four families $\mc{B}(\mf{o},\phi)$, $\mc{K}(\mf{o},\phi)$, $\mc{B}^{\om}(\mf{o},\phi)$ and $\mc{K}^{\om}(\mf{o},\phi)$ are closed under conjugacy and $\phi$--invariant. This is because $\phi$ is bi-Lipschitz with respect to $\|\cdot\|$ (\Cref{lem:length_generalities}).
\end{rmk}

\subsection{Degenerations}\label{sub:degenerations}

Let $G$ be a special group. An infinite sequence in $\Out(G)$ and a choice of a non-principal ultrafilter $\om$ give rise to an isometric action on a median space $G\acts\X_{\om}$, which equivariantly embeds in a finite product of $\R$--trees $T^v_{\om}$. We refer to such actions as \emph{degenerations}. 

We briefly recall this construction and its properties shown in \cite{Fio10e,Fio10a}. Choose a convex-cocompact embedding $\iota\colon G\hookrightarrow A_{\G}$ into a RAAG. For each $v\in\G$, let $A_{\G}\acts T^v$ be the Bass--Serre tree of the HNN splitting of $A_{\G}$ with vertex group $A_{\G\setminus\{v\}}$, edge group $A_{\lk(v)}$ and stable letter $v$. The universal cover $\X_{\G}$ of the Salvetti complex embeds $A_{\G}$--equivariantly and isometrically into the finite product $\prod_{v\in\G}T^v$, using $\ell_1$--metrics.

For any infinite-order $\phi\in\Out(G)$ and any non-principal ultrafilter, the Bestvina--Paulin construction yields an isometric action $G\acts\X_{\om}$: this is the $\om$--limit of countably many copies of the actions $G\acts\X_{\G}$, with the $n$--th one twisted by $\phi^n$, based at a suitable point, and rescaled by the displacement parameter $\tau_{\X_{\G}}^{\phi^n(S)}$, for a fixed finite generating set $S\sq G$. We similarly obtain isometric actions $G\acts T^v_{\om}$ on $\R$--trees, which are the $\om$--limit of copies of $T^v$ twisted by $\phi^n$ and rescaled by the same factors. Naturally, there is a $G$--equivariant isometric embedding $\X_{\om}\hookrightarrow\prod_{v\in\G}T^v_{\om}$. The actions $G\acts T^v_{\om}$ can be elliptic for \emph{some} vertices $v\in\G$, but not for \emph{all} of them.

The following explains why degenerations are useful in studying growth rates, and it completes the proof of \Cref{lem:top_exists}(2). By \Cref{rmk:conjlength_vs_undistortion}, we can equivalently compute conjugacy lengths within $A_{\G}$ or $G$.

\begin{prop}\label{prop:degeneration_basics}
    Let $G\leq A_{\G}$ be convex-cocompact, let $\om$ be a non-principal ultrafilter, and let $\phi\in\Out(G)$ have infinite order.
    \begin{enumerate}
        \item The degeneration $G\acts\X_{\om}$ is an isometric action on a finite-rank median space. There are elements $g\in G$ with $\ell_{\X_{\om}}(g)>0$.
        \item For $g\in G$, we have $\ell_{\X_{\om}}(g)>0$ if and only if $\|\phi^n(g)\|\sim_{\om}\mf{o}^{\om}_{\rm top}(\phi)$. Similarly, $g$ fixes a point of $\X_{\om}$ if and only if $\|\phi^n(g)\|\prec_{\om}\mf{o}^{\om}_{\rm top}$.
    \end{enumerate}
\end{prop}
\begin{proof}
    The fact that $X$ is finite-rank median is shown in \cite[Section~9]{Bow13}. The action $G\acts\X_{\om}$ is non-elliptic by construction, and so \cite[Proposition~3.2]{Fio2} guarantees that it has loxodromics, proving part~(1). 
    
    Letting $\tau_n$ be the scaling factors used in the construction of $\X_{\om}$, the Milnor--Schwarz lemma and \Cref{rmk:conjlength_vs_undistortion} yield $[\tau_n]\sim\overline{\mf{o}}_{\rm top}(\phi)$. By \cite[Lemma~7.9(2)]{Fio10a}, we have $\ell_{\X_{\om}}(g)=\lim_{\om}\frac{1}{\tau_n}\ell_{\X_{\G}}(\phi^n(g))$ for all $g\in G$. Recalling that the function $\ell_{\X_{\G}}(\cdot)$ is bi-Lipschitz equivalent to $\|\cdot\|$ on $G$, we obtain $\|\phi^n(g)\|\sim_{\om}\tau_n$ if $\ell_{\X_{\om}}(g)>0$, and $\|\phi^n(g)\|\prec_{\om}\tau_n$ if $\ell_{\X_{\om}}(g)=0$. Finally, the fact that $[\tau_n]\sim\overline{\mf{o}}_{\rm top}(\phi)$ implies that $[\tau_n]\sim_{\om}\mf{o}^{\om}_{\rm top}(\phi)$, yielding part~(2).
\end{proof}

We generally speak of ``the'' degeneration determined by $\phi$ and $\om$ although, a priori, this also depends on the specific choice of basepoints in $\X_{\G}$ used to take an $\om$--limit. This is no cause for concern.

We now discuss the $\R$--trees $G\acts T^v_{\om}$ in more detail. What matters is that arc-stabilisers are well-behaved, and so Rips--Sela theory can be applied.

\begin{thm}\label{thm:10e-}
    Let $G\leq A_{\G}$ be convex-cocompact, $\phi\in\Out(G)$ infinite-order and $\om$ a non-principal ultrafilter. Let $G\acts\X_{\om}$ be the resulting degeneration. Consider $v\in\G$ such that $G$ is non-elliptic in $T^v_{\om}$.
    \begin{enumerate}
    \setlength\itemsep{.2em}
        \item For every arc $\beta\sq\Min(G,T^v_{\om})$, there exists $Z\in\mc{Z}(G)$ with $G_{\beta}\lhd Z$ and $Z/G_{\beta}$ free abelian. In particular, $G\acts\Min(G,T^v_{\om})$ is BF--stable.
        \item Suppose that {\bf at least one} of the following holds:
            \begin{enumerate}
                \item $\phi$ preserves the coarse median structure induced on $G$ by $A_{\G}$;
                \item all non-cyclic elements of $\mc{Z}(G)$ are elliptic in $T^v_{\om}$;
            \end{enumerate}
            Then, for every arc $\beta\sq\Min(G,T^v_{\om})$, we have $G_{\beta}\in\mc{Z}(G)$.
    \end{enumerate}
\end{thm}
\begin{proof}
    Part~(2b) and the first half of part~(1) follow from \cite[Proposition~5.12]{Fio10e}. BF--stability then holds because chains of arc-stabilisers have uniformly bounded length, see \cite[Corollary~6.13]{Fio10e}. Finally, part~(2a) is shown in \cite[Proposition~5.15(c1)]{Fio10e}. 
\end{proof}

The following consequence does not require finite generation of $H\leq G$.

\begin{lem}\label{lem:beat_vs_degeneration}
    Let $G\leq A_{\G}$ be convex-cocompact, with conjugacy length function denoted $\|\cdot\|_G$. 
    Let $G\acts\X_{\om}$ be the degeneration determined by some $\phi\in\Out(G)$ and an ultrafilter $\om$. Then, a subgroup $H\leq G$ fixes a point of $\X_{\om}$ if and only if $\|\phi^n(h)\|_G\prec_{\om}\mf{o}^{\om}_{\rm top}(\phi)$ for all $h\in H$.
\end{lem}
\begin{proof}
    In view of \Cref{prop:degeneration_basics}(2), we only need to show that, if all elements of $H$ are elliptic in $\X_{\om}$, then $H$ is elliptic. For each $v\in\G$, all finitely generated subgroups of $H$ are elliptic in $T^v_{\om}$ by Serre's lemma \cite[p.~64]{Serre}. Since chains of arc-stabilisers of $T^v_{\om}$ have bounded length by \Cref{thm:10e-}(1), this implies that $H$ is itself elliptic in all $T^v_{\om}$, e.g.\ using \cite[Lemma~2.18]{Fioravanti-Kerr}. From this one deduces that $H$ fixes a point of $\X_{\om}$, as required.
\end{proof}

\section{Accessibility over centralisers}\label{sect:accessibility}

In this section we prove \Cref{thmintro:acc}: special groups are accessible over centralisers, and more generally over semi-parabolic subgroups. 

We explain our terminology. A splitting $G\acts T$ is \emph{reduced} \cite{BF-complexity} if there does not exist a vertex $v\in T$ whose $G$--stabiliser fixes an incident edge and acts transitively on the remaining incident edges. More weakly, we say that $G\acts T$ is \emph{irredundant} if $T$ has no degree--$2$ vertices, except for those where the vertex-stabiliser swaps the two incident edges; in other words, $T$ was not obtained from a smaller splitting of $G$ simply by subdividing an edge. Reduced splittings are irredundant, but the converse does not hold.

\begin{defn}\label{defn:accessible}
A group $G$ is \emph{accessible} over a family of subgroups $\mc{F}$ if there exists a number $N(G)$ such that any reduced splitting of $G$ over $\mc{F}$ has at most $N(G)$ orbits of edges. We say that $G$ is \emph{unconditionally accessible} over $\mc{F}$ if the same is true of irredundant splittings.
\end{defn}

For instance, finitely presented groups are accessible over small subgroups \cite{BF-complexity}, but already the free group $F_2$ fails to be unconditionally accessible over cyclic subgroups \cite[p.\,450]{BF-complexity}. At the end of this section, we prove:

\begin{thm}\label{thm:accessible}
    Every special group $G$ is (unconditionally) accessible over the family of $G$--semi-parabolic subgroups.
\end{thm}

Here accessibility holds unconditionally because semi-parabolics are more rigid than small subgroups: ascending chains have bounded length (\Cref{rmk:semi-parabolic_chains}). Still, the core of \Cref{thm:accessible} lies in accessibility in the usual sense; unconditionality will just make our life easier later on.

Here is a proof sketch of \Cref{thm:accessible}. Every semi-parabolic subgroup $Q\leq G$ admits a $G$--parabolic co-abelian subgroup $P\lhd Q$. One thus hopes to leverage the fact that the quotiented normalisers $N_G(P)/P$ are accessible over small subgroups, together with the fact that there are only finitely many $G$--conjugacy classes of $G$--parabolic subgroups of $G$. This strategy works provided that, given a reduced splitting $G\acts T$, the induced action $N_G(P)\acts\Fix(P;T)$ is sufficiently close to being minimal and reduced. Showing this requires control over the normal closure $\langle\langle P\rangle\rangle_G$ and, here, it turns out that the pro--$p$ topologies on $G$ can be rather useful.  

Let $G$ be a group and $p$ a prime number. A subgroup $H\leq G$ is \emph{$p$--separable} if, for every element $g\in G\setminus H$, there exists a homomorphism $f\colon G\ra F$ such that $F$ is a finite $p$--group and $f(g)\not\in f(H)$. The group $G$ is \emph{residually $p$--finite} if the trivial subgroup is $p$--separable. Our interest in these properties is due to the two facts collected in the following lemma.

\begin{lem}\label{lem:p_separability}
    Let $G$ be a group and let $p$ be any prime.
    \begin{enumerate}
        \item If a proper subgroup $H<G$ is $p$--separable, then $H$ is contained in a normal subgroup of $G$ of index $p$. In particular, $\langle\langle H\rangle\rangle_G\neq G$.
        \item If $G$ is special, then all $G$--parabolic subgroups are $p$--separable in $G$.
    \end{enumerate}
\end{lem}
\begin{proof}
    Part~(1) follows from the observation that, if $F$ is a finite $p$--group and $F_0$ is a maximal proper subgroup of $F$, then $F_0$ is normal in $F$ (and has index $p$). See for instance \cite[Theorem~4.3.2]{Hall}.

    For part~(2), note that retracts of residually $p$--finite groups are $p$--se\-pa\-rable, by the argument used to prove \cite[Lemma~9.2]{HW08}. RAAGs are residually $p$--finite for all $p$, by \cite[Theorem~6.1]{Toinet}, and parabolic subgroups of RAAGs are clear retracts, so they are $p$--separable. For $G\leq A_{\G}$, each $G$--parabolic subgroup is the intersection between $G$ and a parabolic subgroup of $A_{\G}$, and so it is $p$--separable in $G$.
\end{proof}

We will also need the following observation, which can be used to promote accessibility to unconditional accessibility when chains of edge-stabilisers have bounded length. We refer to \cite[Lemma~2.2]{Fio11b} for a proof.

\begin{lem}\label{lem:unconditional_vs_conditional_acc}
    Let $G$ be a finitely generated group that is accessible over a family $\mc{F}$ of subgroups. If there is a uniform bound on the length of chains of subgroups in $\mc{F}$, then $G$ is unconditionally accessible over $\mc{F}$.
\end{lem}

We are now ready to prove the main ingredient of our accessibility result, namely \Cref{prop:acc_from_transverse_cover} below. Say that a group $G$ is \emph{$N$--accessible} over a family of subgroups $\mc{F}$, for some integer $N\geq 0$, if every reduced splitting of $G$ over $\mc{F}$ has at most $N$ orbits of edges. We denote by $b_1(\cdot)$ the 1st Betti number of a group. We will also often speak of \emph{collapses} of $G$--trees: if $G\acts T$ is a simplicial tree and $\mc{E}$ is a $G$--invariant set of edges, we can form a new tree $G\acts T'$ by shrinking each edge in $\mc{E}$ to a point. If $G\acts T$ was minimal or reduced, then so is $G\acts T'$. Elliptic subgroups for $T$ are elliptic for $T'$, and edge-stabilisers for $T'$ are edge-stabilisers for $T$.

\begin{prop}\label{prop:acc_from_transverse_cover}
    Let $G$ be a finitely generated group with a reduced splitting $G\acts T$. Let $U\sq T$ be a subtree such that all the following hold:
    \begin{enumerate}
        \item the subtrees in the family $\{gU\mid g\in G\}$ cover $T$, and distinct ones share at most one point;
        \item the $G$--stabiliser $G_U$ is $p$--separable in $G$ for some prime $p$;
        \item there exists an integer $N\geq 0$ such that:
        \begin{enumerate}
            \item[(a)] either $G_U$ is elliptic in $T$ and $N=0$;
            \item[(b)] or $G_U$ is $N$--accessible over stabilisers of edges of $U$.
        \end{enumerate}
    \end{enumerate}
    Then the quotient $T/G$ has at most $4b_1(G;\mbb{F}_p)+N$ edges.
\end{prop}
\begin{proof}
Note that the actions $G\acts T$ and $G_U\acts U$ have the same number of edge orbits. Indeed, Condition~(1) implies that every edge of $T$ has a $G$--translate in $U$, and that two edges of $U$ lie in the same $G$--orbit if and only if they are in the same $G_U$--orbit. Therefore, our goal in the rest of the proof is to bound the number of edges of the quotient $U/G_U$. 

We cannot simply invoke accessibility of $G_U$, as the action $G_U\acts U$ might not be minimal or reduced. Still, this failure is located at vertices where $U$ meets one of its $G$--translates (since $T$ is minimal and reduced), and the number of such vertices can be bounded uniformly in terms of $G$ alone.

\smallskip
{\bf Claim~1.} \emph{There are at most $2b_1(G;\mbb{F}_p)$ $G_U$--orbits of vertices where $U$ intersects its $G$--translates.}

\smallskip\noindent
\emph{Proof of Claim~1.}
We will study an auxiliary splitting of $G$, which we now describe. By Condition~(1), the family $\mc{Y}:=\{gU\mid g\in G\}$ is a \emph{transverse covering} of $T$ in the sense of \cite[Definition~1.4]{Guir-Fourier}. As explained there, $\mc{Y}$ gives rise to an action $G\acts S$, where $S$ is the following simplicial tree:
\begin{itemize}
    \item $S$ has a ``black'' vertex for each element of $\mc{Y}$;
    \item $S$ has a ``white'' vertex for each point of intersection between distinct subtrees in $\mc{Y}$;
    \item edges of $S$ join the white vertex representing a vertex $x\in T$ to each black vertex of $S$ representing a subtree in $\mc{Y}$ containing $x$.
\end{itemize}
The action $G\acts S$ is minimal, since $G\acts T$ is. 

Since $G$ acts transitively on $\mc{Y}$, there is a single orbit of black vertices in $S$. As $S$ is bipartite with respect to the black-white colouring, we can pick a fundamental domain for $G\acts S$ that realises $G$ as the fundamental group of the following graph of groups:
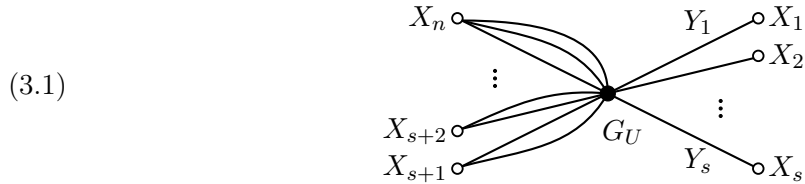
\begin{equation}\label{split:1}
    \begin{tikzpicture}[baseline=(current bounding box.center)]
    \draw[fill] (0,0) circle [radius=0.1cm];
    \draw[line width=0.25mm] (2,1) circle [radius=0.07cm];
    \draw[line width=0.25mm] (2,0.5) circle [radius=0.07cm];
    \draw[line width=0.25mm] (2,-1) circle [radius=0.07cm];
    \draw[thick] (0,0) -- (1.94,0.97);
    \draw[thick] (0,0) -- (1.94,0.47);
    \draw[thick] (0,0) -- (1.94,-0.97);
    \draw[fill] (1.5,-0.2) circle [radius=0.02cm];
    \draw[fill] (1.5,-0.1) circle [radius=0.02cm];
    \draw[fill] (1.5,-0.3) circle [radius=0.02cm];
    \draw[line width=0.25mm] (-2,1) circle [radius=0.07cm];
    \draw[line width=0.25mm] (-2,-0.5) circle [radius=0.07cm];
    \draw[line width=0.25mm] (-2,-1) circle [radius=0.07cm];
    \draw[thick] (0,0) [out=-120, in=15] to (-1.94,-0.97); 
    \draw[thick] (0,0) -- (-1.94,-0.97); 
    \draw[thick] (0,0) -- (-1.94,-0.47); 
    \draw[thick] (0,0) [out=175, in=25] to (-1.94,-0.47);
    \draw[thick] (0,0) -- (-1.94,0.97); 
    \draw[thick] (0,0) [out=90, in=0] to (-1.94,0.97); 
    \draw[thick] (0,0) [out=120, in=-15] to (-1.94,0.97); 
    \draw[fill] (-1.5,0.2) circle [radius=0.02cm];
    \draw[fill] (-1.5,0.1) circle [radius=0.02cm];
    \draw[fill] (-1.5,0.3) circle [radius=0.02cm];
    \node[right] at (2,1) {$X_1$};
    \node[right] at (2,0.5) {$X_2$};
    \node[right] at (2,-1) {$X_s$};
    \node[left] at (-2,-1) {$X_{s+1}$};
    \node[left] at (-2,-0.5) {$X_{s+2}$};
    \node[left] at (-2,1) {$X_n$};
    \node[below] at (0.2,-0.22) {$G_U$};
    \node[above] at (1.2,0.65) {$Y_1$};
    \node[below] at (1.2,-0.6) {$Y_s$};
    \node at (3.5,0) {.};
    \end{tikzpicture}
\end{equation}
This graph has a central black vertex labelled by the group $G_U$ and a finite number $n$ of adjacent white vertices, labelled by groups $X_i$. The first $s$ white vertices are joined to the black vertex by a unique edge, and the remaining ones by multiple edges. Each vertex group $X_i$ is the $G$--stabiliser of a vertex of $U$ and for $i\leq s$ each edge group $Y_i$ is the $G_U$--stabiliser of that same vertex of $U$. We do not name the groups labelling the multiple edges. 

Note that $G$--orbits of edges of $S$ are in bijection with $G_U$--orbits of vertices of $T$ where $U$ intersects one of its translates. Thus, Claim~1 boils down to bounding the size of the quotient graph $S/G$, which is our next goal.

Since $G$ acts minimally on $S$, each $Y_i$ is a proper subgroup of $X_i$ for $i\leq s$. Since $G_U$ is $p$--separable in $G$ by Condition~(2), each intersection $Y_i=X_i\cap G_U$ is $p$--separable in $X_i$. Since $Y_i<X_i$ is proper and $p$--separable, \Cref{lem:p_separability}(1) yields an epimorphism $\rho_i\colon X_i\twoheadrightarrow\Z/p\Z$ vanishing on $Y_i$.

We can therefore define an epimorphism of $G$ onto a free product $F$ of copies of $\Z$ and $\Z/p\Z$. The latter is the fundamental group of the graph of groups described as follows: the underlying graph is the same as in Splitting~\ref{split:1}; we replace by the trivial group all edge groups, as well as the black vertex group in the middle, and the left-hand white vertex groups; the right-hand white vertex groups get replaced by copies of $\Z/p\Z$. Concretely, the epimorphism $G\twoheadrightarrow F$ vanishes on $G_U$ and $X_{s+1},\dots,X_n$, as well as on all edge groups of Splitting~\ref{split:1}, while it equals $\rho_i$ on each $X_i$ with $i\leq s$. 

If $\s$ is the number of edges in $S/G$, the rank of the free product $F$ is $\s-(n-s)\geq \frac{1}{2}\s$. Hence $b_1(G;\mbb{F}_p)\geq b_1(F;\mbb{F}_p)\geq \frac{1}{2}\s$, proving Claim~1.
\hfill$\blacksquare$

\smallskip
Now, let $\mc{M}_U\sq U$ be the $G_U$--minimal subtree of $U$, if $G_U$ is not elliptic in $T$; otherwise, set $\mc{M}_U:=\{u\}$ for some $G_U$--fixed vertex $u\in U$. The action $G_U\acts U$ is cocompact, as it has the same number of edge orbits as the minimal action $G\acts T$, where $G$ is finitely generated. Thus, $U$ remains within bounded distance of the subtree $\mc{M}_U$.  

Choose finitely many finite subtrees $\Phi_i\sq U$ such that each $\Phi_i$ intersects $\mc{M}_U$ at a single vertex $u_i$, such that distinct $\Phi_i$'s do not share edges, and such that each $G_U$--orbit of edges of $U\setminus\mc{M}_U$ intersects the union $\bigcup\Phi_i$ in a single edge. Let $\Delta_f$ be the set of vertices that have degree $\leq 2$ within some $\Phi_i$ (excluding the base vertex $u_i$, if relevant). Let $\Delta_m\sq\mc{M}_U$ be a set of representatives for the $G_U$--orbits of vertices where the action $G_U\acts\mc{M}_U$ fails to be reduced: these are vertices whose $G_U$--stabiliser fixes one incident edge of $\mc{M}_U$ and acts transitively on the remaining incident edges of $\mc{M}_U$. If $G_U$ is elliptic in $T$, we simply set $\Delta_m:=\emptyset$. 
\begin{equation*}
    \begin{tikzpicture}[baseline=(current bounding box.center)]
    \filldraw[black,opacity=0.1,draw=black] (0,0) ellipse (3cm and 0.25cm);
    \draw (0,0) ellipse (3cm and 0.25cm);
    \draw[thick] (-1,0.23) -- (-1.2,0.7);
    \draw[thick] (-1.2,0.7) -- (-0.9,1.1);
    \draw[thick] (-1.2,0.7) -- (-2,1.3);
    \draw[fill] (-1,0.23) circle [radius=0.05cm];
    \draw[fill] (-1.2,0.7) circle [radius=0.05cm];
    \draw[fill] (-0.9,1.1) circle [radius=0.05cm];
    \draw[fill] (-2,1.3) circle [radius=0.05cm];
    \draw[fill] (-1.6,1) circle [radius=0.05cm];
    \draw[thick] (1,0.23) -- (1.2,0.7);
    \draw[thick] (1.2,0.7) -- (0.9,1.1);
    \draw[thick] (1.2,0.7) -- (1.6,1);
    \draw[fill] (1,0.23) circle [radius=0.05cm];
    \draw[fill] (1.2,0.7) circle [radius=0.05cm];
    \draw[fill] (0.9,1.1) circle [radius=0.05cm];
    \draw[fill] (1.6,1) circle [radius=0.05cm];
    \draw[fill] (0,.75) circle [radius=0.02cm];
    \draw[fill] (0.1,.75) circle [radius=0.02cm];
    \draw[fill] (-0.1,.75) circle [radius=0.02cm];
    \node[above right] at (-1.1,0.15) {$u_1$};
    \node[above left] at (1.1,0.15) {$u_k$};
    \node at (-1,1.5) {$\Phi_1$};
    \node at (1,1.5) {$\Phi_k$};
    \node at (0,0) {$\mc{M}_U$};
    \node at (-1.45,1.15) {$w$};
    \node at (1.75,1.15) {$v$};
    \end{tikzpicture}
\end{equation*}

\smallskip
{\bf Claim~2.} \emph{Each vertex in $\Delta_f\cup\Delta_m$ lies in a translate $gU\neq U$ with $g\in G$.}

\smallskip\noindent
\emph{Proof of Claim~2.}
First, by the definition of the $\Phi_i$, each degree--$1$ vertex $v\in\Phi_i\setminus\{u_i\}$ must also have degree $1$ within $U$. Since $T$ does not have degree--$1$ vertices, by the minimality of the $G$--action, it follows that $v$ is incident to an edge of $T$ not contained in $U$. Hence $v$ lies in a $G$--translate of $U$ distinct from $U$, as these translates cover $T$ by assumption.

Second, suppose that some $w\in\Phi_i\setminus\{u_i\}$ has degree $2$ within $\Phi_i$. Since $G_U$ leaves invariant $\mc{M}_U$, the $G_U$--stabiliser of $w$ coincides with the $G_U$--stabiliser of the edge of $\Phi_i$ incident to $w$ in the direction of $u_i\in\mc{M}_U$, and it acts transitively on the remaining edges of $U$ incident to $w$, again by the construction of the $\Phi_i$. Since the action $G\acts T$ is reduced, this again implies that $w$ lies in a $G$--translate of $U$ distinct from $U$.

The same argument applies to the vertices of $\Delta_m$, showing that they are also intersections of distinct translates of $U$. This proves Claim~2.
\hfill$\blacksquare$

\smallskip
Since no two vertices of $\Delta_f\cup\Delta_m$ are in the same $G_U$--orbit, the combination of Claims~1 and~2 shows that $|\Delta_f\cup\Delta_m|\leq 2b_1(G;\mbb{F}_p)$. Let $\mc{E}_{\Phi}$ be the set of edges contained in $\Phi_i$ for some $i$. Considering the finite tree formed by wedging the $\Phi_i$ at their base vertices $u_i$, and applying standard estimates, we see that $|\mc{E}_{\Phi}|\leq 2|\Delta_f|$. Hence we get $|\mc{E}_{\Phi}\cup\Delta_m|\leq 4b_1(G;\mbb{F}_p)$.

If $G_U$ is elliptic, this shows that $U/G_U$ has at most $4b_1(G;\mbb{F}_p)$ edges. If $G_U$ is not elliptic, the above means that, after collapsing at most $4b_1(G;\mbb{F}_p)$ $G_U$--orbits of edges, the action $G_U\acts U$ becomes minimal and reduced. This collapse then has at most $N$ edge orbits by Condition~(3). Either way, $U/G_U$ has at most $4b_1(G;\mbb{F}_p)+N$ edges, and so does $T/G$ as desired.
\end{proof}

We now deduce \Cref{thm:accessible} from \Cref{prop:acc_from_transverse_cover}. Consider a special group $G$ and realise $G$ as a convex-cocompact subgroup of a RAAG, in order to be able to speak of $G$--semi-parabolic subgroups. Recall that $G$ acts by conjugation on the set of $G$--parabolics $\mc{P}(G)$, and the quotient $\mc{P}(G)/G$ is finite by \Cref{lem:parabolics_cofinite}. For each $P\in\mc{P}(G)$, the normaliser $N_G(P)$ is $G$--parabolic by \Cref{lem:parabolic_normaliser2}(2), and hence the quotient $N_G(P)/P$ is finitely presented (as $N_G(P)$ and $P$ are special and hence finitely presented).

By Bestvina and Feighn's accessibility \cite{BF-complexity}, each finitely presented group $H$ admits a constant $\g(H)$ such that every reduced splitting of $H$ over abelian subgroups has at most $\g(H)$ edge orbits. 

\begin{cor}\label{cor:acc_over_semi-parabolic}
    If $G\acts T$ is a reduced splitting of $G$ over $G$--semi-parabolic subgroups, then the number of edges of the quotient $T/G$ is at most:
    \[ 4b_1(G;\Q)\cdot|\mc{P}(G)/G|+\sum_{[P]\in\mc{P}(G)/G}\g(N_G(P)/P) .\]
\end{cor}
\begin{proof}
    Every semi-parabolic subgroup contains a normal, co-abelian, $G$--parabolic subgroup. For each conjugacy class $[P]\in\mc{P}(G)/G$, consider the edges of $T$ whose stabiliser contains a conjugate of $P$ as a normal co-abelian subgroup, and collapse all other edges of $T$. We obtain a reduced splitting $G\acts T_{[P]}$. Bounding the number of edges in $T_{[P]}/G$ for each $[P]\in\mc{P}(G)/G$ then gives a bound on the number of edges in $T/G$.

    Thus, we assume in the rest of the proof that there exists $P\in\mc{P}(G)$ such that all edge-stabilisers of $T$ contain a conjugate of $P$ as a co-abelian subgroup. Our goal becomes showing that the quotient $T/G$ has at most $4b_1(G;\Q)+\g(N_G(P)/P)$ edges. For this, the plan is to invoke \Cref{prop:acc_from_transverse_cover} for the subtree of $P$--fixed points $\Fix(P)\sq T$, so we proceed to check that its assumptions are satisfied.
    
    Every edge-stabiliser $E$ has a conjugate that has $P$ itself as a co-abelian subgroup, and so the $G$--translates of $\Fix(P)$ cover $T$. At the same time, $E$ contains a unique $G$--conjugate of $P$ by \Cref{lem:weak_parabolic part}, and so no two $G$--translates of $\Fix(P)$ share an edge. The same argument shows that the $G$--stabiliser of $\Fix(P)$ is precisely the normaliser $N_G(P)$. The latter is $G$--parabolic by \Cref{lem:parabolic_normaliser2}(2), and hence $p$--separable for all primes $p$ by \Cref{lem:p_separability}(2). Finally, the action $N_G(P)\acts\Fix(P)$ factors through an action $N_G(P)/P\acts\Fix(P)$ with abelian edge-stabilisers, and the group $N_G(P)/P$ is $\g(N_G(P)/P)$--accessible over abelian subgroups by \cite{BF-complexity}. 
    
    \Cref{prop:acc_from_transverse_cover} now shows that $T/G$ has at most $4b_1(G;\mbb{F}_p)+\g(N_G(P)/P)$ edges, for all primes $p$. Finally, for large $p$, we have $b_1(G;\mbb{F}_p)=b_1(G;\Q)$, since $G$ and its abelianisation are finitely generated.
\end{proof}

\Cref{cor:acc_over_semi-parabolic} implies \Cref{thm:accessible}, using \Cref{lem:unconditional_vs_conditional_acc} and \Cref{rmk:semi-parabolic_chains} to promote accessibility to unconditional accessibility.

We conclude by recording an important consequence of \Cref{thm:accessible}. The following was proven in \cite{Fio11b} under the assumption that the special group $G$ be accessible over the family $\mc{Z}(G)$. Since centralisers are $G$--semi-parabolic, \Cref{thm:accessible} implies that this holds generally, and we obtain a powerful tool to analyse point-stabilisers of degenerations of $G$.

\begin{thm}\label{thm:acc_implies_nice}
    Let $G$ be a special group. Let $G\acts T$ be a minimal $\R$--tree with arc-stabilisers in $\mc{Z}(G)$. Let $\mscr{E}$ be the family of subgroups elliptic in $T$.
    \begin{enumerate}
        \item Point-stabilisers of $T$ are convex-cocompact in $G$, with respect to all convex-cocompact embeddings of $G$ into RAAGs.
        \item There are only finitely many $G$--conjugacy classes of point-stabilisers.
        \item If $H\leq G$ is not elliptic in $T$, then there exists a $(\mc{Z}(G),\mscr{E})$--splitting of $G$ in which $H$ is not elliptic.
    \end{enumerate}
\end{thm}

\section{Enhanced JSJ trees}\label{sect:JSJ}

In this section we prove \Cref{thmintro:JSJ}: every $1$--ended special group $G$ admits a (possibly trivial) $\Out(G)$--invariant simplicial $G$--tree over centralisers and cyclic subgroups, relative to singular subgroups. Vertex groups are of two kinds: either quadratically hanging, or rigid (see \Cref{thm:JSJ+}). We informally refer to such $G$--trees as \emph{enhanced JSJ decompositions} of $G$.

In \Cref{sub:canonical_splittings}, we develop techniques to construct $\Out(G)$--invariant splittings of $G$ over non-small subgroups, particularly over non-cyclic centralisers. Here the main result for special groups is \Cref{thm:Z_s-splitting}, though this material is more broadly applicable. Then, in \Cref{sub:JSJ-like}, we add classical JSJ theory in order to canonically split $G$ over cyclic subgroups as well, and prove \Cref{thmintro:JSJ} (or rather its strengthening \Cref{thm:JSJ+}). Finally, in \Cref{sub:JSJ_RAAG}, we briefly give examples of enhanced JSJ decompositions of RAAGs.

\subsection{Singular subgroups}\label{sub:singular}

Let $G$ be a special group. Let $\mc{VP}(G)$ be the collection of subgroups of $G$ that virtually split as the direct product of two infinite groups. We are interested in maximal such subgroups.

\begin{defn}
    Let $\mc{S}(G)$ denote the collection of maximal subgroups in $\mc{VP}(G)$. The elements of $\mc{S}(G)$ are the \emph{singular subgroups} of $G$. 
\end{defn}

The family $\mc{S}(G)$ is clearly $\Aut(G)$--invariant, and we show below that it consists of finitely many $G$--conjugacy classes of subgroups. Thus, a finite-index subgroup of $\Out(G)$ preserves each singular subgroup up to conjugacy, which is useful when running inductive arguments on ``complexity''. Many (though not all) elements of $\mc{Z}(G)$ are contained in singular subgroups, so this will naturally lead to considering splittings of $G$ relative to $\mc{S}(G)$.

Fix as usual some convex-cocompact embedding $G\hookrightarrow A_{\G}$ into a RAAG, so that we can speak of $G$--parabolics. We begin by focusing on a particular kind of singular subgroup. 

An abelian subgroup $A\leq G$ is \emph{isolated} if $Z_G(a)=A$ for all $a\in A\setminus\{1\}$. Rank--$1$ isolated abelians are plentiful: they are the maximal cyclic subgroups generated by contracting elements of $G$. By contrast, there are only finitely many isolated abelian subgroups of rank $\geq 2$ up to $G$--conjugacy:

\begin{lem}\label{lem:IA_properties}
      Isolated abelian subgroups of rank $\geq 2$ are $G$--parabolic.
\end{lem}
\begin{proof}
    Let $A$ be abelian, isolated, and of rank $\geq 2$. Note that $A$ is convex-cocompact (since it is a centraliser) and so it contains a convex-cocompact subgroup $C\leq A$ with $C\cong\Z$ (e.g.\ by \cite[Remark~3.7(6)]{Fio10e}). Now, the group $A=Z_G(C)$ virtually splits as $C\x A'$ for a $G$--parabolic subgroup $A'$, by \Cref{lem:cc_basics}(2). Since $\rk(A)\geq 2$, we have $A'\neq\{1\}$ and hence $Z_G(A')=A$. Thus, another application of \Cref{lem:cc_basics}(2) yields that $C\x A'$ has finite index also in the normaliser $N_G(A')$. Moreover, $N_G(A')$ is $G$--parabolic by \Cref{lem:parabolic_normaliser2}(2). In conclusion, the subgroups $A$ and $N_G(A')$ are commensurable and root-closed, and so they must coincide.
\end{proof}

The following collects the most important properties of the family $\mc{S}(G)$. To state part of this, it is convenient to introduce some notation:
\begin{align*}
    \mc{Z}_c(G)&:=\{Z_G(g) \mid g\in G,\ Z_G(g)\cong\Z\}, \\
    \mc{Z}_s(G)&:=\mc{Z}(G)\setminus\mc{Z}_c(G).
\end{align*}
Equivalently, $\mc{Z}_c(G)$ is the collection of contracting maximal cyclic subgroups of $G$. Note that $\mc{Z}_s(G)$ can still contain cyclic centralisers, just not any that are centralisers of a single element.

\begin{prop}\label{prop:S(G)_properties}
    The following hold.
    \begin{enumerate}
        \item Every element of $\mc{VP}(G)$ is contained in an element of $\mc{S}(G)$.
        \item All elements of $\mc{S}(G)$ are $G$--parabolic.
        \item We have $\mc{S}(G)=\emptyset$ if and only if $G$ is Gromov-hyperbolic.
        \item If $Z\in\mc{Z}_s(G)$ and $Z\neq\{1\}$, then the normaliser $N_G(Z)$ is contained in an element of $\mc{S}(G)$.
    \end{enumerate}
\end{prop}
\begin{proof}
    To begin with, consider a subgroup $H\leq G$ that splits as a nontrivial product $H_1\x H_2$. We claim that $H$ is contained either in an isolated abelian subgroup of $G$ (automatically of rank $\geq 2$), 
    or in the normaliser in $G$ of a co-abelian $G$--parabolic subgroup of some non-abelian element of $\mc{Z}(G)$. Note that these two kinds of subgroups are $G$--parabolic, by \Cref{lem:IA_properties} and \Cref{lem:parabolic_normaliser2}(2). Moreover, they lie in $\mc{VP}(G)$ by \Cref{lem:cc_basics}(2).
    
    To prove our claim, choose a co-abelian $G$--parabolic subgroup $P_1\lhd Z_G(H_1)$ (recall that centralisers are semi-parabolic). The subgroup $H_2$ normalises $P_1$ as $H_2\leq Z_G(H_1)$. At the same time, $H_1$ commutes with $Z_G(H_1)$ and thus it also normalises $P_1$. This shows that $H\leq N_G(P_1)$. If $Z_G(H_1)$ is non-abelian, this proves our claim. If instead $Z_G(H_1)$ is abelian, then $H_2$ is abelian and it follows that we have $H\leq Z_G(H_2)$. In this case, let $Z$ be a maximal element of $\mc{Z}(G)$ containing $H$. If $Z$ is non-abelian and $P\lhd Z$ is co-abelian and $G$--parabolic, we have $H\leq N_G(P)$, again proving our claim. Finally, suppose that $Z$ is abelian. This implies that $Z\leq Z_G(g)$ for each $g\in Z$ and, by maximality of $Z$, we must have $Z=Z_G(g)$ for $g\neq 1$. In other words, $Z$ is an isolated abelian subgroup, proving our claim.

    By the claim, each element of $\mc{VP}(G)$ is virtually contained in a $G$--parabolic element of $\mc{VP}(G)$. Since $G$--parabolics are root-closed, it follows that each element of $\mc{VP}(G)$ is \emph{entirely} contained in a $G$--parabolic element of $\mc{VP}(G)$. Since chains of $G$--parabolics have bounded length (\Cref{rmk:semi-parabolic_chains}), we obtain parts~(1) and (2) of the proposition.

    The rest is routine. We have $\mc{S}(G)=\emptyset$ if and only if $\mc{VP}(G)=\emptyset$, which occurs if and only if $G$ does not contain $\Z^2$ as a subgroup. For special groups, this is equivalent to hyperbolicity \cite{Gen-relhyp}. Regarding part~(4), note that each $Z\in\mc{Z}(G)$ with $Z\neq\{1\}$ is contained in $Z_G(g)$ for some $g\in G\setminus\{1\}$, and $Z_G(g)$ is either cyclic or in $\mc{VP}(G)$.
    Thus, each $Z\in\mc{Z}_s(G)$ with $Z\neq\{1\}$ is contained in an element of $\mc{VP}(G)$. Now, \Cref{lem:cc_basics}(2) shows that $N_G(Z)$ virtually splits as a product of the form $Z\x Z'$. If $Z'\neq\{1\}$, this implies that $N_G(Z)\in\mc{VP}(G)$. If instead $Z'$ is trivial, then $Z$ has finite-index in $N_G(Z)$, and so $N_G(Z)=Z$ because centralisers are root-closed. Either way, $N_G(Z)$ is contained in a virtual product and hence in a singular subgroup.
\end{proof}

By \Cref{prop:S(G)_properties}(4), an element $g\in G\setminus\{1\}$ lies in a singular subgroup of $G$ if and only if $Z_G(g)\not\cong\Z$. Moreover, the elements of the family $\mc{Z}_s(G)$ are precisely the elements of $\mc{Z}(G)$ that are contained in a singular subgroup (with the possible exception of the trivial subgroup). In particular, the family $\mc{Z}_s(G)$ is $\Aut(G)$--invariant and closed under intersections.

Soon, we will be considering splittings of $G$ whose edge groups are either centralisers or cyclic subgroups. Thus, it is convenient to denote by ${\rm Cyc}(G)$ the family of all infinite cyclic subgroups of $G$, and set
\[ \mc{ZZ}(G):=\mc{Z}(G)\cup{\rm Cyc}(G) .\]
We will often need to compare centralisers in the vertex groups of a splitting of $G$ to centralisers in $G$. The next result is the main tool for this.

\begin{lem}\label{lem:Z(V)_in_Z(G)}
    Let $V$ be a convex-cocompact
    vertex group of a tree $G\acts T$. 
    \begin{enumerate}
        \item If $G\acts T$ is a $(\mc{ZZ}(G),\mc{S}(G))$--splitting, then $\mc{Z}_s(V)\sq\mc{Z}_s(G)$ and $\mc{ZZ}(V)\sq\mc{ZZ}(G)$.
        \item If $G\acts T$ is a $(\mc{Z}(G),\mc{S}(G))$--splitting, then $\mc{Z}(V)\sq\mc{Z}(G)$.
    \end{enumerate}
\end{lem}
\begin{proof}
    We prove both parts simultaneously. Consider a $(\mc{ZZ}(G),\mc{S}(G))$--tree $G\acts T$. Consider a centraliser $Z\in\mc{Z}(V)$ and a subset $B\sq V$ such that $Z=Z_V(B)$. We can assume that $Z\neq\{1\}$.

    Suppose first that $Z_G(B)\in\mc{Z}(G)$ does not lie in $\mc{Z}_s(G)$. Thus, $Z_G(B)$ is a contracting cyclic subgroup of $G$, and hence $Z=V\cap Z_G(B)$ is a contracting cyclic subgroup of $V$. In particular, we have $Z\in{\rm Cyc}(G)\sq\mc{ZZ}(G)$ and $Z\not\in\mc{Z}_s(V)$, showing part~(1) in this case. Part~(2) follows by observing that, if edge groups of $T$ lie in $\mc{Z}(G)$, then vertex groups are root-closed by \Cref{rmk:root-closed_vertices}; hence, the cyclic group $Z_G(B)$ equals $Z=V\cap Z_G(B)$.

    Now, suppose instead that $Z_G(B)\in\mc{Z}_s(G)$. Thus, $Z_G(B)$ is contained in an element of $\mc{S}(G)$, and hence it is elliptic in $T$. If $Z_G(B)$ fixes a vertex $v\in T$ of which $V$ is the stabiliser, then $Z=Z_G(B)\in\mc{Z}_s(G)$. Otherwise $Z=V\cap Z_G(B)$ fixes an edge of $T$ incident to $v$, and hence $Z$ is the intersection between $Z_G(B)$ and an edge-stabiliser of $T$. In this case, $Z\in\mc{Z}(G)$ if edge-stabilisers of $T$ lie in $\mc{Z}(G)$, and $Z\in\mc{ZZ}(G)$ if edge-stabilisers of $T$ lie in $\mc{ZZ}(G)$. The latter also shows that $Z\in\mc{Z}_s(G)$ unless $Z\cong\Z$.
    
    We are left to show that $Z\in\mc{Z}_s(G)$ when $Z\cong\Z$ and $Z\in\mc{Z}_s(V)$. Note that $Z=Z_V(B)=\bigcap_{b\in B}Z_V(b)$, where each $Z_V(b)$ again lies in $\mc{Z}_s(V)$. Now, we have $Z_V(b)\not\cong\Z$ for each $b\in B$, as these are centralisers of single elements (and not contracting). The previous discussion then shows that all centralisers $Z_V(b)$ lie in $\mc{Z}_s(G)$ and, since the latter collection is closed under taking intersections, this finally shows that $Z\in\mc{Z}_s(G)$.
\end{proof}

There is no simultaneous refinement of the two parts of \Cref{lem:Z(V)_in_Z(G)}: for $(\mc{ZZ}(G),\mc{S}(G))$--splittings, there can be elements of $\mc{Z}_c(V)$ not in $\mc{Z}(G)$.

\subsection{Constructing canonical splittings}\label{sub:canonical_splittings}

The goal of this subsection is to develop a general procedure that, starting with a sufficiently nice splitting $G\acts T$, produces an $\Out(G)$--invariant splitting of $G$ with similar edge groups and elliptic subgroups. Our construction (\Cref{subsub:invariant_splitting}) shares some similarities with Guirardel and Levitt's construction of trees of cylinders \cite{GL-cyl,GL-JSJ}, but an important difference is that we allow edge-stabilisers to have a complicated intersection pattern. As a consequence, the construction can be applied beyond relatively hyperbolic groups and small splittings, provided that certain conditions are met. 

We will work with a general group $G$ and simply require the splittings to satisfy certain axioms. The main consequence for special groups is:

\begin{thm}\label{thm:Z_s-splitting}
    Let $G$ be special and let $\mc{O}\leq\Out(G)$ be a subgroup. Suppose that $G$ is $1$--ended relative to an $\mc{O}$--invariant collection of subgroups $\mc{H}$ that contains $\mc{S}(G)$. Then $G$ admits an $\mc{O}$--invariant $(\mc{Z}_s(G),\mc{H})$--tree whose vertex groups are $(\mc{Z}_s(G),\mc{H})$--rigid in $G$.
\end{thm}

Taking $\mc{O}=\Out(G)$ and $\mc{H}=\mc{S}(G)$, one gets an $\Out(G)$--invariant splitting of $G$ if $G$ is $1$--ended and not $(\mc{Z}_s(G),\mc{S}(G))$--rigid.

\begin{rmk}\label{rmk:acylindrical}
    If $G$ is $1$--ended, then all $(\mc{Z}_s(G),\mc{S}(G))$--trees of $G$ are acylindrical. Indeed, for any such tree $G\acts T$, there are only finitely many $G$--conjugacy classes of subgroups of the form $E\cap E'$, where $E,E'$ are edge-stabilisers; this follows from \Cref{lem:cc_basics}(5). In addition, for any such intersection, the action $N_G(E\cap E')\acts\Fix(E\cap E')$ is elliptic by \Cref{prop:S(G)_properties}(4) (if $E\cap E'\neq\{1\}$) and cocompact by \Cref{lem:cc_edges}(2). As a consequence, there is a bound, depending only on $T$, on the diameter of $\Fix(E\cap E')$.

    However, a priori, $(\mc{Z}_s(G),\mc{S}(G))$--trees of $G$ are not \emph{uniformly} acylindrical; they might get less and less acylindrical as the number of $G$--orbits of edges increases. For this reason, we cannot make do with Sela's acylindrical accessibility \cite{Sela-acyl-acc} in this section, and we truly need \Cref{thm:accessible}.
\end{rmk}
%


We now discuss the construction of canonical splittings. After some preliminary material, \Cref{subsub:invariant_defspace} finds an $\mc{O}$--invariant deformation space using accessibility. Then \Cref{subsub:invariant_splitting} constructs an $\mc{O}$--invariant splitting within a (suitable) $\mc{O}$--invariant deformation space. Finally, in \Cref{subsub:special_consequences}, we restrict to special groups and prove \Cref{thm:Z_s-splitting}.

\subsubsection{Refinements, collapses and deformation spaces}\label{subsub:refinements}

Let $G$ be a group. If $T$ and $S$ are $G$--trees, a $G$--equivariant map $\pi\colon T\ra S$ is a \emph{collapse} if it preserves alignment of triples of vertices; equivalently, $S$ is obtained from $T$ by collapsing some $G$--orbits of edges. The tree $T$ is a \emph{refinement} of $S$.

\begin{rmk}\label{rmk:restriction_new}
    If a subgroup $\mc{O}\leq\Out(G)$ preserves the $G$--conjugacy class of a subgroup $H\leq G$, we can define a \emph{restriction} $\mc{O}|_H\leq\Out(H)$: this the subgroup of outer classes $[\varphi|_H]$, where $\varphi\in\Aut(G)$ is an automorphism with $\varphi(H)=H$ and $[\varphi]\in\mc{O}$. The conjugacy action $N_G(H)\acts H$ also determines a subgroup $C^G_H\leq\Out(H)$, and we always have $C^G_H\lhd\mc{O}|_H$. For $\phi\in\mc{O}$, the possible restrictions of $\phi$ to $H$ form a coset $\phi|_H\cdot C^G_H\sq \mc{O}|_H$.
\end{rmk}

The following observation allows us to blow up vertices of a $G$--tree.

\begin{lem}\label{lem:blow-up}
    Let $G\acts T$ be a splitting. Consider a vertex $v\in T$ and let $\mc{E}_v$ be the collection of $G$--stabilisers of edges of $T$ incident to $v$. Suppose that the stabiliser $G_v$ admits a splitting $G_v\acts S$ relative to $\mc{E}_v$.
    \begin{enumerate}
        \setlength\itemsep{.25em}
        \item There exists a splitting $G\acts T'$ with an equivariant collapse map $\pi\colon T'\ra T$ such that the preimage $\pi^{-1}(v)$ is $G_v$--equivariantly isomorphic to $S$, while $\pi^{-1}(w)$ is a singleton for all $w\not\in G\cdot v$.
        \item Let $\mc{O}\leq\Out(G)$ preserve the $G$--conjugacy class of $G_v$ and restrict to a subgroup $\mc{O}_v\leq\Out(G_v)$. Suppose that $T$ and $S$ are, respectively, $\mc{O}$-- and $\mc{O}_v$--invariant. Suppose that $G$ and $G_v$ have trivial centre, and that neither $T$ nor $S$ is a line. Finally suppose that, for each $E\in\mc{E}_v$, the tree $\Fix(E;S)$ has finite diameter. Then, the splitting $T'$ in part~(1) can be constructed so that it is $\mc{O}$--invariant. 
    \end{enumerate}
\end{lem}
\begin{proof}
    For part~(1), see for instance \cite[Lemma~4.12]{GL-JSJ}. 

    We prove part~(2). Let $\tilde{\mc{O}}\leq\Aut(G)$ be the preimage of $\mc{O}$ under the quotient projection $\Aut(G)\ra\Out(G)$. The subgroup of inner automorphisms of $G$ is contained in $\tilde{\mc{O}}$ and, since the centre of $G$ is trivial, we can view this as a copy $G\lhd\tilde{\mc{O}}$. We will denote $\tilde{\mc{O}}$ by the more suggestive ``$G\utimes\mc{O}$'', bearing in mind that $\mc{O}$ might not lift to a subgroup of $G\utimes\mc{O}$. Since $G_v$ also has trivial centre, the restriction $\mc{O}_v$ similarly lifts to $G_v\utimes\mc{O}_v\leq\Aut(G_v)$.

    Let $\rho\colon G\ra\Aut(T)$ be the homomorphism corresponding to the action $G\acts T$. Since $T$ is $G$--minimal and $T\not\cong\R$, the centraliser of $\rho(G)$ in $\Aut(T)$ is trivial. Thus, each automorphism $\varphi\in G\utimes\mc{O}\leq\Aut(G)$ admits a \emph{unique} $\Phi\in\Aut(T)$ such that $\Phi(gx)=\varphi(g)\Phi(x)$ for all $g\in G$ and $x\in T$. Uniqueness guarantees that we obtain a homomorphism $\tilde\rho\colon G\utimes\mc{O}\ra\Aut(T)$ extending $\rho$. The action $G_v\acts S$ similarly extends to $G_v\utimes\mc{O}_v\acts S$.

    Now, pick a vertex $v\in T$ and consider its stabiliser $(G\utimes\mc{O})_v$. Its intersection with $G$ is the stabiliser $G_v$, and we have $G_v\lhd (G\utimes\mc{O})_v$. Thus, all automorphisms of $G$ lying in $(G\utimes\mc{O})_v$ leave $G_v$ invariant and, by restriction, we obtain a homomorphism $(G\utimes\mc{O})_v\ra G_v\utimes\mc{O}_v$. The action $G_v\utimes\mc{O}_v\acts S$ then induces an action $(G\utimes\mc{O})_v\acts S$ extending $G_v\acts S$.

    We now wish to apply part~(1) of the lemma to the actions $G\utimes\mc{O}\acts T$ and $(G\utimes\mc{O})_v\acts S$, so as to produce the required blowup $G\utimes\mc{O}\acts T'$. For this, we need to check that, for every edge $e\sq T$ incident to $v$, the stabiliser $(G\utimes\mc{O})_e$ is elliptic in $S$. As in the case of the stabiliser of $v$, we have $G_e\lhd (G\utimes\mc{O})_e$. By hypothesis, the set of fixed points of $G_e\in\mc{E}_v$ in $S$ has finite diameter, and so it has a unique barycentre. The group $(G\utimes\mc{O})_e$ preserves the fixed set of its normal subgroup $G_e$, and hence it fixes its barycentre. This shows that $(G\utimes\mc{O})_e$ is indeed elliptic in $S$, concluding the proof.
\end{proof}

\begin{rmk}
    The finite-diameter assumption in part~(2) of \Cref{lem:blow-up} is truly necessary. For instance, consider the free group $F_3=\langle a,b,c\rangle$ and $\varphi\in\Aut(F_3)$ given by $a\mapsto ab$, $b\mapsto bc$, $c\mapsto c$. There is a $\varphi$--invariant free HNN splitting $F_3\acts T$ with $V:=\langle b,c\rangle$ as a vertex group. The restriction $\varphi|_V$ in turn preserves a free HNN splitting $V\acts S$ with $\langle c\rangle$ as a vertex group. However, the refinement of $T$ given by $S$ is not $\varphi$--invariant.
\end{rmk}

Two $G$--trees lie in the same \emph{deformation space} if they have the same elliptic subgroups. A $G$--tree $T_1$ \emph{dominates} a $G$--tree $T_2$ if every subgroup of $G$ that is elliptic in $T_1$ is also elliptic in $T_2$. The following is a classical consequence of \Cref{lem:blow-up}(1); see e.g.\ \cite[Lemma~2.4]{Fio11b} for a proof.

\begin{lem}\label{lem:ref_dom}
    Let $G$ be finitely generated, and let $G\acts T_1$ and $G\acts T_2$ be two splittings. Suppose that the edge-stabilisers of $T_1\cup T_2$ are elliptic in both $T_1$ and $T_2$. Then there is a splitting $G\acts T$ such that:
    \begin{enumerate}
        \item $T$ is a refinement of $T_1$ that dominates $T_2$;
        \item a subgroup is elliptic in $T$ if and only if it is in both $T_1$ and $T_2$;
        \item edge-stabilisers of $T$ are either edge-stabilisers of $T_1\cup T_2$, or intersections between an edge-stabiliser of $T_1$ and an edge-stabiliser of $T_2$;
        \item $T$ does not have any degree--$2$ vertices whose stabiliser fixes the two incident edges, except for any that $T_1$ might have already had.
    \end{enumerate}
\end{lem}


\subsubsection{Finding an invariant deformation space}\label{subsub:invariant_defspace}

Let $G$ be a finitely generated group. Given a suitable $G$--tree and provided that $G$ is suitably accessible, our first goal is to construct an $\mc{O}$--invariant deformation space. 

Say that $G$ is \emph{unconditionally $(\mc{F},\mc{K})$--accessible}, for two families $\mc{F}$ and $\mc{K}$, if there is a uniform bound on the number of edge orbits in irredundant $(\mc{F},\mc{K})$--splittings of $G$. For $\mc{K}=\emptyset$, we recover \Cref{defn:accessible}. If $\mc{O}\leq\Out(G)$, we say that a subgroup family $\mc{F}$ is \emph{$\mc{O}$--invariant} if it is closed under taking $G$--conjugates and if $\mc{O}$ permutes the conjugacy classes of subgroups in $\mc{F}$. We denote by $\mc{F}_{\rm int}$ the collection of finite intersections of elements of $\mc{F}$. 

\begin{cor}\label{cor:invariant_defspace}
    Let $G$ be finitely generated. Consider a splitting $G\acts T$ and some $\mc{O}\leq\Out(G)$. Suppose that there exist two $\mc{O}$--invariant families $\mc{K}\supseteq\mc{F}$ such that $T$ is an $(\mc{F},\mc{K})$--splitting. Then, if $G$ is unconditionally $(\mc{F}_{\rm int},\mc{K})$--accessible, there is a refinement $G\acts T'$ of $T$ such that:
    \begin{enumerate}
        \item the edge-stabilisers of $T'$ belong to $\mc{F}_{\rm int}$;
        \item the subgroups of $G$ that are elliptic in $T'$ are precisely those whose entire $\mc{O}$--orbit is elliptic in $T$.
    \end{enumerate}
\end{cor}
\begin{proof}
    Let $\mc{T}$ be the collection of splittings obtained by twisting $G\acts T$ by the elements of $\mc{O}$. All elements of $\mc{T}$ are $(\mc{F},\mc{K})$--splittings, because $\mc{F}$ and $\mc{K}$ are $\mc{O}$--invariant. Moreover, the edge-stabilisers of any element of $\mc{T}$ are elliptic in all elements of $\mc{T}$, because $\mc{F}\subseteq\mc{K}$.

    Let $G\acts T_n$ be an enumeration of the elements of $\mc{T}$, with $n\in\N$. \Cref{lem:ref_dom} yields a sequence $G\acts T_n'$ of $(\mc{F}_{\rm int},\mc{K})$--trees such that $T_0'=T_0$, each $T_{n+1}'$ is a refinement of $T_n'$, and a subgroup of $G$ is elliptic in $T_n'$ if and only if it is elliptic in all of $T_0,\dots,T_n$. Moreover, the trees $T_n'$ have a uniformly bounded number of orbits of degree--$2$ vertices whose stabiliser fixes the two incident edges. Since $G$ is unconditionally $(\mc{F}_{\rm int},\mc{K})$--accessible, there is a uniform bound on the number of edge orbits of the $T_n'$, and hence the sequence of refinements $T_n'$ eventually stabilises to a splitting $T'$. This has the properties claimed by the corollary. 
\end{proof}

\subsubsection{Finding an invariant splitting}\label{subsub:invariant_splitting}

Consider now a group $G$, a cocompact\footnote{Cocompactness is immediate if $G$ is finitely generated.} splitting $G\acts T$, and a subgroup $\mc{O}\leq\Out(G)$. Let $\mscr{E}$ be the collection of subgroups of $G$ that are elliptic in $T$, and let $\mscr{A}$ be the collection of $G$--stabilisers of edges of $T$. Equivalently, $\mscr{A}$ and $\mscr{E}$ are, respectively, the smallest and largest families such that $T$ is an $(\mscr{A},\mscr{E})$--splitting.

We are interested in splittings satisfying the following three conditions:
\begin{enumerate}
    \item[$(i)$] the collection $\mscr{E}$ is $\mc{O}$--invariant;
    \item[$(ii)$] no element of $\mscr{A}$ properly contains one of its $G$--conjugates;
    \item[$(iii)$] for every $A\in\mscr{A}$, the normaliser $N_G(A)$ is elliptic in $T$.
\end{enumerate}
Condition~(i) asks that the deformation space of $T$ be $\mc{O}$--invariant, while Condition~(iii) can be viewed as a weak form of acylindricity plus a ban on edges with trivial stabiliser. Our goal is to prove:

\begin{thm}\label{thm:invariant_splittings}
    If $G,\mc{O},T$ satisfy Conditions~$(i)$--$(iii)$, then there exists an $\mc{O}$--invariant $(\mscr{A},\mscr{E})$--splitting $G\acts T'$.
\end{thm}

The first version of this article had a stronger statement in place of \Cref{thm:invariant_splittings} (the deformation space of $T'$ can be chosen to equal that of $T$, rather than just be dominated by it). I am grateful to Vincent Guirardel for pointing out that the above suffices and can be obtained more easily. 

Before proving \Cref{thm:invariant_splittings}, we need to obtain a few lemmas. Here, a fourth and last condition is useful:
\begin{enumerate}
    \item[$(*)$] no vertex-stabiliser of $T$ fixes an edge of $T$.
\end{enumerate}
Denote by $\mscr{E}_{\max}\sq\mscr{E}$ the subset of maximal elliptic subgroups.

\begin{lem}\label{lem:(*)}
    If $G\acts T$ satisfies Condition~$(ii)$, then it admits a collapse $G\acts T_*$ satisfying Condition~$(*)$. A subgroup of $G$ is elliptic in $T_*$ if and only if it either lies in $\mscr{E}$ or normalises a subgroup in $\mscr{E}_{\max}$.
\end{lem}
\begin{proof}
    Say that $(e,v)$ is a \emph{bad pair} if $e\sq T$ is an edge, $v\in e$ is a vertex, and the stabiliser of $v$ fixes $e$. If $w$ is the other vertex of $e$, the stabilisers of $e,v,w$ satisfy $G_v=G_e\leq G_w$. We say that $(e,v)$ is of the \emph{first kind} if $G_e=G_w$, and of the \emph{second kind} if $G_e\lneq G_w$. Condition~$(*)$ is satisfied precisely when there are no bad pairs, of any kind.

    Consider a bad pair $(e,v)$. If $g\in G$ is such that $v\in ge$ and $ge\neq e$, then we cannot have $gv=v$, and so we must have $v=gw$. By Condition~$(ii)$, the latter can only happen if $G_v=G_w$. In particular, if the pair is of the second kind, then $e$ is the only edge in its $G$--orbit to contain the vertex $v$. In this case, the edges in the orbit $G\cdot e$ come arranged in pairwise-disjoint subtrees of $T$ of diameter $\leq 2$: each of these subtrees is a $G$--translate of the union of the edges in $G\cdot e$ that contain $w$. Collapsing these subtrees to points does not affect which subgroups of $G$ are elliptic. At the same time, such a collapse strictly reduces the number of edges in the quotient graph $T/G$, which is finite by the assumption that $G$ acts cocompactly. Thus, after a finite sequence of collapses preserving the collection of elliptic subgroups, we can assume that all bad pairs are of the first kind. 

    Now, let $(e,v)$ be a bad pair and consider the fixed subtree $\Fix(G_e)\sq T$. Since there are no bad pairs of the second kind, all vertices and all edges of $\Fix(G_e)$ have the same $G$--stabiliser, namely $G_e$.
    In particular, for any edge $f\sq T$ intersecting $\Fix(G_e)$ at a single vertex, we have $G_f\lneq G_e$. This also shows that $G_e\in\mscr{E}_{\rm max}$ and that the $G$--stabiliser of $\Fix(G_e)$ is precisely the normaliser $N_G(G_e)$. Finally, for any subgroup $H\in\mscr{E}_{\max}$, either $\Fix(H)$ is a single vertex and $N_G(H)$ fixes it, or $\Fix(H)$ contains an edge, in which case all vertices and edges of $\Fix(H)$ have stabiliser exactly $H$, and all edges of $\Fix(H)$ belong to bad pairs. In particular, we have $\Fix(H)\cap\Fix(H')=\emptyset$ for all distinct subgroups $H,H'\in\mscr{E}_{\max}$.

    Summing up, edges belonging to bad pairs are precisely those fixed by elements of $\mscr{E}_{\max}$. Collapsing all these edges ensures Condition~$(*)$, and a subgroup is elliptic in the resulting tree $T_*$ if and only if it either lies in $\mscr{E}$ or normalises an element of $\mscr{E}_{\max}$. 
\end{proof}

To some extent, Condition~$(*)$ allows us to reconstruct the splitting $T$ algebraically, as the next remark and lemma show. Let $\mscr{A}_{\min}\sq\mscr{A}$ be the subset of subgroups that are not properly contained in an element of $\mscr{A}$.

\begin{rmk}\label{rmk:E_max_vs_vertices}
    Condition~$(*)$ implies that vertex-stabilisers of $T$ are the same as elements of $\mscr{E}_{\max}$, and each of these has a unique fixed point in $T$.
\end{rmk}

\begin{lem}\label{lem:A_is_invariant}
    If $G\acts T$ satisfies Conditions~$(i)$, $(ii)$ and $(*)$, then:
    \begin{enumerate}
        \item the collection $\mscr{A}$ is $\mc{O}$--invariant;
        \item collapsing the edges of $T$ with stabiliser in $\mscr{A}\setminus\mscr{A}_{\min}$ yields a splitting $G\acts T'$ whose deformation space is still $\mc{O}$--invariant.
    \end{enumerate}
\end{lem}
\begin{proof}
    We stratify $\mscr{A}$ by inclusion. That is, let $\mscr{A}_1\sq\mscr{A}$ be the subset of maximal elements, then inductively define $\mscr{A}_{i+1}$ as the set of maximal elements of $\mscr{A}\setminus\mscr{A}_i$. Since $G\acts T$ is cocompact and Condition~(ii) holds, we have $\mscr{A}_{k+1}=\emptyset$ for some integer $k\geq 1$; we define $k$ as the smallest such integer. 
    Each $\mscr{A}_i$ is conjugacy-invariant, and we have $\mscr{A}=\mscr{A}_1\sqcup\dots\sqcup\mscr{A}_k$. 
    
    By \Cref{rmk:E_max_vs_vertices}, the family $\mscr{E}_{\max}$ is $\mc{O}$--invariant, it coincides with the set of vertex-stabilisers of $T$, and each of its elements fixes a unique vertex of $T$. Define $\mscr{E}^1$ as the set of intersections $E\cap E'$ for distinct elements $E,E'\in\mscr{E}_{\max}$. Elements of $\mscr{E}^1$ are precisely $G$--stabilisers of non-degenerate arcs of $T$; moreover, $\mscr{E}^1$ is again $\mc{O}$--invariant. The set $\mscr{A}_1$ coincides with the set of maximal elements of $\mscr{E}^1$, so $\mscr{A}_1$ is $\mc{O}$--invariant. 
    
    Now, inductively for $i\geq 1$, let $\mscr{E}^{i+1}$ be the set of intersections $E\cap E'$ for distinct elements $E,E'\in\mscr{E}_{\max}$ such that there does not exist a sequence $A_1,\dots,A_k\in\mscr{A}_i$ with $A_1\leq E$, $A_k\leq E'$ and $\langle A_j, A_{j+1}\rangle\in\mscr{E}$ for all $1\leq j<k$. In other words, if $v,v'\in T$ are the unique vertices fixed by $E$ and $E'$ respectively, then the arc $[v,v']\sq T$ is non-degenerate and it is not covered by the fixed sets of the subgroups in $\mscr{A}_i$, so this arc contains an edge with stabiliser in $\mscr{A}\setminus\mscr{A}_i$. Since $\mscr{A}_i$ is $\mc{O}$--invariant by the inductive hypothesis, the set $\mscr{E}^{i+1}$ is $\mc{O}$--invariant, because of its algebraic description. It follows that the set of maximal elements of $\mscr{E}^{i+1}$ is also invariant, and this set is precisely $\mscr{A}_{i+1}$. This shows that $\mscr{A}_i$ is $\mc{O}$--invariant for every $i$, and hence $\mscr{A}$ is itself $\mc{O}$--invariant, proving part~(1).

    Regarding part~(2), let $\mc{E}_+$ be the collection of edges of $T$ with stabiliser in $\mscr{A}\setminus\mscr{A}_{\min}$, and let $T'$ be the tree obtained by collapsing all edges in $\mc{E}_+$. Denote by $\mc{U}_+\sq T$ the union of the $0$--skeleton of $T$ with all edges in $\mc{E}_+$; thus, a subgroup is elliptic in $T'$ precisely when it stabilises a connected component of $\mc{U}_+$. Define an equivalence relation on $\mscr{E}_{\max}$ as follows: we have $E\sim E'$ if there exists a finite sequence $E=:E_0,E_1,\dots,E_k:=E'$ such that each intersection $E_i\cap E_{i+1}$ contains an element of $\mscr{A}\setminus\mscr{A}_{\min}$. By part~(1), the families $\mscr{A}$ and $\mscr{A}_{\min}$ are $\mc{O}$--invariant, and so the equivalence relation $\sim$ is also $\mc{O}$--invariant. Note that we have $E\sim E'$ if and only if the vertices of $T$ fixed by $E$ and $E'$ lie in the same connected component of $\mc{U}_+$. Thus, a subgroup $H\leq G$ is elliptic in $T'$ if and only if there exists $E\in\mscr{E}_{\max}$ such that $hEh^{-1}\sim E$ for all $h\in H$. This shows that the family of elliptic subgroups of $T'$ is $\mc{O}$--invariant, proving part~(2).
\end{proof}

We are finally ready to prove \Cref{thm:invariant_splittings}.

\begin{proof}[Proof of \Cref{thm:invariant_splittings}]
    Let $G\acts T$ satisfy Conditions~$(i)$--$(iii)$. By Lem\-mas~\ref{lem:(*)} and~\ref{lem:A_is_invariant}(2), there is a collapse $G\acts T''$ whose deformation space is still $\mc{O}$--invariant, and such that $T''$ has additionally gained Condition~$(*)$ and the property that there are no proper inclusions between edge-stabilisers of $T''$. Note that, being a collapse $T$, it is automatic that $T''$ still satisfies Conditions~$(ii)$ and~$(iii)$ and that $T''$ is still an $(\mscr{A},\mscr{E})$--tree.

    Thus, we can simply assume that $T$ itself satisfies Conditions~$(i)$--$(iii)$ and Condition~$(*)$, and that $\mscr{A}=\mscr{A}_{\min}$. Now, consider the family of subtrees $\mc{Y}:=\{\Fix(A;T)\mid A\in\mscr{A}\}$. These cover $T$ and, since there are no proper inclusions between edge-stabilisers, distinct subtrees share at most one point. This means that $\mc{Y}$ is a \emph{transverse covering} in the sense of \cite{Guirardel-G&T-2004,Guir-Fourier}, and so it gives rise to a dual splitting $G\acts S$ constructed as follows. There is a ``black'' vertex of $S$ for each element of $\mc{Y}$, and a ``white'' vertex of $S$ for each intersection point between distinct elements of $\mc{Y}$; edges correspond to point-subset inclusions. Equivalently, $S$ has a black vertex for each element of $\mscr{A}$, a white vertex for each element of $\mscr{E}_{\max}$, and edges correspond to subgroup inclusions; the action $G\acts S$ is induced by the conjugacy action of $G$ on $\mscr{A}$ and $\mscr{E}_{\max}$. The latter description makes it clear that $S$ is $\mc{O}$--invariant. 
    
    Note that black vertex-stabilisers are of the form $N_G(A)$ with $A\in\mscr{A}$, while white vertex-stabilisers are precisely the elements of $\mscr{E}_{\max}$. The edge-stabilisers of $S$ are the vertex-stabilisers of the actions $N_G(A)\acts\Fix(A)$. Since $\mscr{A}=\mscr{A}_{\min}$, each action $N_G(A)\acts\Fix(A)$ factors through an action $N_G(A)/A\acts\Fix(A)$ with \emph{trivial} edge-stabilisers.

    Now, Condition~$(iii)$ implies that the action $N_G(A)\acts\Fix(A)$ is elliptic for each $A\in\mscr{A}$. Since this action factors through a free splitting of $N_G(A)/A$, it follows that there is a unique vertex $v_A\in\Fix(A)$ that is fixed by $N_G(A)$, and the $N_G(A)$--stabiliser of every vertex of $\Fix(A)\setminus\{v_A\}$ is precisely $A$. Also note that $\Fix(A)$ contains at least two vertices, because $A$ is an edge-stabiliser of $T$. This shows that, for each $A\in\mscr{A}$, the black vertex of $S$ corresponding to $A$ has at least one incident edge whose $G$--stabiliser is precisely $A$ (indeed, $S$ is minimal and so it has no degree--$1$ vertices). In conclusion, collapsing all edges of $S$ whose $G$--stabiliser does not lie in $\mscr{A}$, we obtain the required $\mc{O}$--invariant $(\mscr{A},\mscr{E})$--splitting of $G$.
\end{proof}

\subsubsection{Consequences for special groups}\label{subsub:special_consequences}

We now use the previous discussion to prove \Cref{thm:Z_s-splitting}. We will need the following notation.

\begin{nota}\label{nota:restrictions}
    Let $\mc{F}$ be a family of subgroups of $G$. For any subgroup $H\leq G$, there are two natural ways of restricting $\mc{F}$ to $H$, namely:
    \begin{align*}
        \mc{F}|^H&:=\{K\in\mc{F}\mid K\leq H\} , & \mc{F}|_H&:=\{K\leq H\mid \exists F\in\mc{F} \text{ such that } K\leq F\} .
    \end{align*}
\end{nota}

\begin{proof}[Proof of \Cref{thm:Z_s-splitting}]
    Let $\mscr{T}$ be the collection of minimal, irredundant, $\mc{O}$--invariant $(\mc{Z}_s(G),\mc{H})$--trees of $G$. We have $\mscr{T}\neq\emptyset$, as $\mscr{T}$ contains the tree that is a single vertex. We order $\mscr{T}$ by refinements: $T_1\preceq T_2$ if $T_1$ is a $G$--equi\-va\-riant collapse of $T_2$. Since $G$ is unconditionally accessible over centralisers (\Cref{thm:accessible}), the collection $\mscr{T}$ admits a maximal element $T$. The rest of the proof checks that the vertex groups of $T$ are $(\mc{Z}_s(G),\mc{H})$--rigid.

    Let $V$ be the $G$--stabiliser of a vertex $v\in T$, and let $\mc{E}_v$ be the collection of stabilisers of edges of $T$ incident to $v$. We begin with some observations on $V$. First, $V$ is convex-cocompact in $G$, by \Cref{lem:cc_edges}(1). Observing that $\mc{E}_v\sq\mc{Z}_s(G)|^V\sq\mc{S}(G)|_V\sq\mc{H}|_V$, we also see that $V$ is $1$--ended relative to $\mc{H}|_V$, otherwise \Cref{lem:blow-up}(1) could be used to violate $1$--endedness of $G$ relative to $\mc{H}$. By the same argument, accessibility of $G$ over centralisers implies that $V$ is unconditionally $(\mc{Z}_s(G)|^V,\mc{H}|_V)$--accessible.
    
    Now, suppose for the sake of contradiction that $V$ is non-elliptic in some $(\mc{Z}_s(G),\mc{H})$--splitting $G\acts U$, and let $M\sq U$ be the $V$--minimal subtree. 
    
    We claim that the action $V\acts M$ is a $(\mc{Z}_s(G)|^V,\mc{H}|_V)$--splitting. Indeed, the $V$--stabiliser of an edge of $M$ is an intersection $V\cap Z$ with $Z\in\mc{Z}_s(G)$. The centraliser $Z$ is contained in an element of $\mc{S}(G)\sq\mc{H}$, so it is elliptic in $T$. If $Z$ fixes $v$, then $V\cap Z=Z$. If it does not, we have $Z\cap V=Z\cap G_e$, where $e\sq T$ is an edge incident to $v$ and $G_e\in\mc{Z}_s(G)$. Either way, we have $V\cap Z\in\mc{Z}_s(G)|^V$. (We stress that $\mc{Z}_s(G)|^V\not\sq\mc{Z}(V)$ in general.)

    Let $\mc{O}|_V\leq\Out(V)$ be the restriction of $\mc{O}$ (\Cref{rmk:restriction_new}). We now claim that the action $V\acts M$ can be promoted to an (irredundant) $\mc{O}|_V$--invariant $(\mc{Z}_s(G)|^V,\mc{H}|_V)$--splitting of $V$. Once this is proved, it will complete the proof of the theorem: we can then use this splitting of $V$ and \Cref{lem:blow-up}(2) to refine $T$ into a larger (irredundant) $\mc{O}$--invariant $(\mc{Z}_s(G),\mc{H})$--splitting of $G$, contradicting the definition of $T$. The assumptions of \Cref{lem:blow-up}(2) are indeed satisfied: first, since $G\not\in\mc{S}(G)$ and $V\not\in\mc{S}(V)\sq\mc{H}|_V$, 
    both $G$ and $V$ have trivial centre and neither has a line-splitting over centralisers (as the kernel of the latter would virtually split by \Cref{lem:cc_basics}(2)). Moreover, $(\mc{Z}_s(\cdot),\mc{S}(\cdot))$--splittings are acylindrical, as shown in \Cref{rmk:acylindrical}.

    Thus, we are left to construct an $\mc{O}|_V$--invariant $(\mc{Z}_s(G)|^V,\mc{H}|_V)$--splitting of $V$. Recall that the collection $\mc{Z}_s(G)$ is $\mc{O}$--invariant, closed under intersections, and each element of $\mc{Z}_s(G)$ is contained in an element of $\mc{S}(G)\sq\mc{H}$ by \Cref{prop:S(G)_properties}(4). Thus, the collection $\mc{Z}_s(G)|^V$ is $\mc{O}|_V$--invariant and closed under intersections, the collection $\mc{H}|_V$ is $\mc{O}|_V$--invariant, and we have $\mc{Z}_s(G)|^V\sq\mc{H}|_V$. Moreover, as observed above, $V$ is unconditionally $(\mc{Z}_s(G)|^V,\mc{H}|_V)$--accessible. All this means that we can apply \Cref{cor:invariant_defspace} to $V\acts M$ and we obtain a new $(\mc{Z}_s(G)|^V,\mc{H}|_V)$--splitting $V\acts M'$ with the additional property that its deformation space is $\mc{O}|_V$--invariant.

    We then wish to apply \Cref{thm:invariant_splittings} to $V\acts M'$, for which we need to check the conditions of \Cref{subsub:invariant_splitting}. Condition~$(i)$ holds by construction, and Condition~(ii) follows from convex-cocompactness of the elements of $\mc{Z}_s(G)|^V$ (\Cref{lem:cc_basics}(3)). Finally, regarding Condition~$(iii)$, note that each edge-stabiliser of $M'$ is an element $Z\in\mc{Z}_s(G)|^V$ with $Z\neq\{1\}$, because $V$ is $1$--ended relative to $\mc{H}|_V$. Hence the normaliser $N_G(Z)$ is contained in an element of $\mc{S}(G)$ by \Cref{prop:S(G)_properties}(4), showing that $N_V(Z)$ is elliptic in $M'$. We can thus indeed appeal to \Cref{thm:invariant_splittings}, which yields the required $\mc{O}|_V$--invariant $(\mc{Z}_s(G)|^V,\mc{H}|_V)$--splitting $V\acts M''$. 
    
    Up to removing any degree--$2$ vertices, $M''$ is irredundant and this leads to the required contradiction, as explained earlier in the proof.
\end{proof}

\subsection{The enhanced JSJ tree}\label{sub:JSJ-like}

Let $G$ be special. In \Cref{thm:Z_s-splitting}, we con\-structed a canonical $(\mc{Z}_s(G),\mc{S}(G))$--tree with $(\mc{Z}_s(G),\mc{S}(G))$--rigid vertex groups (provided $G$ is $1$--ended). We now wish to refine this tree so that the vertex groups become $(\mc{Z}(G),\mc{S}(G))$--rigid (\Cref{thm:JSJ+}). 

In order to obtain the refinement, we need to consider splittings of $G$ over elements of $\mc{Z}_c(G)$, and this inevitably leads us to split over more general cyclic subgroups of $G$. We will use the classical JSJ decomposition of $G$ over cyclic subgroups \cite{RS97} and its canonical tree of cylinders \cite{GL-cyl,GL-JSJ}. In relation to this, it is convenient to introduce the following notion.

\begin{defn}[Optimal QH subgroup]\label{defn:optimal_QH}
    Let $G\acts T$ be a splitting relative to a family $\mc{H}$. Let $q\in T$ be a vertex whose stabiliser $Q$ is quadratically hanging relative to $\mc{H}$, and let $\Sigma$ be the compact hyperbolic surface with $Q=\pi_1(\Sigma)$. We say that $Q$ is \emph{optimal} if all the following hold:
    \begin{enumerate}
        \item the surface $\Sigma$ is not a pair of pants;
        \item for each edge $e\sq T$ incident to $q$, the stabiliser $G_e$ is either trivial or $Q$--conjugate to the entire fundamental group of a component of $\partial\Sigma$ (rather than to any subgroup thereof);
        \item for each component $B\sq\partial\Sigma$, there is at most one\footnote{Exactly one, if $G$ is $1$--ended relative to $\mc{H}$; see \cite[Lemma~5.16]{GL-JSJ}.} $Q$--orbit of edges $e\sq T$ incident to $q$ such that $G_e$ is $Q$--conjugate to $\pi_1(B)$.
    \end{enumerate}
    If these hold, we also say that $q$ is an \emph{optimal QH vertex} of the tree $T$.
\end{defn}

If $\Sigma$ is a compact surface, a simple closed curve $\g\sq\Sigma$ is \emph{essential} if it is neither nulhomotopic nor homotopic into the boundary of $\Sigma$. An \emph{essential multicurve} on $\Sigma$ is a finite union of pairwise disjoint, pairwise non-homotopic, essential simple closed curves on $\Sigma$. 

\begin{rmk}\label{rmk:non_cc_QH}
    Let $Q\leq G$ be a QH subgroup of $G$ (optimal or not) and let $\Sigma$ be the associated surface. Essential curves on $\Sigma$ give cyclic subgroups of $Q$ that lie in $\mc{Z}_c(G)$; in particular, these cyclic subgroups are always convex-cocompact. By contrast, the cyclic subgroups of $G$ corresponding to the boundary components of $Q$ are \emph{not} convex-cocompact, in general. For instance, $G$ might be the fundamental group of the special square complex obtained by gluing along the zig-zag diagonal of the standard squared $2$--torus a boundary component of a squared surface (with sufficiently big total angle at the boundary vertices, to ensure that no $3$--cubes are needed).
\end{rmk}

We now start working towards the construction of the enhanced JSJ tree (\Cref{thm:JSJ+}). \Cref{rmk:non_cc_QH} shows that we will be forced to split $G$ over non-convex-cocompact cyclic subgroups, and that our QH subgroups will not be convex-cocompact in general. At the same time, we still would like our non-QH vertex groups to be convex-cocompact, and the key to ensuring this lies in the following observation. This is where Items~(1) and~(2) of \Cref{defn:optimal_QH} become important.

\begin{lem}\label{lem:cc_near_QH}
    Let $G\acts T$ be any splitting with an optimal, quadratically hanging vertex $q\in T$. Suppose that every edge of $T$ has a $G$--translate that is incident to $q$. Then, for any vertex $w\in T\setminus (G\cdot q)$, the stabiliser $G_w$ is convex-cocompact and root-closed. 
\end{lem}
\begin{proof}
    Let $\Sigma$ be the surface with $Q=\pi_1(\Sigma)$. We can assume that $\Sigma$ is orientable. Otherwise, let $\Sigma_0$ be its orientable double cover and note that $Q_0:=\pi_1(\Sigma_0)$ contains $\pi_1(B)$ for each component $B\sq\partial\Sigma$.
    Thus, $G$ has an index--$2$ subgroup $G_0$ that contains the $G$--stabiliser of each vertex of $T\setminus (G\cdot q)$, and has $G_0\cap Q=Q_0$. We can then replace $G$ with $G_0$.

    Now, $\Sigma$ is the genus--$g$ surface with $b\geq 1$ boundary components $S_{g,b}$. 
    If $g=0$, we have $b\geq 4$; indeed, the pair of pants is ruled out by Item~(1) in \Cref{defn:optimal_QH}, and it is not a double cover of any surface.

    Let $B\sq\partial\Sigma$ be a component. We claim that there are two essential multicurves $\g_1,\g_2\sq\Sigma$ such that, if $T_1,T_2$ are the corresponding cyclic splittings of $\pi_1(\Sigma)$, then $\pi_1(B)$ is precisely the $\pi_1(\Sigma)$--stabiliser of some pair of vertices $(v_1,v_2)\in T_1\x T_2$. If $g=0$ and $b\geq 4$, we can take $\g_1,\g_2$ to be single curves, each bounding a pant with boundary components $B,\g_i,B_i$, for components $B_i\sq\partial \Sigma$ such that $B,B_1,B_2$ are pairwise distinct. If $g\geq 1$ and $b\geq 1$, then we can represent $\Sigma$ as a gluing $S_{g-1,b+1}\cup P$, where $P$ is a pair of pants with a component equal to $B$ and the remaining two components glued to two boundary components of $S_{g-1,b+1}$. Call $\g_1$ the multicurve in $\Sigma$ formed by the two boundary components of $P$ other than $B$ (to be precise, $\g_1$ is just one curve for $(g,b)=(1,1)$). We then define $\g_2=\phi\cdot\g_1$ for an element $\phi\in{\rm Mod}(\Sigma)$ that does not preserve $\g_1$.

    Now, the two splittings $Q\acts T_i$ can be used to refine $G\acts T$ into two splittings $G\acts T_{B,i}'$ (\Cref{lem:blow-up}(1)). We then collapse all edges of $T_{B,i}'$ coming from edges of $T$, obtaining two splittings $G\acts T_{B,i}''$ whose edge-stabilisers are conjugate to maximal cyclic subgroups represented by essential simple closed curves on $\Sigma$. In particular, these cyclic edge-stabilisers coincide with their centralisers in $G$, and hence they are root-closed and convex-cocompact. Thus, the $G$--stabiliser of any vertex of $T''_{B,i}$ is convex-cocompact by \Cref{lem:cc_edges}(1), and root-closed by \Cref{rmk:root-closed_vertices}.

    Finally, $G$--stabilisers of vertices of $T\setminus (G\cdot q)$ are intersections of vertex-stabilisers of the trees $T_{B,i}''$ as $B$ varies among the components of $\partial\Sigma$ and $i$ varies in $\{1,2\}$ (this uses Item~(2) in \Cref{defn:optimal_QH}). In conclusion, such stabilisers are root-closed and convex-cocompact, by \Cref{lem:cc_basics}(1).
\end{proof}

We also record here the following observation for later use. The proof is straightforward, using the same splittings as in the proof of \Cref{lem:cc_near_QH}. This is where Item~(3) of \Cref{defn:optimal_QH} is used.

\begin{lem}\label{lem:rigid->semi-rigid}
    Let $G\acts T$ be a splitting relative to a collection $\mc{H}$. Let $q\in T$ be an optimal, quadratically hanging vertex relative to $\mc{H}$. If a subgroup $H\leq G$ is $(\mc{Z}_c(G),\mc{H})$--rigid in $G$, then:
    \begin{enumerate}
        \item if $h\in H$ is loxodromic in $T$, the axis of $h$ does not contain $q$;
        \item if $H$ is contained in $Q$, it is peripheral.
    \end{enumerate}
\end{lem}

In order to canonically split $G$ over cyclic subgroups, we will use the following result, which can be quickly deduced from \cite{GL-JSJ}. The only subtlety is that, since we want non--QH vertex groups to be convex-cocompact, we need to be a little careful over which cyclic subgroups we split. From now on, all QH subgroups will be meant relative to $\mc{H}$ without explicit mention.

\begin{prop}\label{prop:actual_JSJ}
    Let $\mc{O}\leq\Out(G)$ and let $G$ be $1$--ended relative to an $\mc{O}$--invariant collection $\mc{H}\supseteq\mc{S}(G)$. There is an $\mc{O}$--invariant $({\rm Cyc}(G),\mc{H})$--tree $G\acts T$ such that the following hold.
    \begin{enumerate}
        \item Each vertex-stabiliser is of one of two kinds:
        \begin{enumerate}
            \item[(a)] an optimal quadratically hanging subgroup;
            \item[(b)] convex-cocompact, root-closed and $(\mc{Z}_c(G),\mc{H})$--rigid in $G$.
        \end{enumerate}
        \item Each edge $e\sq T$ with both vertices of type~(b) has $G_e\in\mc{Z}_c(G)$.
    \end{enumerate}
\end{prop}
\begin{proof}
    We will use the terminology of \cite{GL-JSJ}. Normalisers of elements of ${\rm Cyc}(G)$ lie in $\mc{Z}(G)$, hence they are either cyclic or contained in elements of $\mc{S}(G)\sq\mc{H}$; in particular, they are small in $({\rm Cyc}(G),\mc{H})$--trees. We can therefore appeal to \cite[Corollary~9.1]{GL-JSJ}, using commensurability as the equivalence relation on ${\rm Cyc}(G)$. The result is an $\mc{O}$--invariant $({\rm Cyc}(G),\mc{H})$--tree $G\acts T$ (denoted $(T_a)_c^*$ in the reference) with the following properties:
    \begin{itemize}
        \setlength\itemsep{.25em}
        \item vertex-stabilisers are either QH, or $({\rm Cyc}(G),\mc{H})$--rigid in $G$;
        \item $T$ is a \emph{JSJ tree} for $({\rm Cyc}(G),\mc{H})$: its edge-stabilisers are elliptic in all $({\rm Cyc}(G),\mc{H})$--trees, and its elliptic subgroups are elliptic in all $({\rm Cyc}(G),\mc{H})$--trees with such edge-stabilisers;
        \item $T$ is \emph{compatible} with all $({\rm Cyc}(G),\mc{H})$--trees: if $G\acts U$ is a one-edge $({\rm Cyc}(G),\mc{H})$--splitting, then $U$ is a collapse of $T$, or a collapse of the refinement of $T$ given by splitting a single QH vertex group over an essential simple closed curve on the associated surface. 
    \end{itemize} 
    (According to \cite{GL-JSJ}, we would have to worry about a third type of vertex group: subgroups of elements of $\mc{Z}(G)$. However, by \Cref{prop:S(G)_properties}(4), such vertex groups are either contained in elements of $\mc{S}(G)\sq\mc{H}$, or commensurable to edge groups of $T$. Hence they are $({\rm Cyc}(G),\mc{H})$--rigid.)

    We can assume that all QH vertex groups of $T$ are optimal. Indeed, QH vertex groups whose associated surface is a pair of pants are $({\rm Cyc}(G),\mc{H})$--rigid in $T$ (see \cite[Proposition~5.4]{GL-JSJ}), 
    and we regard them as such. For any other QH vertex $q\in T$, with stabiliser $Q$ and associated surface $\Sigma$, we can modify the tree $T$ near $q$ as follows. Let $Q\acts S_{\Sigma}$ be the $Q$--action on the diameter--$2$ tree that has a central vertex fixed by $Q$, and one $Q$--orbit of edges for each component of $B\sq\partial\Sigma$, with $Q$--stabilisers that are the conjugates of $\pi_1(B)$ within $Q=\pi_1(\Sigma)$. We can replace $q\in T$ with a copy of $S_{\Sigma}$: attach each edge $e\sq T$ incident to $q$ to the unique vertex of $S_{\Sigma}$ whose $Q$--stabiliser contains $G_e$ as a subgroup of finite index (we have $G_e\neq\{1\}$ because $G$ is $1$--ended relative to $\mc{H}$). Performing this procedure on all QH vertices of $T$, we obtain a refinement of $T$ in the same deformation space; we create some new edges and vertices, but their $G$--stabilisers are commensurable to $G$--stabilisers of previously existing edges, which are $({\rm Cyc}(G),\mc{H})$--rigid. Thus, the above three items still hold for the modified tree, and all QH vertices have become optimal. We are left to check that the modified tree is still $\mc{O}$--invariant. This follows from the fact that $\mc{O}$ preserves the collection of fundamental groups of boundary components of the surfaces associated to the QH vertices $q$ of the original tree $T$. In turn, this holds because, for each such boundary component $B$, there exists at least one edge incident to $q$ with stabiliser commensurable to $B$; indeed, this is due to the fact that $G$ is $1$--ended relative to $\mc{H}$, see \cite[Lemma~5.16]{GL-JSJ}.

    Thus, we assume that all QH vertices of $T$ are optimal. Let $G\acts T'$ be obtained from $T$ by collapsing edges as follows: we retain an edge $e\sq T$ only if $e$ is incident to an (optimal) QH vertex, or if its stabiliser $G_e$ lies in $\mc{Z}_c(G)$. 
    Note that $G\acts T'$ is still an $\mc{O}$--invariant $({\rm Cyc}(G),\mc{H})$--tree. We claim that $T'$ is the tree that we are looking for. Part~(2) is clear, as is the fact that all QH vertex groups of $T'$ are optimal. We only need to prove that the remaining vertex groups are of type~(b).
    
    If $v\in T'$ is a non--QH vertex, then all its incident edges $e\sq T'$ either have stabiliser $G_e\in\mc{Z}_c(G)$, or their other vertex is QH and optimal. Since all elements of $\mc{Z}_c(G)$ are convex-cocompact and root-closed, the combination of \Cref{lem:cc_edges}(1), \Cref{rmk:root-closed_vertices} and \Cref{lem:cc_near_QH} (applied to a suitable collapse of $T'$) shows that the stabiliser $G_v$ is convex-cocompact and root-closed.

    We are left to show that $G_v$ is elliptic in all $(\mc{Z}_c(G),\mc{H})$--splittings $G\acts U$. It suffices to show this when there is a single $G$--orbit of edges in $U$. In this case, the fact that $U$ is compatible with $T$ leaves only two options:
    \begin{itemize}
        \item either there is an edge $e\sq T$ with $G_e\in\mc{Z}_c(G)$ such that $U$ is obtained from $T$ by collapsing all edges outside of the orbit $G\cdot e$;
        \item or there are a QH vertex $q\in T$ and an essential curve $\g$ on the associated surface such that $U$ is obtained by first refining $T$, splitting $G_q$ over $\langle\g\rangle$, and then collapsing all original edges of $T$.
    \end{itemize}
    Either way, the construction of $T'$ makes it clear that $G_v$ is elliptic in such a tree $U$. Indeed, denoting by $\pi\colon T\ra T'$ the collapse map, $G_v$ stabilises the fibre $\pi^{-1}(v)\sq T$, which does not contain any QH vertices of $T$, nor does it contain any edges with $G$--stabiliser in $\mc{Z}_c(G)$. This completes the proof.
\end{proof}

Since we only consider \emph{optimal} QH vertex groups, we have the following.

\begin{lem}\label{lem:type(b)_addenda}
    Let $G\acts T$ be the splitting provided by \Cref{prop:actual_JSJ}. Considering a type~(b) vertex $v\in T$ and its stabiliser $V\leq G$.
    \begin{enumerate}
        \item We have $\mc{Z}_c(V)\sq\mc{Z}_c(G)$.
        \item If the collection $\mc{H}$ contains all cyclic subgroups of $G$ with finite $\mc{O}$--orbit, then the group $V$ is $(\mc{Z}_c(V),\mc{H}|_V)$--rigid in itself.
        \item The splitting $G\acts T$ is acylindrical.
    \end{enumerate}
\end{lem}
\begin{proof}
    For part~(1), consider an element $g\in V$ such that $Z_V(g)=\langle g\rangle$ and suppose for the sake of contradiction that $Z_V(g)\lneq Z_G(g)$. Since $V$ is root-closed in $G$, it follows that $Z_G(g)\in\mc{Z}_s(G)$; thus, $Z_G(g)$ is contained in an element of $\mc{S}(G)$ and it must fix a vertex $x\in T$. Since $Z_V(g)\neq Z_G(g)$, we have $x\neq v$; let $e\sq T$ be the edge incident to $v$ in the direction of $x$. Then $g$ fixes $e$ and, since $\langle g\rangle$ is maximal cyclic in $V$, we have $G_e=\langle g\rangle$. Since $Z_V(g)$ is contained in $Z_G(g)\in\mc{Z}_s(G)$, we have $G_e\not\in\mc{Z}_c(G)$. In particular, the vertex $w\in e\setminus\{v\}$ must be an optimal QH vertex. Since $G_w$ is hyperbolic, we have $x\neq w$, and hence there exists an edge $f\sq T$ incident to $w$ in the direction of $x$. Now, we have $g\in G_e\cap G_f$, but $G_e\cap G_f=\{1\}$ by \Cref{defn:optimal_QH}(3), yielding the required contradiction.

    Regarding part~(2), the assumption on $\mc{H}$ implies that it contains the $G$--stabiliser of every edge of $T$ incident to $v$. Thus, a $(\mc{Z}_c(V),\mc{H}|_V)$--splitting of $V$ could be used to refine $T$ and, collapsing the original edges of $T$, we would obtain a splitting of $G$ over an element of $\mc{Z}_c(V)$ relative to $\mc{H}$, in which $V$ is not elliptic. Since $\mc{Z}_c(V)\sq\mc{Z}_c(G)$ by part~(1), this would violate the fact that $V$ is $(\mc{Z}_c(G),\mc{H})$--rigid in $G$.

    As to part~(3), distinct edges incident to an optimal QH vertex have $G$--stabilisers with trivial intersection. Moreover, the fixed set in $T$ of any $Z\in\mc{Z}_c(G)$ has diameter at most the number of edges of the quotient $T/G$: otherwise, some $G$--orbit would contain distinct edges of $T$ with stabiliser equal to $Z$, violating the fact that $N_G(Z)=Z$. This proves the lemma.
\end{proof}

The following is the main result of \Cref{sect:JSJ} and it implies \Cref{thmintro:JSJ}. Recall that the family $\mc{ZZ}(G)$ was introduced in \Cref{sub:singular}.

\begin{thm}\label{thm:JSJ+}
Let $\mc{O}\leq\Out(G)$ be a subgroup, and let $G$ be $1$--ended relative to an $\mc{O}$--invariant collection $\mc{H}\supseteq\mc{S}(G)$. Then there exists an $\mc{O}$--invariant $(\mc{ZZ}(G),\mc{H})$--tree $G\acts T$ with the following properties.
\begin{enumerate}
    \setlength\itemsep{.25em}
    \item The $G$--stabiliser of each vertex of $T$ is of one of two kinds:
    \begin{enumerate}
        \item[(a)] an optimal quadratically hanging subgroup relative to $\mc{H}$;
        \item[(b)] convex-cocompact, root-closed and $(\mc{Z}(G),\mc{H})$--rigid in $G$.
    \end{enumerate}
    \item Each edge $e\sq T$ with both vertices of type~(b) has $G_e\in\mc{Z}(G)$.
    \item Let ${\rm Cyc}_{\mc{O}}(G)$ be the set of infinite cyclic subgroups of $G$ whose conjugacy class has finite $\mc{O}$--orbit. If $\mc{H}\supseteq{\rm Cyc}_{\mc{O}}(G)$, then each type~(b) vertex group $V\leq G$ is additionally $(\mc{Z}(V),\mc{H}|_V)$--rigid in itself.
    \item If a subgroup $H\leq G$ is $(\mc{Z}(G),\mc{H})$--rigid in $G$, then either $H$ fixes a type~(b) vertex of $T$, or $H$ fixes an edge of $T$.
\end{enumerate}
\end{thm}
\begin{proof}
    By \Cref{thm:Z_s-splitting}, there is an $\mc{O}$--invariant $(\mc{Z}_s(G),\mc{H})$--tree $G\acts T'$ whose vertex groups are $(\mc{Z}_s(G),\mc{H})$--rigid in $G$. By \Cref{lem:cc_edges}(1) and \Cref{rmk:root-closed_vertices}, its vertex groups are convex-cocompact and root-closed. 
    
    Consider a vertex-stabiliser $V$ of $T'$ and the restriction $\mc{O}|_V\leq\Out(V)$ (\Cref{rmk:restriction_new}). As in the proof of \Cref{thm:Z_s-splitting}, relative $1$--endedness of $G$ implies that $V$ is $1$--ended relative to $\mc{H}|_V$.
    We also have $\mc{S}(V)\sq\mc{H}|_V$ and $\mc{Z}_s(G)|^V\sq\mc{H}|_V$, by \Cref{prop:S(G)_properties}(4) and the fact that $\mc{S}(G)\sq\mc{H}$.
    
    We can thus apply \Cref{prop:actual_JSJ} to $V$. We obtain an $\mc{O}|_V$--invariant $({\rm Cyc}(V),\mc{H}|_V)$--tree $V\acts T_V$ whose QH vertex groups are optimal, and whose non--QH vertex groups are $(\mc{Z}_c(V),\mc{H}|_V)$--rigid in $V$, convex-cocompact and root-closed. Each edge of $T_V$ either has stabiliser in $\mc{Z}_c(V)$, or is incident to a QH vertex. Since $\mc{Z}_s(G)|^V\sq\mc{H}|_V$, \Cref{lem:blow-up}(2) allows us to use $T_V$ to refine $T'$ into an $\mc{O}$--invariant $(\mc{ZZ}(G),\mc{H})$--tree of $G$. To check that the hypotheses of the lemma are satisfied: since $G\not\in\mc{S}(G)$ and $V\not\in\mc{S}(V)\sq\mc{H}|_V$, both $G$ and $V$ have trivial centre, and neither $T'$ nor $T_V$ is a line. Moreover, $T_V$ is acylindrical by \Cref{lem:type(b)_addenda}(3).

    We perform this refinement for all $\mc{O}$--orbits of vertex groups of $T'$, and call $G\acts T$ the result. We claim that $T$ is the $G$--tree that we are looking for. Ignoring part~(4) for the moment, we only need to show that the non--QH vertex groups of $T$ satisfy the required rigidity properties. In turn, this simply amounts to showing that each non--QH vertex-stabiliser $W$ of one of the actions $V\acts T_V$ is $(\mc{Z}(G),\mc{H})$--rigid in $G$, and additionally $(\mc{Z}(W),\mc{H}|_W)$--rigid in itself if $\mc{H}\supseteq{\rm Cyc}_{\mc{O}}(G)$.

    Thus, let $W$ be a vertex-stabiliser of $V\acts T_V$. Clearly, $W$ is $(\mc{Z}_s(G),\mc{H})$--rigid in $G$, since $W$ is contained in $V$, which has this property. In addition, $W$ is $(\mc{Z}_c(G),\mc{H})$--rigid in $G$, because it is $(\mc{Z}_c(V),\mc{H}|_V)$--rigid in $V$ by construction. 
    This shows that $W$ is $(\mc{Z}(G),\mc{H})$--rigid in $G$.

    Finally, suppose that $\mc{H}\supseteq{\rm Cyc}_{\mc{O}}(G)$. If $w\in T$ is a vertex of which $W$ is the $G$--stabiliser, it follows that the collection $\mc{H}|_W$ contains all $G$--stabilisers of edges of $T$ incident to $W$: indeed, each of these subgroups has finite $\mc{O}$--orbit, by the $\mc{O}$--invariance of $T$, and all non-cyclic edge-stabilisers lie in $\mc{Z}_s(G)|^W\sq\mc{S}(G)|_W\sq\mc{H}|_W$. Thus, if $W$ were to admit a $(\mc{Z}_s(W),\mc{H}|_W)$--splitting, this could be used to refine $G\acts T$ into a $(\mc{Z}(G),\mc{H})$--splitting of $G$ in which $W$ is not elliptic, a contradiction (this uses that $\mc{Z}_s(W)\sq\mc{Z}(G)$, see \Cref{lem:Z(V)_in_Z(G)}). This shows that $W$ is $(\mc{Z}_s(W),\mc{H}|_W)$--rigid in itself, while it is $(\mc{Z}_c(W),\mc{H}|_W)$--rigid by \Cref{lem:type(b)_addenda}(2). In conclusion, $W$ is $(\mc{Z}(W),\mc{H}|_W)$--rigid in itself, completing the proof of parts~(1)--(3) of the theorem.

    As to part~(4), consider a subgroup $H\leq G$ that is $(\mc{Z}(G),\mc{H})$--rigid in $G$. If $H$ were not elliptic in $T$, then it would contain an element $h\in H$ that is loxodromic in $T$ (this holds even when $H$ is not finitely generated, using that chains in $\mc{ZZ}(G)$ have bounded length and e.g.\ \cite[Lemma~2.18]{Fioravanti-Kerr}). Let $\alpha\sq T$ be the axis of $h$. On the one hand, no edge of $\alpha$ can have $G$--stabiliser in $\mc{Z}(G)$ because $H$ is $(\mc{Z}(G),\mc{H})$--rigid. On the other, no vertex of $\alpha$ can be optimal and QH, by \Cref{lem:rigid->semi-rigid}(1), and this is a contradiction. In conclusion, $H$ fixes a vertex $v\in T$. If $v$ is of type~(a), then $H$ is peripheral in $G_v$ by \Cref{lem:rigid->semi-rigid}(2), and so $H$ fixes an edge of $T$.
\end{proof}

\begin{rmk}
    There is a slight mismatch between the properties of edge and vertex groups of $T$ in \Cref{thm:JSJ+}: we allow cyclic edge groups that are not centralisers, but we claim that non-QH vertex groups are rigid only over $\mc{Z}(G)$. This is the best one can do if we want rigid vertex groups to be \emph{convex-cocompact}, which in turn is essential to carry out many inductive arguments in the paper. On the other hand, we emphasise that type~(a) vertex groups in \Cref{thm:JSJ+} are not convex-cocompact in general (\Cref{rmk:non_cc_QH}).
\end{rmk} 

The tree $T$ in \Cref{thm:JSJ+} is allowed to be a single vertex. When this occurs for $\mc{H}=\mc{S}(G)$ (and if $G$ is not a surface group), we gain the important information that $G$ is $(\mc{Z}(G),\mc{S}(G))$--rigid. As the next lemma shows, rigidity implies that every element of $\Out(G)$ has its top growth rate realised on a singular subgroup of $G$, allowing us to reduce the study of growth to lower-complexity subgroups of $G$. We will return to this idea in \Cref{sub:osing}.

\begin{lem}\label{lem:rigid_bounds_otop}
    Suppose that $G$ is $(\mc{Z}(G),\mc{H})$--rigid, for a collection of subgroups $\mc{H}\supseteq\mc{S}(G)$. 
    Consider an outer automorphism $\phi\in\Out(G)$ and suppose that there exists an abstract growth rate $\mf{o}\in\mf{G}$ such that all subgroups in $\mc{H}$ are $\mf{o}$--controlled (\Cref{sub:beat}). Then, we have $\overline{\mf{o}}_{\rm top}(\phi)\preceq\mf{o}$.
\end{lem}
\begin{proof}
Suppose for the sake of contradiction that $\overline{\mf{o}}_{\rm top}(\phi)\not\preceq\mf{o}$. This implies that there exists a non-principal ultrafilter $\om$ such that $\mf{o}^{\om}_{\rm top}(\phi)\succ_{\om}\mf{o}$ (\Cref{rmk:fail->fail_mod_omega}). 
Realising $G$ as a convex-cocompact subgroup of a RAAG $A_{\G}$, we can thus consider the degeneration $G\acts\X_{\om}$ associated with $\phi$ and $\om$ as in \Cref{sub:degenerations}. Choose $v\in\G$ so that the $\R$--tree $G\acts T^v_{\om}$ is non-elliptic.

By \Cref{lem:beat_vs_degeneration}, the subgroups in $\mc{H}$ are elliptic in $T^v_{\om}$ and, in particular, the elements of $\mc{S}(G)$ are elliptic. All non-cyclic elements of $\mc{Z}(G)$ are contained in singular subgroups, and so they are also elliptic in $T^v_{\om}$. As a consequence, \Cref{thm:10e-}(2b) implies that the $G$--stabiliser of every arc of $\Min(G,T^v_{\om})$ lies in $\mc{Z}(G)$. We can thus appeal to \Cref{thm:acc_implies_nice}(3) to extract from $T^v_{\om}$ a $(\mc{Z}(G),\mc{H})$--splitting of $G$, contradicting rigidity of $G$.
\end{proof}

\subsection{JSJs of RAAGs}\label{sub:JSJ_RAAG}

For a $1$--ended RAAG $A_{\G}$, it should be possible to exhibit an $\Out(A_{\G})$--invariant $(\mc{ZZ}(A_{\G}),\mc{S}(A_{\G}))$ splitting as in \Cref{thm:JSJ+} directly from the presentation graph $\G$. It also seems reasonable that all edge groups should be parabolic and, hence, that there should be no quadratically hanging vertices. We would further expect the edge groups to always be abelian and equal to their centraliser in $A_{\G}$. 

None of this is shown in this article. Here we only give two simple examples of enhanced JSJ decompositions of RAAGs as a proof of concept.

\begin{ex}
    Let $\G$ be a finite tree of diameter $\geq 3$. Let $\G'\sq\G$ be the subgraph obtained by removing all degree--$1$ vertices of $\G$ and their incident edges. The group $A_{\G}$ admits an $\Out(A_{\G})$--invariant $(\mc{Z}(A_{\G}),\mc{S}(A_{\G}))$--splitting $A_{\G}\acts T$ such that the quotient $T/A_{\G}$ is naturally identified with $\G'$. The vertex group associated to $v\in\G'$ is $A_{\St(v)}$, computing stars within the entire tree $\G$. In particular, the vertex groups are singular subgroups of $A_{\G}$ and hence rigid in $A_{\G}$. Edge groups of $T$ are all isomorphic to $\Z^2$ and parabolic, corresponding to the edges of $\G'\sq\G$.
\end{ex}

\begin{ex}
    Let $\G$ be the graph pictured in \Cref{ex:fix_not_raag} below. With a little work, it is possible to use \cite{Hull} to check that $A_{\G}$ is $(\mc{Z}(A_{\G}),\mc{S}(A_{\G}))$--rigid, and so its enhanced JSJ is trivial. Nevertheless, the group $\Out(A_{\G})$ is infinite. \Cref{lem:rigid_bounds_otop} guarantees that all outer automorphisms of $A_{\G}$ achieve their top growth rate on a singular subgroup ($\cong F_2\x F_2$ or $\cong F_3\x\Z$).  
\end{ex}

\section{The complexity-reduction homomorphism}\label{sect:exact}

This section is devoted to the proof of \Cref{thmintro:exact}: for any virtually special group $U$, the outer automorphism group $\Out(U)$ is boundary amenable, satisfies the Tits alternative, and has finite virtual cohomological dimension (\Cref{cor:exact}). The main ingredient is a ``complexity-reduction'' homomorphism, whose existence can be quickly deduced from the enhanced JSJ decomposition constructed in \Cref{sect:JSJ}. More precisely, for $G$ special, there is a homomorphism from $\Out(G)$ to a finite product of groups $\Out(P_i)$, where $P_i\leq G$ are lower-complexity special groups (\Cref{prop:product_embedding}).

``Complexity'' is always meant in the following sense. Here and through most of the section, we work with an \emph{actual} special group $G$.

\begin{defn}\label{defn:ambient_rank}
    The \emph{ambient rank} $\ar(G)$ is the smallest integer $r$ such that $G$ is a convex-cocompact subgroup of a RAAG $A_{\G}$ where $\G$ has $r$ vertices.
\end{defn}

If $G$ does not virtually split as a direct product, then $\ar(S)<\ar(G)$ for all $S\in\mc{S}(G)$, by \Cref{prop:S(G)_properties}(2). We also need the following.

\begin{rmk}\label{rmk:restriction+}
    Let $G$ be special, let $H\leq G$ be convex-cocompact, and let $\mc{O}\leq\Out(G)$ preserve the $G$--conjugacy class of $H$. As in \Cref{rmk:restriction_new}, let $C^G_H\leq\Out(H)$ be the the image of the conjugacy action $N_G(H)\acts H$. Note that $C^G_H$ is finite because, by \Cref{lem:cc_basics}(2), $N_G(H)$ is virtually generated by $H$ and $Z_G(H)$. If $C^G_H=\{1\}$ (for instance, if $N_G(H)=H$), then each $\phi\in\mc{O}$ has a uniquely defined restriction $\phi|_H\in\Out(H)$. In this case, there is a well-defined restriction homomorphism $\mc{O}\ra\Out(H)$.
	
    In general, we can always find a finite-index subgroup $\mc{O}'\leq\mc{O}$ with a well-defined restriction homomorphism $\mc{O}'\ra\Out(H)$. Indeed, let $\tilde{\mc{O}}\leq\Aut(G)$ be the group of automorphisms $\varphi$ with $\varphi(H)=H$ and $[\varphi]\in\mc{O}$. Let $r\colon\tilde{\mc{O}}\ra\Out(H)$ denote the homomorphism $\varphi\mapsto[\varphi|_H]$. Since $H$ is special, $\Out(H)$ is residually finite by \cite{AMS}, and so there exists a finite-index subgroup $\Out'(H)\leq\Out(H)$ that intersects $C^G_H$ trivially. Consider $\tilde{\mc{O}}':=r^{-1}(\Out'(H))$, and let $\mc{O}'$ be its projection to a finite-index subgroup of $\mc{O}$. Any inner automorphism of $G$ that lies in $\tilde{\mc{O}}'$ must restrict to an inner automorphism of $H$, and thus the restriction of $r$ to $\tilde{\mc{O}}'$ descends to a well-defined homomorphism $\mc{O}'\ra\Out(H)$.
\end{rmk}

The next result gives the complexity-reduction homomorphism mentioned above. This is a fairly direct consequence of \Cref{thm:JSJ+}, using the work of Levitt on automorphisms preserving a graph of groups \cite{Levitt-GD}. 

\begin{prop}\label{prop:product_embedding}
    Let $G$ be special, $1$--ended, and not virtually a direct product. There exist a finite-index subgroup $\Out^1(G)\leq\Out(G)$ and finitely many convex-cocompact subgroups $H_1,\dots,H_k\leq G$ such that the following hold:
    \begin{enumerate}
        \item we have $N_G(H_i)=H_i$ and the $G$--conjugacy class of $H_i$ is $\Out^1(G)$--invariant;
        \item the restriction morphism $\rho\colon\Out^1(G)\ra\Out(H_1)\x\dots\x\Out(H_k)$ has kernel isomorphic to a special group;
        \item we have $\ar(H_i)<\ar(G)$ unless $H_i$ is a free or surface group.
    \end{enumerate}
\end{prop}
\begin{proof}
    Denote by $\mc{S}^*(G)$ the union of $\mc{S}(G)$ with the collection of cyclic subgroups of $G$ whose conjugacy class has finite $\Out(G)$--orbit. 
    
    \Cref{thm:JSJ+} provides an $\Out(G)$--invariant $(\mc{ZZ}(G),\mc{S}^*(G))$--tree $G\acts T$ such that each vertex-stabiliser $V$ is either QH relative to $\mc{S}^*(G)$, or convex-cocompact and $(\mc{Z}(V),\mc{S}^*(G)|_V)$--rigid in itself. Let $\Out^1(G)\leq\Out(G)$ be the finite-index subgroup that acts trivially on the quotient graph $T/G$. Let $V_1,\dots,V_s$ be representatives of the $G$--conjugacy classes of (maximal) vertex-stabilisers of $T$. By the fact that $T$ is relative to $\mc{S}(G)$, each normaliser $N_G(V_i)$ is elliptic in $T$ (recall \Cref{lem:cc_basics}(2)), and hence $N_G(V_i)=V_i$ for all $i$ (by maximality of $V_i$).
    Thus, by \Cref{rmk:restriction+}, there is no ambiguity in restricting outer automorphisms of $G$ to the $V_i$, and we obtain a diagonal restriction homomorphism $\rho\colon\Out^1(G)\ra \prod_i\Out(V_i)$. 
    
    Passing to a further finite-index subgroup $\Out^2(G)\leq\Out^1(G)$, we can assume that we have a restriction homomorphism $\Out^2(G)\ra\Out(E)$ for each edge group $E$, and that the image of this homomorphism has trivial intersection with the finite subgroup $C^G_E\leq\Out(E)$ given by the conjugacy action $N_G(E)\acts E$ (see again \Cref{rmk:restriction+}).

    The kernel of $\rho\colon\Out^2(G)\ra\prod_i\Out(V_i)$ can be described using \cite[Section~2]{Levitt-GD}. In Levitt's terminology, $\ker\rho$ is generated by \emph{bitwists} around the edges of $T/G$ \cite[Proposition~2.2]{Levitt-GD}. In fact, we can say more: for any $\phi\in\ker\rho$, the restriction of $\phi$ to each (maximal) vertex group of $T$ is inner and hence, for each edge group $E$ of $T$, the restriction $\phi|_E$ lies in $C^G_E$. By our choice of $\Out^2(G)$, the latter implies that $\phi|_E$ is trivial. Thus, any element of $\ker\rho$ restricts to the trivial outer automorphism on each edge group of $T$. A bitwist that is inner on the corresponding edge group represents the same outer class in $G$ as a product of two \emph{twists} (again in Levitt's terminology), so we conclude that $\ker\rho$ is generated by twists. Now, twists around distinct edges commute in $\Out(G)$, while twists around a given edge $e$ generate a subgroup of $\Out(G)$ isomorphic to a direct product of (virtual) direct factors of the centralisers of $G_e$ in the two vertex groups of $e$ (this uses \Cref{lem:cc_basics}(2)). From this, one can deduce that $\ker\rho$ is isomorphic to a finite direct product of cyclic edge groups and convex-cocompact subgroups of non-QH vertex groups of $T$ (we leave the details of this to the reader). This implies that $\ker\rho$ is special.

    Now, if $V_i$ is a QH vertex group of $T$, then $V_i$ is isomorphic to a free or surface group. However, when $V_i$ is a non-QH vertex group of $T$, we might still have $\ar(V_i)=\ar(G)$. In order to conclude the proof, we thus consider a vertex $x\in T$ whose stabiliser $V$ is not QH. Let $\mc{O}$ be the image of the restriction homomorphism $\Out^2(G)\ra\Out(V)$. We wish to construct an injective homomorphism from a finite-index subgroup of $\mc{O}$ to a finite product of outer automorphism groups of subgroups of $V$ with strictly lower ambient rank (this will prove the proposition). 

    Let $\mc{E}_x$ be the collection of $G$--stabilisers of edges of $T$ incident to the vertex $x$, and let ${\rm Cyc}_{\mc{O}}(V)$ be the collection of cyclic subgroups of $V$ whose $V$--conjugacy class has finite $\mc{O}$--orbit. Recall that $V$ is $(\mc{Z}(V),\mc{S}^*(G)|_V)$--rigid in itself. Since all elements of $\mc{S}^*(G)$ are elliptic in $T$, the maximal elements of $\mc{S}^*(G)|_V$ are either elements of $\mc{E}_x$, or elements of $\mc{S}^*(G)$ that happen to be contained in $V$. The latter lie in $\mc{S}(V)\cup{\rm Cyc}_{\mc{O}}(V)$. In conclusion, the group $V$ is $(\mc{Z}(V),\mc{S}(V)\cup\mc{E}_x\cup{\rm Cyc}_{\mc{O}}(V))$--rigid in itself.

    Let $\mc{O}_0\leq\mc{O}$ be a finite-index subgroup that preserves the $V$--conjugacy class of each subgroup in $\mc{S}(V)\cup\mc{E}_x$, and let $H_1,\dots,H_k$ be representatives of the finitely many $V$--conjugacy classes of non-cyclic, maximal subgroups in this collection. Each $H_i$ satisfies $N_V(H_i)=H_i$ (by maximality), and so we obtain a restriction homomorphism $\rho'\colon\mc{O}_0\ra\prod_i\Out(H_i)$. Each $H_i$ is a subgroup of a singular subgroup of $G$ (using \Cref{prop:S(G)_properties}(4) when $H_i\in\mc{E}_x$), and hence $\ar(H_i)<\ar(G)$ for all $i$. 
    
    Now, the kernel of $\rho'$ is finite: indeed, any infinite sequence in $\ker\rho'$ would give a degeneration $V\acts\X_{\om}$ in which all elements of the family $\mc{S}(V)\cup\mc{E}_x\cup{\rm Cyc}_{\mc{O}}(V)$ are elliptic, and we would be able to use \Cref{thm:acc_implies_nice}(3) to extract a $(\mc{Z}(V),\mc{S}(V)\cup\mc{E}_x\cup{\rm Cyc}_{\mc{O}}(V))$--splitting of $V$, violating rigidity (see the proof of \Cref{lem:rigid_bounds_otop} for details). Finally, residual finiteness of $\Out(V)$ yields a finite-index subgroup $\mc{O}_1\leq\mc{O}_0$ on which $\rho'$ is injective.
    
    Summing up, we have found a homomorphism $\rho\colon\Out^2(G)\ra\prod_i\Out(V_i)$ with special kernel and, for each $V_i$ such that $\ar(V_i)=\ar(G)$ and $V_i$ is neither free nor surface, we have constructed a further injective homomorphism $\rho_i'\colon\mc{O}_1^i\ra\prod_j\Out(H_j)$ where $\ar(H_j)<\ar(G)$ and $\mc{O}_1^i$ has finite index in the image of $\Out^2(G)\ra\Out(V_i)$. Composing $\rho$ with the product of the $\rho_i'$ (and some identity homomorphisms), this proves the proposition.
\end{proof}

It is important to stress that \Cref{prop:product_embedding} gives no information whatsoever on the \emph{image} of the homomorphism $\rho\colon\Out^1(G)\ra\prod_i\Out(H_i)$. Better understanding this image would be a reasonable approach to proving finite generation of $\Out(G)$, though this would require further ideas.

Before we continue with the proof of \Cref{thmintro:exact}, we briefly introduce the three properties in its statement. Let $U$ be a countable, discrete group.

\emph{Boundary amenability} is also known as coarse amenability, exactness, or property A \cite{Yu,Higson-Roe}; see \cite{Ananth,Ozawa} for background. If $U$ is boundary amenable, then it satisfies Novikov's conjecture on higher signatures \cite{Higson,Baum-Connes-Higson}. If in addition $U$ is finitely generated, then $U$ satisfies the coarse Baum--Connes and strong Novikov conjectures. It remains unknown whether $\Out(G)$ is finitely generated for all special groups $G$.

A group $U$ satisfies the \emph{Tits alternative} if, for each subgroup $H\leq U$ (possibly not finitely generated), 
either $H$ is virtually solvable or $H$ contains a subgroup isomorphic to the free group $F_2$. We will work with a slightly stronger version that we refer to as the \emph{Tits$^*$ alternative}: either $H$ is virtually polycylic or $H$ contains $F_2$.

The \emph{cohomological dimension} ${\rm cd}(U)$ is the supremum of degrees in which $U$ has nontrivial group cohomology, over all coefficient modules. If $U$ has torsion, then ${\rm cd}(U)$ is infinite. If $U$ is torsion-free, we have ${\rm cd}(U)={\rm cd}(U')$ for all finite-index subgroups $U'\leq U$. Thus, if $U$ is virtually torsion-free, one can define the \emph{virtual cohomological dimension} (\emph{vcd} for short) ${\rm vcd}(U)$ as the cohomological dimension of any finite-index torsion-free subgroup of $U$.

\begin{rmk}\label{rmk:stability_properties}
    Let $U$ be a countable, discrete group. Let $\mscr{P}$ mean any of the following four properties: boundary amenability, the Tits$^*$ alternative, virtual torsion-freeness, finiteness of vcd. The property $\mscr{P}$ is stable under the following moves. (See e.g.\ \cite[Section~2.4]{BGH22} for boundary amenability and \cite{Brown} for vcd, 
    when we do not give other references.)
    \begin{enumerate}
        \item \emph{Arbitrary subgroups.} If $U$ satisfies $\mscr{P}$ and $V\leq U$ is any subgroup, then $V$ satisfies $\mscr{P}$.
        \item \emph{Finite-index overgroups.} If $U\leq V$ has finite-index and $U$ satisfies $\mscr{P}$, then $V$ satisfies $\mscr{P}$.
        \item \emph{Finite direct products.} If $U_1,\dots,U_k$ satisfy $\mscr{P}$, then the product $U_1\x\dots\x U_k$ satisfies $\mscr{P}$.
        \item \emph{Extensions.} Consider a short exact sequence $1\ra N\ra U\ra Q\ra 1$. If the groups $N$ and $Q$ are boundary amenable, then so is $U$ \cite[Theorem~5.1]{Kirchberg-Wassermann}. If $N$ and $Q$ satisfy the Tits$^*$ (resp.\ Tits) alternative, then so does $U$; indeed, if $N$ and $Q$ are virtually polycyclic (resp.\ virtually solvable), then so is $U$ (see e.g.\ \cite[Lemma~5.5]{Dinh}). 
        If $N$ and $Q$ have finite vcd and $U$ is virtually torsion-free, then ${\rm vcd}(U)\leq{\rm vcd}(N)+{\rm vcd}(Q)<+\infty$. Finally, if $N$ is torsion-free and $Q$ is virtually torsion-free, then $U$ is virtually torsion-free.
        \item \emph{Quotients.} Suppose that $N\lhd U$. If $N$ is amenable and $U$ is boundary amenable, then $U/N$ is boundary amenable. If $N$ is finite and $U$ satisfies the Tits$^*$ alternative, then $U/N$ satisfies the Tits$^*$ alternative (finite-by-polycyclic groups are polycyclic-by-finite).
    \end{enumerate}
\end{rmk}

The next lemma is needed to extend results on $\Out(G)$, with $G$ special, to results on $\Out(U)$, with $U$ \emph{virtually} special. The proof is similar to that of \cite[Section~2]{Krstic} and \cite[Lemma~5.4]{GL-outerspace}.

\begin{lem}\label{lem:Out(fi)_new}
    Let $U$ be finitely generated and let $L\leq U$ be finite-index.
    \begin{enumerate}
        \item If $\Out(L)$ is boundary amenable, then $\Out(U)$ is boundary amenable.
        \item If $\Out(L)$ satisfies the Tits alternative, then so does $\Out(U)$. If $L$ is normal, has finitely generated centre, and $\Out(L)$ satisfies the Tits$^*$ alternative, then $\Out(U)$ satisfies the Tits$^*$ alternative. 
        \item If $U$ and $\Out(L)$ are virtually torsion-free and have finite ${\rm vcd}$, and if $\Out(U)$ is virtually torsion-free, then $\Out(U)$ has finite ${\rm vcd}$.
        \item If $L$ is normal, the centre of $L$ is finitely generated, $\Out(L)$ is virtually torsion-free, and $\Out(U)$ is residually finite, then $\Out(U)$ is virtually torsion-free.
    \end{enumerate}
\end{lem}
\begin{proof}
	For convenience, let $\mscr{P}$ denote any of the properties under consideration: boundary amenability, the Tits$^*$ (or Tits) alternative,  finiteness of vcd, virtual torsion-freeness. Let $K_L\leq\Aut(U)$ be the subgroup of automorphisms that fix $L$ pointwise. 

    \smallskip
    {\bf Claim.} \emph{The group $K_L$ virtually embeds in a finite product $C\x\dots\x C$, where $C$ is the centre of a finite-index subgroup of $L$.}

    \smallskip\noindent
    \emph{Proof of claim.}
    Let $L_0\leq L$ be a finite-index subgroup that is normal in $U$. 
    Let $K_L^0\leq K_L$ be a finite-index subgroup acting trivially on the coset space $U/L_0$. Write $U=u_1L_0\sqcup\dots\sqcup u_kL_0$, choosing some coset representatives $u_i\in U$. For any $\varphi\in K_L^0$ and $1\leq i\leq k$, we have $\varphi(u_iL_0)=(u_iL_0)$ and so there are elements $\ell_i(\varphi)\in L_0$ such that $\varphi(u_i)=u_i\ell_i(\varphi)$. Note that, for $\varphi,\psi\in K_L^0$, we have $\varphi\psi(u_i)=\varphi(u_i)\ell_i(\psi)=u_i\ell_i(\varphi)\ell_i(\psi)$ and thus each map $\ell_i\colon K_L^0\ra L_0$ is a homomorphism. It follows that the tuple $(\ell_1,\dots,\ell_k)$ gives a group embedding $K_L^0\hookrightarrow L_0^k$. Now, for any $x\in L_0$, we have $u_ixu_i^{-1}\in L_0$ because $L_0$ is normal in $U$. Hence, for any $\varphi\in K_L^0$, we have
    \[ u_ixu_i^{-1}=\varphi(u_ixu_i^{-1})=\varphi(u_i)x\varphi(u_i)^{-1}=u_i\ell_i(\varphi)x\ell_i(\varphi)^{-1}u_i^{-1} ,\]
    which yields $x=\ell_i(\varphi)x\ell_i(\varphi)^{-1}$ for all $x\in L_0$. This shows that each homomorphism $\ell_i$ takes values in the centre of $L_0$, proving the claim.
    \hfill$\blacksquare$

    \smallskip
    Let $\Aut_L(U)\leq\Aut(U)$ be the subgroup of automorphisms leaving $L$ invariant. Since $U$ is finitely generated, it has only finitely many subgroups of the same index as $L$, and hence $\Aut_L(U)$ has finite index in $\Aut(U)$. Consider the homomorphisms
    \begin{align*}
        &\alpha\colon\Aut_L(U)\ra\Out(U), & &\beta\colon\Aut_L(U)\ra\Out(L),
    \end{align*}
    defined as follows: $\alpha$ is the composition of the inclusion $\Aut_L(U)\hookrightarrow\Aut(U)$ with the quotient projection to $\Out(U)$, while $\beta$ is the composition of the restriction $\Aut_L(U)\ra\Aut(L)$ with the quotient projection to $\Out(L)$. 
    
    The image ${\rm im}(\alpha)$ has finite index in $\Out(U)$, while the image ${\rm im}(\beta)$ is just some subgroup of $\Out(L)$. By \Cref{rmk:stability_properties}(1) we know that ${\rm im}(\beta)$ satisfies property $\mscr{P}$, and by \Cref{rmk:stability_properties}(2) it suffices to prove that ${\rm im}(\alpha)$ satisfies $\mscr{P}$ in order to obtain it for $\Out(U)$ (in each of the four parts of the lemma). Note that $\ker(\alpha)\leq \Aut(U)$ is the group of conjugations by elements of the normaliser $N_U(L)$. As such, it has a finite-index subgroup $\ker^0(\alpha)$ formed by conjugations by elements of $L$, and we have $\ker^0(\alpha)\lhd\Aut_L(U)$. The group $\ker(\beta)$ is the group of automorphisms of $U$ restricting to an inner automorphism on $L$, so it splits as a product $\ker^0(\alpha)\cdot K_L$.

    Now, we have an epimorphism
    \[ \eta\colon\Aut_L(U)/\ker^0(\alpha)\twoheadrightarrow\Aut_L(U)/\ker(\beta) \cong{\rm im}(\beta) \]
    with kernel $\ker(\beta)/\ker^0(\alpha)\cong K_L/(K_L\cap\ker^0(\alpha))$. By the claim, $K_L$ is virtually abelian, and hence $\ker(\eta)$ is virtually abelian as well. Moreover, ${\rm im}(\eta)$ satisfies property $\mscr{P}$, and the group $\Aut_L(U)/\ker^0(\alpha)$ maps onto ${\rm im}(\alpha)\cong\Aut_L(U)/\ker(\alpha)$ with finite kernel $\ker(\alpha)/\ker^0(\alpha)$. 
    
    Using parts~(4) and~(5) of \Cref{rmk:stability_properties}, these observations imply that ${\rm im}(\alpha)$ is boundary amenable and satisfies the Tits alternative, provided that the respective property holds for ${\rm im}(\beta)$. As discussed above, this proves part~(1) and the first half of part~(2) of the lemma.

    Let $\Aut_L^0(U)\leq\Aut_L(U)$ be the finite-index subgroup that acts trivially on the coset space $U/L$. If $L$ is normal in $U$, we can take $L_0=L$ in the proof of the claim, which then shows that $K_L\cap\Aut_L^0(U)$ embeds in a product of finitely many copies of the centre of $L$. In particular $K_L\cap\Aut_L^0(U)$ is abelian, and it is finitely generated if the centre of $L$ is finitely generated. Thus, under the assumptions of the second half of part~(2), $K_L$ and $\ker(\eta)$ are virtually polycyclic, 
    and \Cref{rmk:stability_properties} implies that $\Out(U)$ satisfies the Tits$^*$ alternative, completing the proof of part~(2).
    
    We now prove part~(3). Let $\mc{U}\leq\Out(U)$ and $\mc{L}\leq\Out(L)$ be finite-index torsion-free subgroups, and let $\mc{I}$ be the projection of the intersection $\alpha^{-1}(\mc{U})\cap\beta^{-1}(\mc{L})$ to  $\Aut_L(U)/\ker^0(\alpha)$. Thus, $\mc{I}$ has finite index in $\Aut_L(U)/\ker^0(\alpha)$, and the image of the restricted homomorphism $\eta|_{\mc{I}}$ is torsion-free with finite cohomological dimension, as it is a subgroup of $\mc{L}$. 
    
    Let $F\lhd\mc{I}$ be the finite subgroup obtained by intersecting $\ker(\alpha)/\ker^0(\alpha)$ with $\mc{I}$. The quotient $\mc{I}/F$ is a finite-index subgroup of $\Aut_L(U)/\ker(\alpha)\cong{\rm im}(\alpha)$, and it is identified with a finite-index subgroup of $\mc{U}\leq\Out(U)$, so it is torsion-free. Since there is no torsion in the image of $\eta|_{\mc{I}}$, we have $F\lhd\ker(\eta|_{\mc{I}})$ and $\eta$ yields a short exact sequence
    \[ 1\ra \ker(\eta|_{\mc{I}})/F \ra \mc{I}/F\ra {\rm im}(\eta|_{\mc{I}}) \ra 1 .\]
    The kernel of $\eta|_{\mc{I}}$ has finite index in $\ker(\eta)\cong K_L/(K_L\cap\ker^0(\alpha))$, and so it is virtually abelian with finite vcd by the claim (since $U$ has finite vcd). Any finite-index torsion-free subgroup of $\ker(\eta|_{\mc{I}})$ projects injectively to a finite-index subgroup of $\ker(\eta|_{\mc{I}})/F$, and thus the latter has finite vcd as well. Summing up, $\mc{I}/F$ is torsion-free, and both image and kernel of $\eta$ restricted to $\mc{I}/F$ have finite vcd. \Cref{rmk:stability_properties}(4) now shows that $\mc{I}/F$ has finite vcd. Since the latter has finite index in $\Out(U)$, this proves part~(3).
    
    As to part~(4), let $\mc{U}'$ be the finite-index subgroup of $\Aut_L(U)/\ker^0(\alpha)$ that is the projection of $\Aut_L^0(U)$. We have already observed that the hypotheses of part~(4) imply that $K_L\cap\Aut_L^0(U)$ is abelian 
    and finitely generated. Let again $\mc{L}\leq\Out(L)$ be a finite-index torsion-free subgroup, and consider $\mc{L}':=\beta^{-1}(\mc{L})/\ker^0(\alpha)$ and $\mc{I}':=\mc{U}'\cap\mc{L}'$, which have finite index in $\Aut_L(U)/\ker^0(\alpha)$. The restriction of $\eta$ to $\mc{I}'$ gives an exact sequence
    \[ 1\ra \ker(\eta|_{\mc{I}'}) \ra\mc{I}' \ra \mc{L} ,\]
    where as above $\ker(\eta|_{\mc{I}'})$ is a quotient of $K_L\cap \Aut_L^0(U)$. 
    In particular, this kernel is abelian and finitely generated. Let $F'\lhd\mc{I}'$ be the intersection of $\mc{I}'$ with the finite group $\ker(\alpha)/\ker^0(\alpha)$. As above, the fact that $\mc{L}$ is torsion-free implies that $F'$ is contained in the kernel of $\eta|_{\mc{I}'}$ and we obtain
    \[ 1\ra \ker(\eta|_{\mc{I}'})/F' \ra\mc{I}'/F' \ra \mc{L} .\]
    Now, $\mc{I}'/F'$ has finite index in $\Out(U)$, and so it is residually finite by hypothesis. Since $\ker(\eta|_{\mc{I}'})/F'$ is abelian and finitely generated, it has only finitely many torsion elements. A finite-index subgroup of $\mc{I}'/F'$ avoids all of the latter, and so it is torsion-free by \Cref{rmk:stability_properties}(4). This is the required finite-index torsion-free subgroup of $\Out(U)$, proving part~(4).
\end{proof}

We are finally ready to prove \Cref{thmintro:exact}, which we restate here.

\begin{cor}\label{cor:exact}
    The following hold for every virtually special group $U$:
    \begin{enumerate}
        \item $\Out(U)$ is boundary amenable;
        \item $\Out(U)$ is virtually torsion-free and it has finite vcd;
        \item $\Out(U)$ satisfies the Tits$^*$ alternative;
        \item any Baumslag--Solitar subgroup ${\rm BS}(m,n)\leq\Out(U)$ has $|m|=|n|$.
    \end{enumerate}
\end{cor}
\begin{proof}
    Part~(4) is immediate from part~(3): if ${\rm BS}(m,n)$ is contained in $\Out(G)$ for $0<m<|n|$, then ${\rm BS}(m,n)$ is residually finite by \cite{AMS}, and so we have $m=1$. At the same time, ${\rm BS}(1,n)$ is solvable and not virtually polycyclic for $|n|>1$, violating the Tits$^*$ alternative.
    
    Now, every virtually special group $U$ has a finite-index, normal, special subgroup $G$. Special groups have finite cohomological dimension and their abelian subgroups have bounded rank. Moreover, $\Out(U)$ is residually finite by \cite{AMS}. Thus, using \Cref{lem:Out(fi)_new}, it suffices to prove each of parts~(1)--(3) of the corollary for a finite-index special subgroup. 

    In the rest of the proof we therefore assume that $U=G$ is special. We argue by induction on the ambient rank of $G$, with the base case being trivial. For the inductive step, suppose that, for every special group $H$ with $\ar(H)<\ar(G)$, we know that $\Out(H)$ satisfies the corollary.

    We can assume that $G$ is $1$--ended. Indeed, in general $G$ admits a free splitting $G=G_1\ast\dots\ast G_k\ast F_m$, where the $G_i$ are $1$--ended. The factors $G_i$ are special by \Cref{lem:cc_edges}(1) and thus they are boundary amenable \cite[Theorem~13]{Campbell-Niblo}, satisfy the Tits$^*$ alternative \cite{Sageev-Wise}, and are torsion-free with finite vcd. If $Z_i$ denotes the centre of $G_i$, then $G_i/Z_i$ is virtually special by \Cref{lem:cc_basics}(2), and so it has finite vcd as well. Now, if the $\Out(G_i)$ are boundary amenable, then so is $\Out(G)$ by \cite[Corollary~5.3]{BGH22}; if the $\Out(G_i)$ satisfy the Tits$^*$ alternative, then so does $\Out(G)$ by \cite[Theorem~6.1]{Horbez14}; and if the $\Out(G_i)$ are virtually torsion-free and have finite vcd, then so does $\Out(G)$ by \cite[p.\,709, Theorem~5.2(i)]{GL-outerspace}.

    Now, suppose that $G\in\mc{S}(G)$. Let $L\leq G$ be a finite-index subgroup 
    of the form $L_1\x\dots\x L_k\x\Z^m$, where each $L_i$ is directly indecomposable and has trivial centre. By \Cref{lem:Out(fi)_new}, it again suffices to prove the corollary for $\Out(L)$. Note that $\ar(L_i)<\ar(G)$ for all $i$. Automorphisms of $L$ permute the subgroups $\langle L_i,\Z^m\rangle$, e.g.\ by \cite[Lemma~3.5]{Fio11a}. Let $\Out^*(L)$ be the group of outer classes $[\varphi]$ for $\varphi\in\Aut(L)$ with $\varphi|_{\Z^m}=\id$ and $\varphi(L_i)=L_i$ for all $i$. Denoting by $L_{\rm ab}$ the abelianisation of $L$, the natural homomorphism $\eta\colon\Out(L)\ra\Out(L_{\rm ab})$ has $\ker(\eta)$ virtually contained in $\Out^*(L)$. 
    
    The group $\Out(L_{\rm ab})$ is isomorphic to ${\rm GL}_p(\Z)$ for some $p\in\N$, which is boundary amenable by \cite{GHW}, satisfies the Tits$^*$ alternative by \cite{Tits,Malcev}, 
    is virtually torsion-free by Selberg's lemma, and has finite vcd because it acts properly on the symmetric space of ${\rm SL}_p(\R)$. On the other hand, we have $\Out^*(L)\cong\prod_i\Out(L_i)$, which satisfies all parts of the corollary by the inductive hypothesis; by \Cref{rmk:stability_properties}(1), so does $\ker(\eta)$. Using \Cref{rmk:stability_properties}(4), this implies that $\Out(L)$ is boundary amenable and satisfies the Tits$^*$ alternative, as this is true of both ${\rm im}(\eta)$ and $\ker(\eta)$. The same shows that $\Out(L)$ has finite vcd, provided that we find a finite-index subgroup of $\Out(L)$ that has torsion-free intersection with $\Out^*(L)$; the existence of the latter follows from the description of automorphisms of products in terms of formal triangular matrices, as in \cite[Section~3.1]{Fio11a}, using that $\Out^*(L)$ is virtually torsion-free by the inductive hypothesis.
    Summing up, this proves the entire corollary when $G$ virtually splits as a direct product.
    
    In conclusion, we can assume that $G$ is $1$--ended and $G\not\in\mc{S}(G)$. We can then invoke \Cref{prop:product_embedding}, which yields a short exact sequence
    \[ 1\longrightarrow K\longrightarrow\Out^1(G)\longrightarrow \Out(H_1)\x\dots\x\Out(H_k) ,\]
    where $\Out^1(G)\leq\Out(G)$ has finite index, $K$ is special (in particular, torsion-free), and we have $\ar(H_i)<\ar(G)$ for each $H_i$ that is not a free or surface group. Outer automorphism groups of surface and free groups are boundary amenable \cite{Ham-exact,BGH22}, satisfy the Tits$^*$ alternative \cite{McCarthy,BFH2}, and have finite vcd \cite{CV86}. Using the inductive assumption, it follows that the product $\prod_i\Out(H_i)$ has all these properties, and so does the group $K$ by specialness. One last application of \Cref{rmk:stability_properties} finally implies that $\Out(G)$ satisfies these properties as well, concluding the proof.
\end{proof}

\section{Growth of general automorphisms}\label{sect:tameness}

In this section we prove \Cref{thmintro:general_aut}, describing the top growth rate of a general automorphism of a (virtually) special group $G$. In fact, we will prove a stronger, more technical statement (\Cref{thm:tame}). 

We begin by introducing an important property of automorphisms, \emph{docility}, which plays a fundamental role in proofs. Most of \Cref{sect:tameness} is devoted to showing that exponentially-growing automorphisms of special groups are docile (\Cref{sub:main_step}), after which \Cref{thmintro:general_aut} can be deduced relatively quickly from the existence of the enhanced JSJ tree. For \emph{virtually} special groups, \Cref{thmintro:general_aut} follows by combining \Cref{thm:tame} with \Cref{lem:fi_fix} below.

\subsection{Soundness and docility}\label{sub:sound_docile}

We use the notation from \Cref{sub:growth}.

\begin{defn}\label{defn:tame}
    An abstract growth rate $[x_n]\in\mf{G}$ is: 
    \begin{enumerate}
        \item \emph{pure} if $[x_n]\sim [n^p\lambda^n]$ for some $\lambda>1$ and $p\in\N$;
        \item \emph{$(\lambda,p)$--tame}, for some $\lambda>1$ and $p\in\N$, if we have $[x_n]\sim[a_n\lambda^n]$ for a weakly increasing sequence $a_n$ with $[1]\preceq [a_n]\preceq [n^p]$.
    \end{enumerate}
\end{defn}

Pure growth rates are tame, but the converse does not hold. Pure growth rates form a \emph{totally} ordered subset of $(\mf{G},\preceq)$, while tame ones do not. Also note that tame rates are at least exponential, as we ask that $\lambda>1$.

Here in \Cref{sect:tameness}, we are mostly interested in tame growth rates: we will show that the top growth rate of any automorphism of a special group is either sub-polynomial or tame. We do not know if this top growth rate is pure in general, so pure rates will only return in \Cref{sect:cmp}, where we study coarse-median preserving automorphisms and prove \Cref{thmintro:cmp_aut}.

Recall that, for a finitely generated group $G$ and any $\varphi\in\Aut(G)$ with outer class $\phi\in\Out(G)$, we always have $\overline{\mc{O}}_{\rm top}(\varphi)\succeq\overline{\mf{o}}_{\rm top}(\phi)$.

\begin{defn}\label{defn:tame_aut}
    Consider $\varphi\in\Aut(G)$ with outer class $\phi\in\Out(G)$. 
    \begin{enumerate}
    \item The automorphism $\varphi$ is \emph{sound} if we have $\overline{\mc{O}}_{\rm top}(\varphi)\sim\overline{\mf{o}}_{\rm top}(\phi)$. Similarly, $\phi$ is \emph{sound} if all its representatives $\varphi'\in\Aut(G)$ are sound.
    \item The automorphism $\varphi$ is \emph{docile} if, at the same time, $\varphi$ is sound and $\overline{\mc{O}}_{\rm top}(\varphi)$ is tame. The outer automorphism $\phi$ is \emph{docile} if $\varphi$ is docile (this is independent of the representative $\varphi$, by \Cref{rmk:tameness_independent_of_representative} below). We also say that $\varphi$ and $\phi$ are $(\lambda,p)$--docile if we wish to specify the parameters for which $\overline{\mc{O}}_{\rm top}(\varphi)$ and $\overline{\mf{o}}_{\rm top}(\phi)$ are $(\lambda,p)$--tame.
    \end{enumerate}
\end{defn}

Thus, automorphisms are sound when, under their iterates, word length does not grow much faster than conjugacy length. Exponentially-growing automorphisms of free and surface groups are all sound, while inner automorphisms clearly are not. 

The crux of \Cref{sect:tameness} lies in showing that exponentially-growing automorphisms of special groups are sound (and docile). As train-track techniques are not available in this context, settling this harmless-looking point will require some work. In turn, soundness and docility are extremely important, because they are necessary to relate the growth of an automorphism of $G$ to its growth on the vertex groups of the enhanced JSJ decomposition of $G$.

\begin{rmk}\label{rmk:tame_implies_sum-stable}
    If $[x_n]\in\mf{G}$ is tame, then we have $\big[\sum_{i\leq n}x_i\big]\sim[x_n]$. Indeed, since $[x_n]\sim[a_n\lambda^n]$ for a weakly increasing sequence $a_n$ by definition, we obtain $\sum_{i\leq n}a_i\lambda^i\leq a_n\sum_{i\leq n}\lambda^i\preceq a_n\lambda^n$.
\end{rmk}

The need for \Cref{rmk:tame_implies_sum-stable} is precisely what motivated us to require the sequence $a_n$ in \Cref{defn:tame} to be weakly increasing (and that $\lambda\neq 1$). We will use the previous remark in many forms, the first of which is the following. Given $\varphi\in\Aut(G)$ and an element $g\in G$, we can define the elements $g_n:=g\varphi(g)\varphi^2(g)\dots\varphi^n(g)$. If the growth rate $\big[\,|\varphi^n(g)|\,\big]$ is tame, then $|g_n|\preceq |\varphi^n(g)|$. 
This leads to the next observation:

\begin{rmk}\label{rmk:tameness_independent_of_representative}
    If $\phi\in\Out(G)$ has a docile representative $\varphi\in\Aut(G)$, then \emph{all} its representatives are docile. Indeed, suppose that $\varphi$ is docile and let $\psi(x)=g\varphi(x)g^{-1}$ be another representative. On the one hand, we immediately get $\overline{\mc{O}}_{\rm top}(\psi)\succeq\overline{\mf{o}}_{\rm top}(\phi)\sim \overline{\mc{O}}_{\rm top}(\varphi)$. On the other, using the above elements $g_n$:
    \[ |\psi^n(x)|=|g_{n-1}\varphi^n(x)g_{n-1}^{-1}|\leq 2|g_{n-1}| + |\varphi^n(x)| ,\]
    for all $x\in G$ and $n\in\N$. Thus, \Cref{rmk:tame_implies_sum-stable} implies that $\overline{\mc{O}}_{\rm top}(\psi)\preceq\overline{\mc{O}}_{\rm top}(\varphi)$. In conclusion $\overline{\mc{O}}_{\rm top}(\psi)\sim\overline{\mc{O}}_{\rm top}(\varphi)$, showing that $\psi$ is docile.
\end{rmk}

Finally, a rate $[x_n]\in\mf{G}$ is \emph{sub-polynomial} if $[x_n]\preceq[n^p]$ for some $p\in\N$. An automorphism $\varphi\in\Aut(G)$ is \emph{sub-polynomial} if $\overline{\mc{O}}_{\rm top}(\varphi)$ is sub-polynomial.
A priori, the fact that an outer class $\phi\in\Out(G)$ has sub-polynomial $\overline{\mf{o}}_{\rm top}(\phi)$ does not imply that any automorphism representing $\phi$ is sub-polynomial. However, we have the following analogue of \Cref{rmk:tameness_independent_of_representative}.

\begin{rmk}\label{rmk:sub-polynomial_independent_of_representative}
    If $\phi\in\Out(G)$ has a representative $\varphi$ with $\overline{\mc{O}}_{\rm top}(\varphi)\preceq n^p$, then all representatives $[\psi]=\phi$ satisfy the weaker inequality $\overline{\mc{O}}_{\rm top}(\psi)\preceq n^{p+1}$. This is shown as in the previous remark: if $\psi(x)=g\varphi(x)g^{-1}$, we have $|\psi^n(x)|\leq 2|g_{n-1}|+|\varphi^n(x)|$ and:
    \[ |g_n|\leq\sum_{i\leq n}|\varphi^i(x)|\preceq \sum_{i\leq n} i^p\leq n^{p+1}. \]
\end{rmk}

To extend results from special to \emph{virtually} special groups, we can use:

\begin{lem}\label{lem:fi_fix}
    Let $G$ be a finitely generated group with a finite-index characteristic subgroup $G_0\lhd G$. Let $\varphi\in\Aut(G)$ be an automorphism. 
    \begin{enumerate}
        \item If $\varphi|_{G_0}$ is docile, then $\overline{\mc{O}}_{\rm top}(\varphi)\sim\overline{\mc{O}}_{\rm top}(\varphi|_{G_0})$ and $\overline{\mf{o}}_{\rm top}(\varphi)\sim\overline{\mf{o}}_{\rm top}(\varphi|_{G_0})$.
        \item If $\overline{\mc{O}}_{\rm top}(\varphi|_{G_0})\preceq[n^p]$ for some $p\in\N$, then $\overline{\mc{O}}_{\rm top}(\varphi)\preceq[n^{p+1}]$.
    \end{enumerate}
\end{lem}
\begin{proof}
    It suffices to prove the lemma for a power of $\varphi$, since $\overline{\mc{O}}_{\rm top}(\varphi^k)$ and $\overline{\mf{o}}_{\rm top}(\varphi^k)$ completely determine $\overline{\mc{O}}_{\rm top}(\varphi)$ and $\overline{\mf{o}}_{\rm top}(\varphi)$, by \Cref{lem:length_generalities}(1). Up to raising $\varphi$ to a power, we can assume that $\varphi$ acts trivially on the quotient $G/G_0$. For all elements $g\in G$, we then have $\varphi(g)g^{-1}\in G_0$ and hence $|\varphi^n(g)|\sim |\varphi^n(g)g^{-1}|\preceq \sum_{k=0}^{n-1}|\varphi^k\big(\varphi(g)g^{-1}\big)|$. The right-hand side is $\preceq\overline{\mc{O}}_{\rm top}(\varphi|_{G_0})$ if the latter growth rate is tame, and it is $\preceq[n^{p+1}]$ if we have $\overline{\mc{O}}_{\rm top}(\varphi|_{G_0})\preceq[n^p]$. \Cref{rmk:dominating_otop} then implies that $\overline{\mc{O}}_{\rm top}(\varphi)\preceq\overline{\mc{O}}_{\rm top}(\varphi|_{G_0})$ in the former case, and $\overline{\mc{O}}_{\rm top}(\varphi)\preceq[n^{p+1}]$ in the latter. The inequality $\overline{\mc{O}}_{\rm top}(\varphi)\succeq\overline{\mc{O}}_{\rm top}(\varphi|_{G_0})$ is immediate, so this proves part~(2) and the first equivalence in part~(1). Finally, if $\varphi|_{G_0}$ is docile, we also have $\overline{\mc{O}}_{\rm top}(\varphi|_{G_0})\sim\overline{\mf{o}}_{\rm top}(\varphi|_{G_0})\preceq\overline{\mf{o}}_{\rm top}(\varphi)\preceq\overline{\mc{O}}_{\rm top}(\varphi)$, and so all these growth rates coincide.
\end{proof}

\subsection{Proof of soundness}\label{sub:main_step}

In this subsection, we isolate the most technical part of the proof of \Cref{thmintro:general_aut}, which mainly has to do with the proof of soundness. For this, we introduce the following auxiliary concept.

\begin{defn}
    Let $H$ be a special group. If $\psi\in\Aut(H)$ is an automorphism with outer class $[\psi]\in\Out(H)$, we say that $\psi$ is \emph{$\alpha$--good}, for some $\alpha\in\N$, if {\bf both} the following conditions hold:
    \begin{enumerate}
        \item[(a)] either $\overline{\mf{o}}_{\rm top}([\psi])$ is sub-polynomial or $\psi$ is docile;
        \item[(b)] if $\overline{\mf{o}}_{\rm top}([\psi])\preceq [n^p]$ for some $p\in\N$, then $\overline{\mc{O}}_{\rm top}(\psi)\preceq [n^{p+\alpha}]$.
    \end{enumerate}
    We say that the group $H$ is \emph{$\alpha$--good} if all automorphisms of $H$ are $\alpha$--good. Finally, $H$ is \emph{good} if there exists $\alpha$ such that $H$ is $\alpha$--good.
\end{defn}

Our goal for this subsection is then to prove the following proposition. 

\begin{prop}\label{prop:alpha_good}
    All special groups are good.
\end{prop}

Throughout the following discussion, we fix a convex-cocompact subgroup $G\leq A_{\G}$ with $|\G^{(0)}|={\rm ar}(G)$, and we consider an automorphism $\varphi\in\Aut(G)$ with outer class $\phi\in\Out(G)$. 

We will prove \Cref{prop:alpha_good} by induction on the ambient rank. As a base step, one can take the case when ${\rm ar}(G)=1$, which implies that $G$ is cyclic. However, we will use in the inductive step the fact that all free and surface groups are $1$--good, which requires train tracks and for which we do not obtain a new argument. In a sense, that is the true base step.

As an outline of the inductive step, the plan is to consider the enhanced JSJ tree of $G$ (\Cref{thm:JSJ+}) and show that its rigid vertex groups are good, from which one quickly deduces that $G$ is itself good. The core of the inductive step lies in showing that rigid special groups are either good or relatively hyperbolic (\Cref{lem:tame_auxiliary}), and then in proving goodness in the relatively hyperbolic case (\Cref{lem:relhyp_new}).

Recall that $\mc{P}(G)$ is the family of $G$--parabolic subgroups of $G$.

\begin{lem}\label{lem:tame_auxiliary}
    Let $G$ be $(\mc{Z}(G),\mc{H})$--rigid, for a $\varphi$--invariant collection $\mc{H}\supseteq\mc{S}(G)$ consisting of finitely many $G$--conjugacy classes of subgroups of $G$. 
    Suppose that all non-cyclic elements of $\mc{H}$ are convex-cocompact and that $G\not\in\mc{S}(G)$. Also suppose that there exists $\alpha_0\in\N$ such that all subgroups in $\mc{H}\cup\mc{P}(G)\setminus\{G\}$ are $\alpha_0$--good, and that the automorphism $\varphi$ is not $(\alpha_0+1)$--good. Then $G$ is hyperbolic relative to $\mc{S}(G)$, and each non-cyclic subgroup in $\mc{H}$ is contained in an element of $\mc{S}(G)$.
\end{lem}
\begin{proof}
    The proof of this lemma is long and intricate, so we subdivide it into five steps. We begin with some basic observations. First, the projection $\phi\in\Out(G)$ has infinite order; otherwise it would be $1$--good, as inner automorphisms have linear growth (cyclic subgroups of special groups are undistorted). Second, $G$ does not virtually split as a direct product (and so $G$ has trivial centre) because we are assuming that $G\not\in\mc{S}(G)$. Third, up to raising $\varphi$ to a power, which does not affect the lack of goodness, we can assume that $\varphi$ preserves each $G$--conjugacy class in $\mc{H}$.

    \smallskip
    {\bf Step~1.} \emph{The growth rate $\overline{\mf{o}}_{\rm top}(\phi)$ is either sub-polynomial or tame.}

    \smallskip\noindent
    Let $H_1,\dots,H_k$ be representatives of the finitely many conjugacy classes of subgroups in $\mc{H}$. 
    Consider the growth rate $\mf{s}:=\sum_{j=1}^k\overline{\mf{o}}_{\rm top}(\phi|_{H_j})$ and note that it is $\succeq\|\phi^n(h)\|$ for each element $h$ of each subgroup $H\in\mc{H}$. Since $G$ is $(\mc{Z}(G),\mc{H})$--rigid, \Cref{lem:rigid_bounds_otop} shows that $\overline{\mf{o}}_{\rm top}(\phi)\preceq\mf{s}$. The opposite inequality is also clear, so we obtain $\overline{\mf{o}}_{\rm top}(\phi)\sim\mf{s}$. Finally, each rate $\overline{\mf{o}}_{\rm top}(\phi|_{H_j})$ is either sub-polynomial or tame, by the hypothesis that the $H_j$ are good, and a sum of such growth rates is itself either sub-polynomial or tame. In conclusion, $\overline{\mf{o}}_{\rm top}(\phi)$ is either sub-polynomial or tame.
       
    \smallskip
    {\bf Step~2.} \emph{Each $\psi\in\Aut(G)$ with outer class $\phi$ falls into one of the following two cases. Case~(i): we have $\overline{\mc{O}}_{\rm top}(\psi)\succ\overline{\mf{o}}_{\rm top}(\phi)$ and $\overline{\mf{o}}_{\rm top}(\phi)$ is tame. Case~(ii): we have $\overline{\mf{o}}_{\rm top}(\phi)\preceq[n^p]$ for some $p\in\N$ and $\overline{\mc{O}}_{\rm top}(\psi)\not\preceq [n^{p+\alpha_0}]$.} 

    \smallskip\noindent
    By Step~1, $\overline{\mf{o}}_{\rm top}(\phi)$ is either tame or sub-polynomial. If it is tame, we cannot have $\overline{\mc{O}}_{\rm top}(\psi)\sim\overline{\mf{o}}_{\rm top}(\phi)$ as this would imply that $\varphi$ is docile (\Cref{rmk:tameness_independent_of_representative}); the latter would mean that $\varphi$ is $0$--good, against our hypotheses. If instead $\overline{\mf{o}}_{\rm top}(\phi)\preceq[n^p]$ for some $p\in\N$, we cannot have $\overline{\mc{O}}_{\rm top}(\psi)\preceq [n^{p+\alpha_0}]$ as this would imply that $\overline{\mc{O}}_{\rm top}(\varphi)\preceq [n^{p+\alpha_0+1}]$ by \Cref{rmk:sub-polynomial_independent_of_representative}, which would again violate the assumption that $\varphi$ is not $(\alpha_0+1)$--good.

    \smallskip
    {\bf Step~3.} \emph{For each $\psi\in\Aut(G)$ with outer class $\phi$, there is a $\psi$--invariant subgroup $G_0(\psi)\leq G$ that is contained in an element of $\mc{Z}(G)$ and that contains all $\alpha_0$--good convex-cocompact subgroups $K\leq G$ with $\psi(K)=K$.}
    
    \smallskip\noindent
    Choose a non-principal ultrafilter $\om$ as follows. In Case~(i) of Step~2, we pick $\om$ such that $\mc{O}_{\rm top}^{\om}(\psi)\succ_{\om}\mf{o}_{\rm top}^{\om}(\phi)$. In Case~(ii), we choose $\om$ so that $\mc{O}_{\rm top}^{\om}(\psi)\succ_{\om} [n^{p+\alpha_0}]$. Then, we define:
    \[ G_0(\psi):=\{g\in G \mid |\psi^n(g)|\prec_{\om} \mc{O}_{\rm top}^{\om}(\psi) \}. \]
    Since $|\psi^n(g_1g_2)|\preceq_{\om}\max_i|\psi^n(g_i)|$ for any $g_1,g_2\in G$, we see that $G_0(\psi)$ is a subgroup of $G$. Recalling that $\psi$ is bi-Lipschitz with respect to $|\cdot|$ (\Cref{lem:length_generalities}), we also see that $G_0(\psi)$ is $\psi$--invariant.
    
    By our choice of $\om$, the group $G_0(\psi)$ contains all $\psi$--invariant, $\alpha_0$--good, convex-cocompact subgroups $K\leq G$. Indeed, if $\overline{\mf{o}}_{\rm top}([\psi|_K])$ is tame, then 
    \[\overline{\mc{O}}_{\rm top}(\psi|_K)\sim\overline{\mf{o}}_{\rm top}([\psi|_K])\preceq\overline{\mf{o}}_{\rm top}(\phi)\prec_{\om}\mc{O}^{\om}_{\rm top}(\psi) .\] 
    If $\overline{\mf{o}}_{\rm top}([\psi|_K])$ is sub-polynomial, then $\overline{\mc{O}}_{\rm top}(\psi|_K)$ is also sub-polynomial. Thus, if $\overline{\mf{o}}_{\rm top}(\phi)$ is tame, we have again $\overline{\mc{O}}_{\rm top}(\psi|_K)\prec \overline{\mf{o}}_{\rm top}(\phi)\prec_{\om}\mc{O}^{\om}_{\rm top}(\psi)$. Finally, if $\overline{\mf{o}}_{\rm top}(\phi)\preceq[n^p]$ for some integer $p$, we obtain $\overline{\mf{o}}_{\rm top}([\psi|_K])\preceq[n^p]$ and hence $\overline{\mc{O}}_{\rm top}(\psi|_K)\preceq[n^{p+\alpha_0}]\prec_{\om}\mc{O}^{\om}_{\rm top}(\psi)$.

    In the rest of Step~3, we show that $G_0(\psi)$ is contained in a centraliser. Recall that we are thinking of $G$ as a convex-cocompact subgroup of $A_{\G}$. Consider the RAAG $A_{\G}\ast\Z$, denoting by $t$ a generator of the cyclic free factor. Set $\hat G:=\langle G,t\rangle\cong G\ast\Z$, which is a convex-cocompact subgroup of $A_{\G}\ast\Z$. Define $\hat\psi\in\Aut(\hat G)$ by $\hat\psi|_G:=\psi$ and $\hat\psi(t)=t$. Note that we have $\|gt\|=|g|+1$ for all $g\in G$ (extending a generating set of $G$ by adding $t$). From this, it is straightforward to deduce that $\mc{O}^{\om}_{\rm top}(\psi)\sim_{\om}\mf{o}^{\om}_{\rm top}(\hat\psi)$.

    Now, let $\hat G\acts\mc{Y}_{\om}$ be the degeneration determined by $[\hat\psi]\in\Out(\hat G)$ and the ultrafilter $\om$ chosen at the start of Step~3. As usual, $\mc{Y}_{\om}$ embeds equivariantly in a product of $\R$--trees $\hat G\acts\mscr{T}^v_{\om}$ for $v\in\hat\G:=\G\sqcup\{t\}$. Since we have $\mf{o}^{\om}_{\rm top}(\hat\psi)\sim_{\om}\mc{O}^{\om}_{\rm top}(\psi)\succ_{\om}\mf{o}^{\om}_{\rm top}(\phi)$, the subgroup $G\leq\hat G$ is elliptic in $\mc{Y}_{\om}$ (\Cref{lem:beat_vs_degeneration}). Similarly, the subgroup $G_0(\psi)\ast\langle t\rangle\leq\hat G$ is also elliptic in $\mc{Y}_{\om}$, since each of its elements $h$ satisfies 
    \[ \|\hat\psi^n(h)\|\leq|\hat\psi^n(h)|\prec_{\om}\mc{O}^{\om}_{\rm top}(\psi)\sim_{\om}\mf{o}^{\om}_{\rm top}(\hat\psi) .\]
    
    By construction, $\hat G$ is not elliptic in $\mc{Y}_{\om}$, and so there exists $w\in\hat\G$ such that $\hat G$ is not elliptic in $\mscr{T}^w_{\om}$. This implies that the fixed sets of $G$ and $G_0(\psi)\ast\langle t\rangle$ in $\mscr{T}^w_{\om}$ are disjoint. Since both sets are fixed by $G_0(\psi)$, the shortest arc $\beta\sq\mscr{T}^w_{\om}$ between them is fixed by $G_0(\psi)$. Applying \Cref{thm:10e-}(1), we see that the $\hat G$--stabiliser of $\beta$ is contained in some centraliser $\hat Z\in\mc{Z}(\hat G)$, and so we have $G_0(\psi)\leq\hat Z$. Assuming that $G_0(\psi)\neq\{1\}$, this implies that we in fact have $\hat Z\leq G$ and so $\hat Z\in\mc{Z}(G)$, completing Step~(3).

    \smallskip
    {\bf Step~4.} \emph{Every singular subgroup of $G$ contains all centralisers that it intersects nontrivially.}
    
    \smallskip\noindent
    Consider some $U\in\mc{S}(G)$. As assumed at the beginning of the proof, the conjugacy class of $U$ is $\varphi$--invariant, and so there exists $\psi\in\Aut(G)$ with $[\psi]=\phi$ and $\psi(U)=U$. Since $U\in\mc{H}$, our hypotheses imply that $U$ is $\alpha_0$--good. By Step~3, $U$ is then contained in a centraliser containing $G_0(\psi)$ and hence $U=G_0(\psi)$ (as singular subgroups are maximal elements of $\mc{VP}(G)$, and non-cyclic centralisers are contained in elements of $\mc{VP}(G)$).
    
    Now, consider the following family of subgroups of $U=G_0(\psi)$:
    \[ \mscr{Z}:=\{ U\cap Z_G(g) \mid g\in G\setminus U \} . \]
    The collection $\mscr{Z}$ is $\psi$--invariant and it has the following property.

    \smallskip
    {\bf Claim.} \emph{Maximal elements of $\mscr{Z}$ are $U$--parabolic.}
    
    \smallskip\noindent
    \emph{Proof of claim.}
    Consider a maximal element $M\in\mscr{Z}$, then choose an element $g\in G\setminus U$ with $M=U\cap Z_G(g)$. Up to replacing $g$, we can assume that $Z_G(g)$ is a maximal element of the family $\{Z_G(x)\mid x\in G\setminus U\}$; this does not alter the intersection with $U$ because $M$ is maximal in $\mscr{Z}$. 

    By \cite[Remark~3.17]{Fio10a}, there exists $m\geq 1$ such that we can write $g^m=h_1\cdot\ldots\cdot h_k$ for elements $h_i\in G$ such that $\langle h_i\rangle$ is convex-cocompact and $Z_G(g)\leq Z_G(h_i)$. At least one of the $h_i$ lies outside $U$ and so, replacing $g$ by this $h_i$, we can assume that $\langle g\rangle$ is convex-cocompact.

    Now, $Z_G(g)$ virtually splits as $\langle g\rangle\x P$ with $P\in\mc{P}(G)$ by \Cref{lem:cc_basics}(2). Since $U$ is root-closed, we have $U\cap\langle g\rangle=\{1\}$, and so the intersection $U\cap P$ has finite index in $M$ by \Cref{lem:cc_basics}(6). In fact, since $U\cap P$ is root-closed, we must have $M=U\cap P$, showing that $M$ is $U$--parabolic.
    \hfill$\blacksquare$

    \smallskip
    We now conclude Step~4. Suppose for the sake of contradiction that there exists some $g\in G\setminus U$ such that $M:=U\cap Z_G(g)\neq\{1\}$. Without loss of generality, $M$ is a maximal element of $\mscr{Z}$. 
    By the claim, $\mscr{Z}$ contains only finitely many $U$--conjugacy classes of maximal subgroups (\Cref{lem:parabolics_cofinite}). Thus, up to raising $\psi$ to a power, 
    we can assume that the $U$--conjugacy class of $M$ is $\psi$--invariant. Hence there exists an element $u\in U$ such that the automorphism $\psi'(x):=u\psi(x)u^{-1}$ leaves invariant $M$ (as well as $U$).

    The normaliser $N_G(M)$ is also $\psi'$--invariant. It is $G$--parabolic by the claim and \Cref{lem:parabolic_normaliser2}(2). By our hypotheses, the fact that $N_G(M)\in\mc{P}(G)$ implies that it is $\alpha_0$--good. Thus, Step~3 shows that some centraliser containing $G_0(\psi')$ also contains $N_G(M)$, as well as $U$. Since $U\in\mc{S}(G)$, we again have $U=G_0(\psi')\geq N_G(M)$. Since $g\in Z_G(M)\leq N_G(M)$ by construction, this violates the fact that $g\not\in U$, a contradiction.

    In conclusion, we have shown that $\{1\}$ is the only element of $\mscr{Z}$. That is, no element of $U\setminus\{1\}$ commutes with an element of $G\setminus U$. Now, if $U$ intersects nontrivially some $Z\in\mc{Z}(G)$, then we can pick an element $g\in G$ such that $Z\leq Z_G(g)$. A first application of the previous observation implies that $g\in U$, since $Z_G(g)\cap U\neq\{1\}$, and a second one shows that $U\geq Z_G(g)\geq Z$. Thus, $U$ contains all elements of $\mc{Z}(G)$ that it intersects nontrivially.
    
    \smallskip
    {\bf Step~5.} \emph{The group $G$ is hyperbolic relative to $\mc{S}(G)$.}

    \smallskip\noindent
    We wish to prove relative hyperbolicity by appealing to Genevois' criterion \cite[Theorem~1.3]{Gen-relhyp}. Singular subgroups are convex-cocompact and contain all non-cyclic abelian subgroups of $G$, so we only need to show that distinct singular subgroups have trivial intersection.

    Recall that every $U\in\mc{S}(G)$ virtually splits as a direct product. Consider a finite-index subgroup $U^0=U_1\x\dots\x U_k\x A$, where each $U_i$ has trivial centre and $A$ is abelian; the $U_i$ and $A$ can all be chosen to be convex-cocompact. If another element $V\in\mc{S}(G)$ intersects $U$ nontrivially, then it intersects $U^0$ nontrivially. Since the intersection $U^0\cap V$ is convex-cocompact, it contains a nontrivial element $u$ lying either in $A$ or in one of the $U_i$, by \Cref{lem:cc_basics}(6). Step~4 then implies that $Z_G(u)\leq U\cap V$. If $u\in A$, we obtain $U\leq Z_G(u)\leq V$. If instead $u\in U_i$, we get that $V$ contains all factors of $U^0$ other than $U_i$ and, repeating the same argument with $u$ replaced by an element of $V$ in a different factor of $U$, we obtain $U\leq V$. Finally, by maximality of singular subgroups, we have $U=V$, completing Step~5.

    \smallskip
    Summing up, $G$ is relatively hyperbolic by Step~5. All non-cyclic elements of $\mc{H}$ are contained in non-cyclic centralisers by Step~3 (since $\mc{H}$ consists of finitely many conjugacy classes, each preserved by $\varphi$),
    and hence they are contained in elements of $\mc{S}(G)$. This concludes the proof of the lemma.
\end{proof}

Using Guirardel and Levitt's work on automorphisms of relative hyperbolic groups \cite{GL-relhyp}, we can now deduce goodness from \Cref{lem:tame_auxiliary}.

\begin{lem}\label{lem:relhyp_new}
    Let $G$ be $1$--ended and hyperbolic relative to $\mc{S}(G)$. If $\alpha_0\geq 1$ is such that all subgroups in $\mc{S}(G)$ are $\alpha_0$--good, then $G$ is $(\alpha_0+2)$--good.
\end{lem}
\begin{proof}
    By \cite{GL-relhyp}, there exist a finite-index subgroup $\Out^1(G)\leq\Out(G)$ and an $\Out^1(G)$--invariant, minimal tree $G\acts T$ whose edge groups are either infinite cyclic, or finitely generated subgroups of elements of $\mc{S}(G)$. Vertex groups of $T$ are either quadratically hanging, or elements of $\mc{S}(G)$, or ``rigid'' subgroups. The group $\Out^1(G)$ preserves the conjugacy class of each vertex group $V$ and, if $V$ is rigid, $\Out^1(G)$ restricts to a finite subgroup of $\Out(V)$.
    (See Sections~3.3 and~4.1 in \cite{GL-relhyp}.)

    Now, consider some $\phi\in\Out(G)$. Up to raising $\phi$ to a power, which does not affect goodness, we can suppose that $\phi$ lies in $\Out^1(G)$ and restricts to an inner automorphism on all rigid vertex groups of $T$. 
    The restrictions of $\phi$ to all QH vertex groups are $1$--good (see \cite{BH92,Levitt-GAFA}), while the restrictions to the elements of $\mc{S}(G)$ are $\alpha_0$--good by hypothesis. Also observe that non-rigid vertex groups $W\leq G$ are \emph{conjugacy-undistorted}, meaning that, for all finite generating sets $T\sq W$ and $S\sq G$, the conjugacy length functions $\|\cdot\|_T$ and $\|\cdot\|_S$ are bi-Lipschitz equivalent on $W$. Indeed, $W$ is either a singular subgroup (in which case this follows from \Cref{rmk:conjlength_vs_undistortion}), or QH (in which case, this is easily shown by considering a graph of spaces).  

    From this, it is routine to deduce that $\phi$ is $(\alpha_0+2)$--good as required. See \cite[Proposition~4.2]{Fio11a} for details.
\end{proof}

We can finally prove \Cref{prop:alpha_good}. The \emph{Grushko rank} $\Gr(G)$ is the sum $k+m$ in any writing $G=G_1\ast\dots\ast G_k\ast F_m$ with freely indecomposable $G_i$.

\begin{proof}[Proof of \Cref{prop:alpha_good}]
    We proceed by induction on the ambient rank of $G$. The base step ${\rm ar}(G)=1$ is immediate. For the inductive step, suppose that all special groups $H$ with $\ar(H)<\ar(G)$ are good.

    If $G$ virtually splits as a direct product, then $G$ is good. Indeed, there exists $\alpha\in\N$ such that the virtual direct factors of $G$ are $\alpha$--good by the inductive assumption, and this implies that $G$ is itself $\alpha$--good (for instance, see \cite[Corollary~3.6]{Fio11a}). Similarly, if $G$ is freely decomposable and its indecomposable factors are $\beta$--good for some $\beta\in\N$, then $G$ is $(\beta+2\Gr(G))$--good, where $\Gr(G)$ denotes the Grushko rank of $G$. This follows by considering, for each $\phi\in\Out(G)$, either a relative train track map for $\phi$, or a $\phi$--invariant splitting of $G$ as a graph of groups with trivial edge groups (the increase in the goodness parameter is due to sub-polynomially-growing automorphisms preserving a sporadic free factor). For details, we refer to \cite[Corollary~5.3]{Fio11a} and \cite[Proposition~4.2]{Fio11a}, respectively. Note that the freely indecomposable factors of $G$ are themselves special groups by \Cref{lem:cc_edges}(1), and their ambient rank does not exceed that of $G$.

    In conclusion, it suffices to prove the proposition under the assumption that $G$ is $1$--ended and $G\not\in\mc{S}(G)$. Denote by $\mc{S}^*(G)$ the union of $\mc{S}(G)$ with the family of cyclic subgroups of $G$ whose conjugacy class has finite $\Out(G)$--orbit. \Cref{thm:JSJ+} yields an $\Out(G)$--invariant $(\mc{ZZ}(G),\mc{S}^*(G))$--tree $G\acts T$ such that each vertex-stabiliser $V$ is either QH relative to $\mc{S}^*(G)$, or convex-cocompact and $(\mc{Z}(V),\mc{S}^*(G)|_V)$--rigid in itself. The QH vertex groups of $T$ are free or surface groups, and thus $1$--good. Our goal is then to show that rigid vertex groups are also good, for all practical purposes.

    \smallskip
    {\bf Claim.} \emph{For every rigid vertex group $V$ there exists $\alpha_V\in\N$ such that, for every $\varphi\in\Aut(G)$ with $\varphi(V)=V$, the restriction $\varphi|_V$ is $\alpha_V$--good.}

    \smallskip\noindent
    \emph{Proof of claim.}
    It suffices to prove the claim under the assumption that $\ar(V)=\ar(G)$, otherwise it follows from the inductive hypothesis. Similarly, we can assume that $V\not\in\mc{S}(V)$ and that $V$ is not cyclic. In particular, $V$ is the $G$--stabiliser of a unique vertex $v\in T$.
    
    We wish to apply \Cref{lem:tame_auxiliary} to $V$. Let $\mc{H}_V$ be the union of $\mc{S}(V)$ with the family of $G$--stabilisers of edges of $T$ incident to $v$. There are finitely many $V$--conjugacy classes of subgroups in $\mc{H}_V$, and all non-cyclic elements of $\mc{H}_V$ are convex-cocompact with ambient rank strictly smaller than $\ar(G)=\ar(V)$. In particular, the inductive hypothesis implies that there exists $\alpha_V\in\N$ such that all groups in $\mc{H}_V\cup\mc{P}(V)\setminus\{V\}$ are $\alpha_V$--good. Moreover, every subgroup in $\mc{S}^*(G)|_V$ is contained in an element of $\mc{H}_V$ (see the proof of \Cref{prop:product_embedding} for details), and hence $V$ is $(\mc{Z}(V),\mc{H}_V)$--rigid in itself. Finally, letting $\Phi\in\Aut(T)$ be the map representing $\varphi$, the fact that $\varphi(V)=V$ implies that $\Phi(v)=v$,
    and hence the collection $\mc{H}_V$ is $\varphi|_V$--invariant.
    
    Now, if $\varphi|_V\in\Aut(V)$ is not $(\alpha_V+1)$--good, \Cref{lem:tame_auxiliary} shows that $V$ is hyperbolic relative to $\mc{S}(V)$ and that all elements of $\mc{H}_V$ are contained in elements of $\mc{S}(V)$. Since $G$ is $1$--ended, $V$ is $1$--ended relative to $\mc{H}_V$, and the previous fact implies that $V$ is $1$--ended relative to $\mc{S}(V)$; in fact, since the elements of $\mc{S}(V)$ are freely indecomposable, $V$ is actually $1$--ended in the absolute sense. Finally, \Cref{lem:relhyp_new} shows that $V$ is $(\alpha_V+2)$--good, and hence $\varphi|_V$ is $(\alpha_V+2)$--good in all cases.
    \hfill$\blacksquare$

    \smallskip
    Since $T$ has only finitely many $G$--orbits of vertices, the claim implies that there exists an integer $\alpha\in\N$ such that all restrictions of automorphisms of $G$ to the vertex groups of $T$ are $\alpha$--good. Recalling that $T$ is $\Out(G)$--invariant and all its vertex groups are conjugacy-undistorted (as they are QH or convex-cocompact), we conclude that $G$ is $(\alpha+2)$--good (see again \cite[Proposition~4.2]{Fio11a} for details). This proves the proposition.
\end{proof}

\subsection{The singular growth rate}\label{sub:osing}

In the previous subsection, we have shown that automorphisms of special groups are either sub-polynomial or docile. We can now use this information to reduce the study of the top growth rate to singular and quadratically hanging subgroups (\Cref{prop:pure_above_osing}). 

Let $G$ be a special group. Let $\Out^0(G)\leq\Out(G)$ be the finite-index subgroup of automorphisms preserving each $G$--conjugacy class in $\mc{S}(G)$ (recall \Cref{lem:parabolics_cofinite}). We are naturally led to the following auxiliary growth rate.

\begin{defn}\label{defn:osing}
    For $\phi\in\Out^0(G)$, the \emph{singular growth rate} of $\phi$ is
    \[ \overline{\mf{o}}_{\rm sing}(\phi):=\sum_S \overline{\mf{o}}_{\rm top}(\phi|_S), \]
    where the sum is taken over finitely many representatives $S\in\mc{S}(G)$ of the $G$--conjugacy classes of singular subgroups. For a general element $\phi\in\Out(G)$, we set $\overline{\mf{o}}_{\rm sing}(\phi):=\frac{1}{k}\ast\overline{\mf{o}}_{\rm top}(\phi^k)$ for any integer $k\in\N$ such that $\phi^k\in\Out^0(G)$ (using the operations introduced in \Cref{subsub:operations_growth_rates}).
\end{defn}

When $\mc{S}(G)=\emptyset$, we simply set $\overline{\mf{o}}_{\rm sing}(\phi):=[1]$. If $\om$ is a non-principal ultrafilter, we similarly write $\mf{o}^{\om}_{\rm sing}(\phi)$ for the projection of $\overline{\mf{o}}_{\rm sing}(\phi)$ to $\mf{G}_{\om}$. Note that \Cref{defn:osing} makes sense also because of the following:

\begin{rmk}\label{rmk:restrictions_vs_otop}
    Let $H\leq G$ be a convex-cocompact subgroup, and let $\phi\in\Out(G)$ preserve the $G$--conjugacy class of $H$. Although the restriction $\phi|_H\in\Out(H)$ is not uniquely defined in general (\Cref{rmk:restriction_new}), the growth rate $\overline{\mf{o}}_{\rm top}(\phi|_H)$ is well-defined. Indeed, any two possible restrictions $\phi|_H$ differ by the restriction to $H$ of an inner automorphism of $G$, and so conjugacy lengths grow at the same speed under their powers.
\end{rmk}

We can now connect the enhanced JSJ decomposition of $G$ to the growth behaviour of an outer automorphism $\phi\in\Out(G)$. Denote by $\mc{K}_{\rm sing}(\phi)$ the collection of $\overline{\mf{o}}_{\rm sing}(\phi)$--controlled subgroups of $G$ (in the sense of \Cref{sub:beat}). The collection $\mc{K}_{\rm sing}(\phi)$ always contains both $\mc{S}(G)$ and the family of cyclic subgroups of $G$ whose $G$--conjugacy class has finite $\phi$--orbit. Moreover, $\mc{K}_{\rm sing}(\phi)$ is $\phi$--invariant (\Cref{rmk:beat_invariant}). 

\begin{prop}\label{prop:JSJ_osing}
    Let $G$ be special and $1$--ended. For any $\phi\in\Out(G)$, there exists a $\phi$--invariant $(\mc{ZZ}(G),\mc{K}_{\rm sing}(\phi))$--tree $G\acts T$ such that the $G$--stabiliser of each vertex of $T$ is:
    \begin{enumerate}
        \item[(a)] either an optimal quadratically hanging subgroup relative to $\mc{K}_{\rm sing}(\phi)$;
        \item[(b)] or a convex-cocompact root-closed subgroup of $G$ that lies in $\mc{K}_{\rm sing}(\phi)$.
	\end{enumerate}
    Moreover, edges with stabiliser $\not\in\mc{Z}(G)$ contain type~(a) vertices.
\end{prop}
\begin{proof}
    This is a near-immediate consequence of \Cref{thm:JSJ+} applied with $\mc{O}:=\langle\phi\rangle$ and $\mc{H}:=\mc{K}_{\rm sing}(\phi)$. We only need to check that type~(b) vertex groups lie in $\mc{K}_{\rm sing}(\phi)$. Thus, consider a convex-cocompact subgroup $V\leq G$ that is $(\mc{Z}(V),\mc{K}_{\rm sing}(\phi)|_V)$--rigid in itself, and whose $G$--conjugacy class has finite $\phi$--orbit. For every subgroup $H\in\mc{K}_{\rm sing}(\phi)|_V$ and every element $h\in H$, we have $\|\phi^n(h)\|\preceq\overline{\mf{o}}_{\rm sing}(\phi)$ (it does not matter whether we compute conjugacy lengths in $V$ or in $G$, by \Cref{rmk:conjlength_vs_undistortion}). Up to raising $\phi$ to a power, it preserves the conjugacy class of $V$. Now, \Cref{lem:rigid_bounds_otop} shows that $\overline{\mf{o}}_{\rm top}(\phi|_V)\preceq\overline{\mf{o}}_{\rm sing}(\phi)$, and hence $V\in\mc{K}_{\rm sing}(\phi)$ as required.
\end{proof}

We refer to the tree $T$ from \Cref{prop:JSJ_osing} as a \emph{JSJ tree adapted to $\phi$}, and we denote it by $G\acts T_{\rm sing}(\phi)$ when necessary.

\begin{cor}\label{cor:K_sing}
    Let $G$ be special and $1$--ended. Consider $\phi\in\Out(G)$.
    \begin{enumerate}
        \item Every element of $\mc{K}_{\rm sing}(\phi)$ is contained in a maximal element of $\mc{K}_{\rm sing}(\phi)$. There are only finitely many $G$--conjugacy classes of maximal elements, and they are all convex-cocompact.
        \item If a subgroup $H\leq G$ is $(\mc{Z}(G),\mc{S}(G))$--rigid in $G$, then $H\in\mc{K}_{\rm sing}(\phi)$.
    \end{enumerate}
\end{cor}
\begin{proof}
    Let $G\acts T$ be a JSJ tree adapted to $\phi$. We can assume that $T$ is a splitting, otherwise $G$ is either QH or an element of $\mc{K}_{\rm sing}(\phi)$, and the corollary is clear. Note that all peripheral subgroups of the QH vertex groups of $T$ are contained in edge groups of $T$, because $G$ is $1$--ended. 
    
    Every subgroup $H\in\mc{K}_{\rm sing}(\phi)$ fixes a vertex of $T$. If all vertices fixed by $H$ are of type~(a), then $H$ is peripheral in all of these QH vertex groups. By the previous observation, it follows that either $H=\{1\}$ or $H$ is contained in the stabiliser of an edge $e\sq T$ both of whose vertices are of type~(a). In this case, $G_e$ is a maximal element of $\mc{K}_{\rm sing}(\phi)$, it is infinite cyclic, and $G_e$ is convex-cocompact by \Cref{lem:cc_near_QH} (applied to a suitable collapse of the barycentric subdivision of $T$). 
    
    The other option is that $H$ fixes at least one vertex $x\in T$ of type~(b). In this case, we have $G_x\in\mc{K}_{\rm sing}(\phi)$ and, up to replacing $x$, we can assume that $G_x$ is maximal among $G$--stabilisers of vertices of $T$ (recall \Cref{lem:cc_basics}(3)).
    In conclusion, the maximal elements of $\mc{K}_{\rm sing}(\phi)$ are precisely the $G$--stabilisers of the edges $e\sq T$ with both vertices of type~(a), as well as the maximal $G$--stabilisers of the type~(b) vertices of $T$. This proves part~(1).

    As to part~(2), \Cref{thm:JSJ+}(4) shows that each $(\mc{Z}(G),\mc{S}(G))$--rigid subgroup fixes either a type~(b) vertex, or an edge connecting two type~(a) vertices. Together with the previous discussion, this completes the proof.
\end{proof}

Let $G$ be special and $1$--ended, and let $G\acts T_{\rm sing}(\phi)$ be a JSJ tree adapted to $\phi$. Let $k\geq 1$ be an integer such that $\phi^k$ acts trivially on the quotient graph $T_{\rm sing}(\phi)/G$. For every QH vertex group $Q$, we can consider the restriction $\phi^k|_Q\in\Out(G)$. As $\phi^k|_Q$ preserves the peripheral subgroups of the associated surface $\Sigma_Q$, it is represented by an element of the mapping class group of $\Sigma_Q$; we can then consider the Nielsen--Thurston decomposition of this mapping class. We define $\Lambda_{\rm sing}(\phi^k):=\{\lambda_1,\dots,\lambda_{\ell}\}$, where the $\lambda_i$ are the stretch factors of the pseudo-Anosov components of $\phi^k|_Q$, as $Q$ varies through the QH vertex groups of $T$. Finally, we define
\[ \Lambda_{\rm sing}(\phi):=\{\lambda_1^{1/k},\dots,\lambda_{\ell}^{1/k}\} ,\]
and note that this set is independent of the choice of the integer $k$.

So far in \Cref{sub:osing}, we have only used the material from \Cref{sect:JSJ}, but we are about to introduce \Cref{prop:alpha_good} into the discussion. This allows us to deduce that $\phi$ has only a finite number of growth rates above $\overline{\mf{o}}_{\rm sing}(\phi)$ and that each is realised on a quadratically hanging vertex group.

\begin{prop}\label{prop:pure_above_osing}
    Let $G$ be special and $1$--ended. Consider $\phi\in\Out(G)$.
    \begin{enumerate}
        \setlength\itemsep{.25em}
        \item If $\overline{\mf{o}}_{\rm sing}(\phi)$ is not sub-polynomial, then each element $g\in G$ satisfies either $\|\phi^n(g)\|\preceq\overline{\mf{o}}_{\rm sing}(\phi)$, or $\|\phi^n(g)\|\sim\lambda^n$ for some $\lambda\in\Lambda_{\rm sing}(\phi)$.
        \item If $\overline{\mf{o}}_{\rm sing}(\phi)\preceq n^p$ for some $p\in\N$, then each element $g\in G$ satisfies either $\|\phi^n(g)\|\preceq n^{p+2}$ or $\|\phi^n(g)\|\sim\lambda^n$ for some $\lambda\in\Lambda_{\rm sing}(\phi)$.
        \item Either $\overline{\mf{o}}_{\rm top}(\phi)$ is sub-polynomial, or we have $\overline{\mf{o}}_{\rm top}(\phi)\sim\overline{\mf{o}}_{\rm sing}(\phi)$, or we have $\overline{\mf{o}}_{\rm top}(\phi)\sim\lambda^n$ for $\lambda=\max\Lambda_{\rm sing}(\phi)$.
    \end{enumerate}
\end{prop}
\begin{proof}
    It suffices to prove parts~(1) and~(2) of the proposition, as part~(3) is an immediate consequence of these (using \Cref{rmk:dominating_otop}). 
    
    We write $T:=T_{\rm sing}(\phi)$ for simplicity and suppose that $\phi$ acts trivially on the graph $T/G$, which can be achieved by raising $\phi$ to a power. We assume that, for each type~(a) vertex group $Q$ of $T$, the restriction $\phi|_Q$ is represented by a pseudo-Anosov homeomorphism on the associated surface $\Sigma_Q$ (fixing $\partial\Sigma_Q$ pointwise). This can be achieved by refining $T_{\rm sing}(\phi)$, splitting each type~(a) vertex group $Q$ according to the Nielsen--Thurston decomposition of $\phi|_Q$; any resulting surface on which $\phi$ restricts to a finite-order or linearly-growing homeomorphism should be considered of type~(b).
    
    Let $\mc{E}$ be the set of edges of $T$ with both vertices of type~(b). Let $G\acts T'$ be the tree obtained from $T$ by collapsing all edges in $\mc{E}$. Note that $T'$ is still $\phi$--invariant and $\phi$ still acts trivially on $T'/G$. We can still speak of type~(a) and type~(b) vertices of $T'$, as each fibre of the collapse map $\pi\colon T\ra T'$ is either a single type~(a) vertex, or a subtree of type~(b) vertices.

    We begin by analysing growth rates of elements of $G$ fixing type~(b) vertices of $T'$. Thus, consider a type~(b) vertex $w\in T'$ and its stabiliser $W$. If $\overline{\mf{o}}_{\rm sing}(\phi)\preceq n^p$, we have $\overline{\mf{o}}_{\rm top}(\phi|_V)\preceq n^p$ for every type~(b) vertex group $V$ of $T$ and so, considering the tree $W\acts\pi^{-1}(w)$, we see that $\overline{\mf{o}}_{\rm top}(\phi|_W)\preceq n^{p+2}$ (for instance, using \cite[Proposition~4.2]{Fio11a}). If instead $\overline{\mf{o}}_{\rm sing}(\phi)$ is not sub-polynomial, then $\overline{\mf{o}}_{\rm sing}(\phi)$ is tame by \Cref{prop:alpha_good}, since sums of tame growth rates are tame. Thus, considering again the tree $W\acts\pi^{-1}(w)$ and using that all growth rates on the vertex groups are $\preceq\overline{\mf{o}}_{\rm sing}(\phi)$, we obtain\footnote{We emphasise that this deduction requires the property described in \Cref{rmk:tame_implies_sum-stable}, and so it could fail if we did not know that $\overline{\mf{o}}_{\rm sing}(\phi)$ is tame. Particularly, if have a $\phi$--invariant HNN extension and the top growth rate on the vertex group is not tame, it is a priori possible for $\phi$ to grow strictly faster on the whole group than it does on the vertex group.}
    that $\overline{\mf{o}}_{\rm top}(\phi|_W)\preceq\overline{\mf{o}}_{\rm sing}(\phi)$. (Again, see \cite[Proposition~4.2]{Fio11a} for details.) In other words, we have $W\in\mc{K}_{\rm sing}(\phi)$.

    Summing up, if an element $g\in G$ fixes a type~(b) vertex of $T'$, then we have $\|\phi^n(g)\|\preceq n^{p+2}$ if $\overline{\mf{o}}_{\rm sing}(\phi)\preceq n^p$, and we instead have $\|\phi^n(g)\|\preceq\overline{\mf{o}}_{\rm sing}(\phi)$ if $\overline{\mf{o}}_{\rm sing}(\phi)$ is not sub-polynomial.

    In order to complete the proof of parts~(1) and~(2), we are left to describe the growth rate $\big[\,\|\phi^n(g)\|\,\big]$ when $g$ fixes a type~(a) vertex or is loxodromic in $T'$. In the former case, Nielsen--Thurston theory shows that $\big[\,\|\phi^n(g)\|\,\big]$ is either $[1]$ or $[\lambda^n]$ for some $\lambda\in\Lambda_{\rm sing}(\phi)$. 
    Suppose instead that $g$ is loxodromic and let $\alpha\sq T'$ be its axis. By the construction of $T'$, every edge of $T'$ is incident to at least one type~(a) vertex, and thus $\alpha$ contains a type~(a) vertex $q$ and two incident edges $e,e'$. Since $q$ is an optimal QH vertex, the edges $e$ and $e'$ correspond to distinct full peripheral subgroups of its stabiliser $Q$, and the pair $(e,e')$ determines a homotopy class $[\g]$ of properly embedded arcs on the associated surface $\Sigma_q$. As the restriction of $\phi$ to $Q$ is represented by a pseudo-Anosov homeomorphism $\psi_q$ of $\Sigma_q$, the length of the arcs $\psi_q^n(\g)$ grows like $\lambda_q^n$ for some $\lambda_q\in\Lambda_{\rm sing}(\phi)$. From this, it is straightforward to deduce that $\|\phi^n(g)\|\sim\lambda_q^n$ for the maximal $\lambda_q\in\Lambda_{\rm sing}(\phi)$ such that $q$ is a type~(a) vertex on the axis $\alpha\sq T'$. This concludes the proof.
\end{proof}

\begin{cor}\label{cor:pure_above_osing}
    If $G$ is special and $1$--ended, there exists an integer $\ell=\ell(G)$ such that the following holds. For each $\phi\in\Out(G)$, there are at most $\ell$ non-sub-polynomial growth rates $\mf{o}\in\mf{g}(\phi)$ with $\mf{o}\not\preceq\overline{\mf{o}}_{\rm sing}(\phi)$. For each such $\mf{o}$, we have $\mf{o}\sim\lambda^n$ with $\lambda$ the stretch factor of a pseudo-Anosov.
\end{cor}
\begin{proof}
    This immediately follows from \Cref{prop:pure_above_osing}, observing that the cardinality of the set $\Lambda_{\rm sing}(\phi)$ is bounded above in terms of the total complexity of the QH vertex groups appearing in $T_{\rm sing}(\phi)$. All QH subgroups of $G$ appear in any JSJ decomposition of $G$, in the sense of \cite{GL-JSJ}, and thus their total complexity is bounded independently of the automorphism $\phi$.
\end{proof}

\subsection{Structure of top growth rates}\label{sub:concluding_tame}

We can finally reap the fruits of our work in the earlier parts of \Cref{sect:tameness} and prove \Cref{thm:tame} below, which implies \Cref{thmintro:general_aut} (as discussed after the statement).

Before stating the theorem, we need some terminology. A subgroup $H\leq G$ is \emph{$\phi$--invariant}, for some $\phi\in\Out(G)$, if the representatives of $\phi$ in $\Aut(G)$ preserve the $G$--conjugacy class of $H$. Following \cite{GH22}, a \emph{factor system} for $G$ is the collection of nontrivial vertex-stabilisers of some free splitting of $G$. A factor system is \emph{sporadic} if it either consists of the conjugates of a subgroup $H$ such that $G=H\ast\Z$, or it consists of the conjugates of subgroups $H,K$ such that $G=H\ast K$. For a factor system $\mc{F}$, we denote by $\Out(G,\mc{F})\leq\Out(G)$ the subgroup of outer classes leaving invariant each conjugacy class of subgroups in $\mc{F}$. Finally, a \emph{$(G,\mc{F})$--free factor} is a subgroup of $G$ arising as a vertex group in a free splitting of $G$ relative to $\mc{F}$, and we say that $\phi\in\Out(G,\mc{F})$ is \emph{fully irreducible} if no (nontrivial) power of $\phi$ preserves the $G$--conjugacy class of a proper $(G,\mc{F})$--free factor. 

\begin{thm}\label{thm:tame}
    Let $G$ be a special group. There is an integer $\pi=\pi(G)$ such that the following hold for each $\varphi\in\Aut(G)$ with outer class $\phi\in\Out(G)$.
    \begin{enumerate}
    \setlength\itemsep{.25em}
        \item Either we have $\overline{\mc{O}}_{\rm top}(\varphi)\preceq[n^{\pi}]$, or there exists an algebraic integer $\lambda>1$ such that $\phi$ is $(\lambda,\pi)$--docile. 
        \item If $\phi$ is $(\lambda,\pi)$--docile, then there exist an integer $k\geq 1$ and a $\phi^k$--invariant subgroup $H\leq G$ of one of the following three kinds.
            \begin{enumerate}
                \item[$(i)$] $H$ is isomorphic to the fundamental group of a compact surface on which $\phi^k$ is represented by a pseudo-Anosov homeomorphism with stretch factor $\lambda^k$. In particular, all non-peripheral elements $h\in H$ satisfy $\|\phi^n(h)\|\sim\lambda^n$.
                \item[$(ii)$] $H$ is a convex-cocompact, infinitely-ended subgroup with a non-sporadic factor system $\mc{F}$ such that $\phi^k$ is a fully irreducible element of $\Out(H,\mc{F})$ and $\lambda^k$ is its Perron--Frobenius eigenvalue. Moreover, every element $h\in H$ not conjugate into a subgroup in $\mc{F}$ satisfies either $\|\phi^n(h)\|\sim 1$ or $\|\phi^n(h)\|\sim\lambda^n$.
                \item[$(iii)$] $H$ is a convex-cocompact subgroup of the form $H'\x\Z^m$ for some $m\geq 1$, and there exist elements $h\in H$ with $\|\phi^n(h)\|\sim n^p\lambda^n$ for some $0\leq p\leq\pi$. Moreover, either $H'$ is a surface group, or $H'$ is infinitely ended, or $\phi^k$ descends to an automorphism of the abelianisation of $H$ having $\lambda^k$ as the maximum modulus of an eigenvalue.
            \end{enumerate}
        \item If $\overline{\mf{o}}_{\rm top}(\phi)$ is pure, there exists a $\phi^k$--invariant subgroup $H$ containing elements $h$ with $\|\phi^n(h)\|\sim\overline{\mf{o}}_{\rm top}(\phi)$, and $H$ is of type (i), (iii), or:
            \begin{enumerate}
                \item[$(ii)'$] $H$ is convex-cocompact, infinitely-ended and has a non-sporadic factor system $\mc{F}$ such that $\phi^k\in\Out(H,\mc{F})$ is fully irreducible. Each element $h\in H$ not conjugate into a subgroup in $\mc{F}$ satisfies either $\|\phi^n(h)\|\sim 1$ or $\|\phi^n(h)\|\sim\overline{\mf{o}}_{\rm top}(\phi)$.
            \end{enumerate}
    \end{enumerate}
\end{thm}

Parts~(1) and (2) of \Cref{thm:tame} imply \Cref{thmintro:general_aut}: this is immediate for $G$ special and, for a \emph{virtually} special group $U$, it simply requires applying \Cref{lem:fi_fix} to a finite-index, characteristic, special subgroup $G\lhd U$. Note that parts~(1) and (2) of \Cref{thm:tame} are actually a little stronger than \Cref{thmintro:general_aut}, and this is an important difference: as we have already hinted at in \Cref{sub:osing}, tameness/docility are essential for the \emph{entire} proof of \Cref{thm:tame} to work.

Part~(3) of \Cref{thm:tame} will be useful in \Cref{sect:cmp}, where we study coarse-median preserving automorphisms and show that their growth rates are either sub-polynomial or pure. 

Before proving \Cref{thm:tame}, we make a few more comments.

\begin{rmk}
    In parts~(2) and~(3) of \Cref{thm:tame}, Case~(i) is the only one where the subgroup $H$ might not be convex-cocompact (confront \Cref{rmk:non_cc_QH}). However, $H$ is still a quadratically hanging vertex group in a cyclic splitting of a $\phi^k$--invariant convex-cocompact subgroup of $G$. As such, all non-peripheral elements of $H$ are convex-cocompact in $G$.
\end{rmk}

\begin{rmk}
     Although \Cref{thm:tame}(2) constructs elements $h\in G$ with $\|\phi^n(h)\|\sim\lambda^n$ or $\|\phi^n(h)\|\sim n^p\lambda^n$, there is no claim that these growth rates coincide with $\overline{\mf{o}}_{\rm top}(\phi)$, or even that $\overline{\mf{o}}_{\rm top}(\phi)$ be precisely realised on an element of $G$. We only know that $\overline{\mf{o}}_{\rm top}(\phi)$ is between $[\lambda^n]$ and $[n^{\pi}\lambda^n]$, and that there exist elements $h\in G$ realising $\overline{\mf{o}}_{\rm top}(\phi)$ with at most polynomial error. Things will be cleaner for coarse-median preserving automorphisms.
\end{rmk}

\begin{rmk}
    When $\phi$ is docile, the base of exponential growth $\lambda$ is not just any algebraic integer: either it is the stretch factor of a pseudo-Anosov on a compact surface, or it is the Perron--Frobenius eigenvalue of a fully irreducible automorphism of a free group, or it is the maximum modulus of an eigenvalue of a matrix in ${\rm GL}_m(\Z)$ for some $m\in\Z$. This completely characterises which algebraic integers $\lambda$ can appear in \Cref{thm:tame}.

    In the first two of the three listed cases, $\lambda$ is a \emph{Perron number} \cite{Thurston-BAMS,BH92}: $\lambda$ has modulus strictly larger than that of its Galois conjugates. As to the third case, a number $\lambda$ is the maximum modulus of an eigenvalue of an element of ${\rm GL}_m(\Z)$ if and only if $\lambda$ is both an \emph{algebraic unit} (i.e.\ both $\lambda$ and $\lambda^{-1}$ are algebraic integers) and a \emph{weak Perron number} (i.e.\ the modulus of $\lambda$ is at least as big as that of its Galois conjugates). To see this, it suffices to consider the companion matrix of the minimal polynomial of $\lambda$.
\end{rmk}

We now prove \Cref{thm:tame}.

\begin{proof}[Proof of \Cref{thm:tame}]
    \Cref{prop:alpha_good} shows that either $\phi$ is docile, or all its representatives in $\Aut(G)$ grow sub-polynomially. We prove the rest of the theorem by induction on the ambient rank $\ar(G)$. This amounts to bounding the integer $\pi$ in part~(1) independently of the automorphism, to checking that $\lambda$ is an algebraic integer, and to finding a subgroup $H$ as in parts~(2) and~(3). The base step is trivial, so we assume that the theorem holds for all special groups $G'$ with $\ar(G')<\ar(G)$.

    Suppose first that $G$ is $1$--ended and that $G\not\in\mc{S}(G)$. Then \Cref{prop:pure_above_osing} guarantees that there are only three cases to consider: 
    \begin{enumerate}
        \item[$(x)$] either there exists $p\in\N$ such that $\overline{\mf{o}}_{\rm sing}(\phi)\preceq n^p$ and $\overline{\mf{o}}_{\rm top}(\phi)\preceq n^{p+2}$;
        \item[$(y)$] or we have $\overline{\mf{o}}_{\rm top}(\phi)\sim\overline{\mf{o}}_{\rm sing}(\phi)$ and this growth rate is tame;
        \item[$(z)$] or we have $\overline{\mf{o}}_{\rm top}(\phi)\sim\lambda^n$ for some $\lambda>1$, and $\overline{\mf{o}}_{\rm top}(\phi)$ is realised on a QH vertex group $Q$ of a $\phi$--invariant splitting of $G$.
    \end{enumerate}
    In Case~$(x)$, the inductive assumption applied to the finitely many $G$--conjugacy classes of subgroups in $\mc{S}(G)$ yields $p\leq\pi'$ for some $\pi'=\pi'(G)$. \Cref{prop:alpha_good} then yields an integer $\alpha=\alpha(G)$ such that every representative $\varphi\in\Aut(G)$ of $\phi$ satisfies $\overline{\mc{O}}_{\rm top}(\varphi)\preceq n^{p+2+\alpha}$. Thus, part~(1) of the theorem holds and parts~(2) and (3) are void. 
    
    In Case~$(y)$, the theorem similarly holds by the inductive assumption. Indeed, up to raising $\phi$ to a bounded power, each conjugacy class in $\mc{S}(G)$ is $\phi$--invariant, and thus $\overline{\mf{o}}_{\rm sing}(\phi)$ is a sum of growth rates $\overline{\mf{o}}_{\rm top}(\phi|_S)$ with $S$ in a finite subset $\mc{S}_0\sq\mc{S}(G)$. Since $\overline{\mf{o}}_{\rm sing}(\phi)$ is $(\lambda,p)$--tame for some $\lambda>1$ and $p\in\N$, there exists $S\in\mc{S}_0$ such that $\overline{\mf{o}}_{\rm top}(\phi|_S)$ is $(\lambda,p)$--tame. Similarly, if $\overline{\mf{o}}_{\rm sing}(\phi)\sim n^p\lambda^n$, then there exists $S\in\mc{S}_0$ such that $\overline{\mf{o}}_{\rm top}(\phi|_S)\sim n^p\lambda^n$. Applying the inductive hypothesis to $S$, we obtain the theorem.

    In Case~$(z)$, a power $\phi^k$ preserves the conjugacy class of $Q$ and we can assume that the restriction $\phi^k|_Q$ is represented by a pseudo-Anosov homeomorphism of the associated surface. This proves the theorem in this case.

    Summing up, under the inductive hypothesis, we have shown that the theorem holds when $G$ is $1$--ended and $G\not\in\mc{S}(G)$.

    Suppose now that $G\in\mc{S}(G)$. It is not restrictive to assume that $G$ itself splits as a direct product, as each finite-index subgroup of $G$ is preserved by a power of each automorphism of $G$. Thus, let $G=G_1\x\dots\x G_a\x\Z^b$, where the $G_j$ are directly indecomposable and have trivial centre. 

    Consider an automorphism $\varphi\in\Aut(G)$. Up to raising $\varphi$ to a power, each subgroup $\langle G_j,\Z^b\rangle$ is left invariant by $\varphi$ (see e.g.\ \cite[Section~3]{Fio11a}). Thus, we can represent $\varphi$ as a (formal) triangular matrix with entries automorphisms $(\varphi_1,\dots,\varphi_a)\in\prod_j\Aut(G_j)$, $\psi\in{\rm GL}_b(\Z)$ and a homomorphism $\alpha\colon \prod_jG_j\ra\Z^b$. Part~(1) of the theorem is easily deduced from this description and the inductive hypothesis; see \cite[Corollary~3.6]{Fio11a} for details. Parts~(2) and~(3) are similarly immediate if $b=0$, in which case we can restrict to the $G_j$, which are all $\varphi$--invariant.

    Thus, suppose that $b\geq 1$ and let $\varphi$ be $(\lambda,\pi)$--docile. If none of the $\varphi_j$ is $(\lambda,\pi)$--docile (that is, if they grow exponentially-slower than $[\lambda^n]$), then $\overline{\mf{o}}_{\rm top}(\varphi)$ is realised on the abelianisation of $G$, and in particular it is pure; in this case, we can simply take $H:=G$ in parts~(2) and~(3) of the theorem. If instead there exists $j$ such that $\varphi_j$ is $(\lambda,\pi)$--docile, then it suffices to consider the subgroup $H_j\leq G_j$ provided by the inductive hypothesis and set $H:=H_j\x\Z^b$ (when $\overline{\mf{o}}_{\rm top}(\varphi)$ is pure, one should choose the index $j$ so that $\overline{\mf{o}}_{\rm top}(\varphi_j)\sim\overline{\mf{o}}_{\rm top}(\varphi)$). This concludes the proof when $G\in\mc{S}(G)$.

    In order to complete the proof of the theorem, we are only left to consider the possibility that $G$ is not $1$--ended. In this case, we argue by induction on the Grushko rank $\Gr(G)$.

    Up to raising $\phi\in\Out(G)$ to a power, we can assume that there exists a factor system $\mc{F}$ such that $\phi$ is a fully irreducible element of $\Out(G,\mc{F})$. Each subgroup $F\in\mc{F}$ has strictly lower Grushko rank than $G$, so the restriction $\phi|_F$ satisfies the theorem. Thus, part~(1) of the theorem can be routinely deduced by considering a relative train-track map for $\phi$ (if $\mc{F}$ is non-sporadic) or a $\phi$--invariant free splitting of $G$ with vertex groups in $\mc{F}$ (if $\mc{F}$ is sporadic). See \cite[Proposition~5.2]{Fio11a} and \cite[Proposition~4.2]{Fio11a} for details.

    Regarding part~(2), suppose that $\phi$ is $(\lambda,\pi)$--docile. If there exists $F\in\mc{F}$ such that the restriction $\phi|_F$ is $(\lambda,\pi)$--docile, we obtain the required subgroup $H$ by the inductive assumption. Otherwise, all restrictions $\phi|_F$ are sub-polynomial or $(\mu,\pi)$--docile for values $\mu<\lambda$. In this case, all elements $g\in G$ not conjugate into subgroups in $\mc{F}$ satisfy either $\|\phi^n(g)\|\sim 1$ or $\|\phi^n(g)\|\sim\lambda^n$, and the latter is the growth rate $\overline{\mf{o}}_{\rm top}(\phi)$, as required.
    Part~(3) is similar: if $\overline{\mf{o}}_{\rm top}(\phi)$ is pure and coincides with $\overline{\mf{o}}_{\rm top}(\phi|_F)$ for some $F\in\mc{F}$, we conclude by the inductive assumption; if instead $\overline{\mf{o}}_{\rm top}(\phi)$ is pure and strictly faster than all $\overline{\mf{o}}_{\rm top}(\phi|_F)$, then all elements $g\in G$ not conjugate into a subgroup in $\mc{F}$ satisfy either $\|\phi^n(g)\|\sim 1$ or $\|\phi^n(g)\|\sim\overline{\mf{o}}_{\rm top}(\phi)$, as required. Again, we refer to \cite[Proposition~5.2]{Fio11a} for details.
\end{proof}

\begin{rmk}
    It is interesting to pinpoint why we were not able to show that $\overline{\mf{o}}_{\rm top}(\phi)$ is pure for all $\phi\in\Out(G)$ growing exponentially. If $G$ is $1$--ended, not a virtual direct product, and $\overline{\mf{o}}_{\rm top}(\phi)\not\sim\overline{\mf{o}}_{\rm sing}(\phi)$, then we have seen that $\overline{\mf{o}}_{\rm top}(\phi)\sim\lambda^n$ for some $\lambda>1$, and so $\overline{\mf{o}}_{\rm top}(\phi)$ is pure in this case. The problems start if we happen to have $\overline{\mf{o}}_{\rm top}(\phi)\sim\overline{\mf{o}}_{\rm sing}(\phi)$. We might then have $\overline{\mf{o}}_{\rm top}(\phi)\sim\overline{\mf{o}}_{\rm top}(\phi|_S)$ for a singular subgroup $S\in\mc{S}(G)$, and there might be a direct factor $S'$ of $S$ such that $\overline{\mf{o}}_{\rm top}(\phi|_S)\sim\overline{\mf{o}}_{\rm top}(\phi|_{S'})$ and such $S'$ splits as a nontrivial free product $S'=S_1\ast\dots\ast S_a\ast F_b$. Now, even if the growth rates $\overline{\mf{o}}_{\rm top}(\phi|_{S_i})$ are all pure, there is no guarantee that the growth rate $\overline{\mf{o}}_{\rm top}(\phi|_{S'})$ is pure. Indeed, we would need to know that there is at least a \emph{linear} gap between $\overline{\mf{o}}_{\rm top}(\phi|_{S_i})$ and the next-fastest growth rate on each $S_i$. 

    This is beyond reach in general, but we will prove it in the coarse-median preserving case, since we can control \emph{all} growth rates in that case.
\end{rmk}

\section{The coarse-median preserving case}\label{sect:cmp}

Having shown that automorphisms of special groups are either sub-poly\-no\-mial or docile, we now proceed to gather more refined information under the assumption that the automorphism be \emph{coarse-median preserving}. 

Let $G$ be a special. Realise $G$ as a convex-cocompact subgroup of a RAAG $A_{\G}$, and equip $G$ with the induced coarse-median operator $m_{\G}\colon G^3\ra G$. We denote by $\Aut(G,m_{\G})\leq\Aut(G)$ and $\Out(G,m_{\G})\leq\Out(G)$ the subgroups of automorphisms coarsely preserving $m_{\G}$, as defined in \Cref{sub:special_prelims}.

The main result of this section is that coarse-median preserving automorphisms of $G$ always have finitely many exponential growth rates, and each of them is of the form $[n^p\lambda^n]$ for some $p\in\N$ and some algebraic integer $\lambda>1$. This proves \Cref{thmintro:cmp_aut} from the Introduction.

\begin{thm}\label{thm:cmp_main}
    Let $G\leq A_{\G}$ be convex-cocompact and let $\phi\in\Out(G,m_{\G})$.
    \begin{enumerate}
        \setlength\itemsep{.25em}
        \item Each growth rate in the set $\mf{g}(\phi)$ (defined in \Cref{subsub:growth_of_automorphisms}) is either pure or sub-polynomial. Moreover, we have $\overline{\mf{o}}_{\rm top}(\phi)\in\mf{g}(\phi)$.
        \item There are only finitely many pure growth rates in $\mf{g}(\phi)$. There exists $p\in\N$ such that all sub-polynomial growth rates in $\mf{g}(\phi)$ are $\preceq[n^p]$.
        \item For every $\varphi\in\Aut(G)$ in the outer class $\phi$, all growth rates in the symmetric difference $\mc{G}(\varphi)\triangle\mf{g}(\phi)$ are sub-polynomial.
        \item For any pure growth rate $\mf{o}\in\mf{g}(\phi)$, there are only finitely many $G$--conjugacy classes of maximal subgroups in the family $\mc{B}(\mf{o})$ (\Cref{sub:beat}). Each of these subgroups is convex-cocompact
        and its $G$--conjugacy class is preserved by a power of $\phi$.
        \item For any pure growth rate $\mf{o}\in\mf{g}(\phi)$, there are an integer $k\geq 1$ and a $\phi^k$--invariant subgroup $H\leq G$ of one of the following two kinds.
        \begin{enumerate}
            \item[(i)] $H$ is isomorphic to the fundamental group of a compact surface on which $\phi^k$ is represented by a pseudo-Anosov homeomorphism with stretch factor $\lambda^k$. In particular, all non-peripheral elements $h\in H$ satisfy $\|\phi^n(h)\|\sim\mf{o}\sim \lambda^n$.
            \item[(ii)] $H$ is a convex-cocompact, infinitely-ended subgroup with a non-sporadic factor system $\mc{F}$ such that $\phi^k$ is a fully irreducible element of $\Out(H,\mc{F})$ and $\lambda^k$ is its Perron--Frobenius eigenvalue. Each element $h\in H$ not conjugate into a subgroup in $\mc{F}$ satisfies either $\|\phi^n(h)\|\sim 1$ or $\|\phi^n(h)\|\sim\mf{o}\sim n^p\lambda^n$ for some $p\in\N$.
        \end{enumerate}
    \end{enumerate}
\end{thm}

We prove \Cref{thm:cmp_main} at the end of \Cref{sect:cmp}. Before embarking in its proof, we make a few observations on its statement.

\begin{rmk}\label{rmk:fi_cmp_fix}
    Let $G$ be a group that is only \emph{virtually} special. Let $G_0\lhd G$ be a characteristic, finite-index, special subgroup, and choose an integer $k\geq 1$ such that we have $g^k\in G_0$ for all $g\in G$. Extending \Cref{thmintro:cmp_aut} to automorphisms of $G$ with coarse-median preserving restriction to $G_0$ requires additional information on $G$ (satisfied e.g.\ by right-angled Coxeter groups):
    \begin{enumerate}
        \item Suppose that there exists a constant $\eps>0$ such that $1+\|g^k\|\geq\eps\|g\|$ for all $g\in G$. (This holds if $G$ is a CAT(0) or injective group, for instance.)
        Then, for all $\phi\in\Out(G)$ and all $g\in G$, we have that the sequences $n\mapsto\|\phi^n(g)\|$ and $n\mapsto\|\phi^n(g^k)\|$ are bi-Lipschitz equivalent. As a consequence, $\phi$ and $\phi|_{G_0}$ have the same growth rates, and \Cref{thmintro:cmp_aut}(1) holds for $G$.
        \item Suppose that there exists a constant $\eps>0$ such that $1+|g^k|\geq\eps|g|$ for all $g\in G$. (This is strictly stronger than the previous property, and implies e.g.\ that $G$ has only finitely many order--$k$ elements.) Then it is not hard to check that the root-closure in $G$ of each of the subgroups in the families $\mc{B}(\mf{o},\phi|_{G_0})$ is a subgroup (rather than just a subset). Thus, under this stronger assumption, also \Cref{thmintro:cmp_aut}(2) holds for $G$.
    \end{enumerate}
\end{rmk}

\begin{ex}\label{ex:fix_not_raag}
    Consider a RAAG $A_{\G}$ and $\phi\in\Out(A_{\G},m_{\G})$. The maximal subgroups in $\mc{B}(\overline{\mf{o}}_{\rm top}(\phi))$ need not be isomorphic to RAAGs, and the following is likely the simplest example of this. Let $\G$ be the graph below, and let $\tau_i$ be the transvection mapping $w_i\mapsto w_iv$ and fixing all other standard generators. Consider the composition $\varphi:=\tau_1\tau_2\tau_3$ and let $\phi$ be its outer class.
    
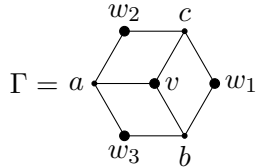
\[
    \begin{tikzpicture}[scale=0.8]
    \draw[fill] (0,0) -- (0.5,0.87);
    \draw[fill] (0,0) -- (0.5,-0.87);
    \draw[fill] (0,0) -- (-1,0);
    \draw[fill] (-1,0) -- (-0.5,0.87);
    \draw[fill] (-1,0) -- (-0.5,-0.87);
    \draw[fill] (1,0) -- (0.5,0.87);
    \draw[fill] (1,0) -- (0.5,-0.87);
    \draw[fill] (-0.5,0.87) -- (0.5,0.87);
    \draw[fill] (-0.5,-0.87) -- (0.5,-0.87);
    \draw[fill] (0,0) circle [radius=0.08cm];
    \draw[fill] (-1,0) circle [radius=0.04cm];
    \draw[fill] (1,0) circle [radius=0.08cm];
    \draw[fill] (0.5,0.87) circle [radius=0.04cm];
    \draw[fill] (0.5,-0.87) circle [radius=0.04cm];
    \draw[fill] (-0.5,0.87) circle [radius=0.08cm];
    \draw[fill] (-0.5,-0.87) circle [radius=0.08cm];
    \node[right] at (0,0) {$v$};
    \node[left] at (-1,0) {$a$};
    \node[right] at (1,0) {$w_1$};
    \node[above] at (0.5,0.87) {$c$};
    \node[below] at (0.5,-0.87) {$b$};
    \node[above] at (-0.5,0.87) {$w_2$};
    \node[below] at (-0.5,-0.87) {$w_3$};
    \node at (-2,0) {$\G=$};
    \end{tikzpicture}
    \]
    The outer automorphism $\phi$ grows linearly and there is a single conjugacy class of maximal subgroups of $\mc{B}(\overline{\mf{o}}_{\rm top}(\phi))$, namely that of the fixed subgroup $\Fix(\varphi)$. Moreover, we have 
    \[ \Fix(\varphi)=\langle a,\,b,\,c,\,v,\,w_2w_1^{-1},\,w_3w_1^{-1},\,w_1vw_1^{-1} \rangle .\]
    We expect this subgroup to not be isomorphic to any RAAG, though proving this formally may prove challenging. (It is easy to see that the generating set we have given does not correspond to a right-angled Artin presentation of $\Fix(\varphi)$, as $a$ does not commute with $w_2w_1^{-1}$ or $w_3w_1^{-1}$, but it commutes with their product $(w_2w_1^{-1})(w_3w_1^{-1})^{-1}=w_2w_3^{-1}$.)
\end{ex}

\begin{rmk}\label{rmk:uniformly_finite?}
    \Cref{thm:cmp_main} leaves open many natural questions even in the coarse-median preserving case. Is the set $\mf{g}(\phi)$ finite? Is every sub-polynomial growth rate \emph{exactly} polynomial? Is the number of pure growth rates in $\mf{g}(\phi)$ bounded only in terms of the group $G$? Confront \Cref{cor:pure_above_osing}.
\end{rmk}

The proof of \Cref{thm:cmp_main} is spread out over the next two subsections.

\subsection{Preliminaries}

Here we collect a few results on coarse-median preserving automorphisms that are needed in the proof of \Cref{thm:cmp_main}. For an expanded discussion, we refer the reader to \cite{Fio10a,Fio10e,FLS}.

Let $G\leq A_{\G}$ be convex-cocompact, equipped with the induced coarse median $m_{\G}$. If $H\leq G$ is convex-cocompact and $\varphi\in\Aut(G,m_{\G})$, then $\varphi(H)$ is again convex-cocompact \cite[Corollary~3.3]{Fio10a}. Because of convex-cocompactness, $H$ inherits a coarse median structure from $G$ and, if $\varphi(H)=H$, then the restriction $\varphi|_H$ is again coarse-median preserving.

\begin{rmk}\label{rmk:no_skewing}
    Coarse-median preserving automorphisms do not skew the factors of a direct product. More precisely, this means the following: Suppose that $G=G_1\x\dots\x G_k\x A$, where $A$ is free abelian, while the $G_i$ are convex-cocompact, directly indecomposable and have trivial centre. If we have $\varphi\in\Aut(G,m_{\G})$, then $\varphi$ permutes the subgroups $G_i$, and the restriction $\varphi|_A\in\Aut(A)$ has finite order. (See \cite[Lemma~3.5]{Fio11a} and \cite[Definition~2.23]{Fio10e}, and note that ``orthogonality'' is preserved by coarse-median preserving automorphisms.)
\end{rmk}

The following is the most important property of coarse-median automorphisms for this article: this is the single reason why \Cref{thmintro:cmp_aut} is stronger than \Cref{thmintro:general_aut}. In turn, this derives from the fact that degenerations of coarse-median preserving automorphisms are particularly well-behaved, by \Cref{thm:10e-}(2a), and thus point-stabilisers of the degeneration are convex-cocompact in $G$, by \Cref{thm:acc_implies_nice}.

For simplicity, set $\mc{B}_{\rm top}(\phi):=\mc{B}(\overline{\mf{o}}_{\rm top}(\phi))$ and $\mc{B}^{\om}_{\rm top}(\phi)=\mc{B}^{\om}(\mf{o}^{\om}_{\rm top}(\phi))$.

\begin{lem}\label{lem:B_om}
    The following statements hold for every $\phi\in\Out(G,m_{\G})$ and every non-principal ultrafilter $\om$.
    \begin{enumerate}
        \item Each $B\in\mc{B}^{\om}_{\rm top}(\phi)$ is contained in a maximal element of $\mc{B}^{\om}_{\rm top}(\phi)$.
        \item There are only finitely many $G$--conjugacy classes of maximal subgroups in $\mc{B}^{\om}_{\rm top}(\phi)$ and each of them is convex-cocompact.
        \item If $B,B'$ are distinct maximal subgroups in $\mc{B}^{\om}_{\rm top}(\phi)$, then the intersection $B\cap B'$ is contained in an element of $\mc{Z}(G)\setminus\{G\}$.
    \end{enumerate}
\end{lem}
\begin{proof}
    We suppose that $\phi\in\Out(G,m_{\G})$ has infinite order, otherwise $\mc{B}^{\om}_{\rm top}(\phi)$ is empty and the lemma is vacuously true.

    Let $G\acts\X_{\om}$ be the degeneration determined by $\phi$ and the ultrafilter $\om$. Recall that $\X_{\om}$ equivariantly embeds in a finite product of $\R$--trees $\prod_{v\in\G}T^v_{\om}$, and a subgroup of $G$ is elliptic in $\X_{\om}$ if and only if it is elliptic in all $T^v_{\om}$. Since $\phi$ is coarse-median preserving, \Cref{thm:10e-}(2a) shows that all arc-stabilisers of the minimal subtrees $\Min(G,T^v_{\om})$ lie in $\mc{Z}(G)$. By \Cref{lem:beat_vs_degeneration}, a subgroup of $G$ lies in $\mc{B}^{\om}_{\rm top}(\phi)$ if and only if it fixes a point of $\X_{\om}$.

    Let $\mc{I}$ be the collection of subgroups of the form $\bigcap_{v\in\G}P_v$, where each $P_v$ is the $G$--stabiliser of a point of $\Min(G,T^v_{\om})$. Point-stabilisers of $\X_{\om}$ fix points in each of the trees $T^v_{\om}$, and each point-stabiliser of $T^v_{\om}$ fixes a point in $\Min(G,T^v_{\om})$. Thus, point-stabilisers of $\X_{\om}$ are contained in elements of $\mc{I}$ and, conversely, each element of $\mc{I}$ is elliptic in $\X_{\om}$.

    By \Cref{thm:acc_implies_nice}(1), the point-stabilisers of the trees $\Min(G,T^v_{\om})$ are con\-vex-cocompact, and so the subgroups in $\mc{I}$ are convex-cocompact. Moreover, \Cref{thm:acc_implies_nice}(2) shows that there are only finitely many conjugacy classes of point-stabilisers of $\Min(G,T^v_{\om})$, and so an iterated application of Lem\-ma~\ref{lem:cc_basics}(5) shows that there are only finitely many $G$--conjugacy classes of subgroups in $\mc{I}$. Finally, using \Cref{lem:cc_basics}(3), we conclude that chains of subgroups in $\mc{I}$ have bounded length. This shows that every point-stabiliser of $\X_{\om}$ is contained in a maximal point-stabiliser and that the latter all lie in $\mc{I}$, so they are convex-cocompact and there are only finitely many conjugacy classes of them. The same then holds for maximal elements of $\mc{B}^{\om}_{\rm top}(\phi)$, as they are maximal point-stabilisers of $\X_{\om}$. This proves parts~(1) and~(2).

    As to part~(3), if $B,B'\in\mc{B}^{\om}_{\rm top}(\phi)$ are distinct maximal elements, there exist distinct points $p,p'\in\X_{\om}$ such that $B=G_p$ and $B'=G_{p'}$. Without loss of generality, the projections of $p$ and $p'$ to $T^v_{\om}$ lie in $\Min(G,T^v_{\om})$ for all $v\in\G$ for which $G$ is not elliptic in $T^v_{\om}$. Since $p\neq p'$, there exists $v\in\G$ such that $p$ and $p'$ project to distinct points of $\Min(G,T^v_{\om})$, and hence $B\cap B'$ fixes an arc of $\Min(G,T^v_{\om})$. Thus, $B\cap B'$ is contained in a centraliser.
\end{proof}

Using the previous lemma, we can immediately prove parts~(3) and~(4) of \Cref{thm:cmp_main}, under the assumption that parts~(1) and (2) hold:

\begin{lem}\label{lem:parts3+4}
    Let $\phi\in\Out(G,m_{\G})$ be such that $\mf{g}(\phi)$ contains only finitely many pure growth rates, and such that all other elements of $\mf{g}(\phi)$ are sub-polynomial. Then all the following hold.
    \begin{enumerate}
        \item If $\overline{\mf{o}}_{\rm top}(\phi)$ is not sub-polynomial, then $\overline{\mf{o}}_{\rm top}(\phi)\in\mf{g}(\phi)$.
        \item For any pure growth rate $\mf{o}\in\mf{g}(\phi)$, there are only finitely many conjugacy classes of maximal subgroups in $\mc{B}(\mf{o})$. Each of these is convex-cocompact and its conjugacy class is preserved by a power of $\phi$.
        \item For any $\varphi\in\Aut(G,m_{\G})$ in the outer class $\phi$, all growth rates in the symmetric difference $\mc{G}(\varphi)\triangle\mf{g}(\phi)$ are sub-polynomial.
    \end{enumerate}
\end{lem}
\begin{proof}
    If $\mf{g}(\phi)$ contains at least one pure growth rate, then it has a $\preceq$--maximum $\mf{o}_{\max}$ by our hypotheses. Recalling \Cref{lem:top_exists}(2), we then have $\overline{\mf{o}}_{\rm top}(\phi)\sim_{\om}\mf{o}_{\max}$ for every non-principal ultrafilter $\om$, and this implies that $\overline{\mf{o}}_{\rm top}(\phi)\sim\mf{o}_{\max}\in\mf{g}(\phi)$ (\Cref{rmk:fail->fail_mod_omega}), proving part~(1).

    We now prove part~(2). Since all elements of $\mf{g}(\phi)$ are either pure or sub-polynomial, we have $\mc{B}(\mf{o})=\mc{B}^{\om}(\mf{o})$ for any pure growth rate $\mf{o}\in\mf{G}$ and any non-principal ultrafilter $\om$. In particular $\mc{B}_{\rm top}(\phi)=\mc{B}^{\om}_{\rm top}(\phi)$, and so \Cref{lem:B_om}(2) implies that there are only finitely many conjugacy classes of maximal subgroups in $\mc{B}_{\rm top}(\phi)$, and that these are convex-cocompact. Since $\mc{B}^{\om}_{\rm top}(\phi)$ is $\phi$--invariant by \Cref{rmk:beat_invariant}, each conjugacy class of subgroups in $\mc{B}^{\om}_{\rm top}(\phi)$ is preserved by a power of $\phi$. Thus, part~(2) of the lemma holds for $\mf{o}=\overline{\mf{o}}_{\rm top}(\phi)$. 
    For a general pure growth rate $\mf{o}\in\mf{g}(\phi)$, each maximal subgroup in $\mc{B}(\mf{o})$ is contained in a maximal subgroup in $\mc{B}_{\rm top}(\phi)$, so it suffices to restrict $\phi$ to the latter (after raising $\phi$ to a power) and repeat the previous argument. This procedure eventually terminates, by the assumption that $\mf{g}(\phi)$ contains only finitely many pure growth rates, proving part~(2).

    Finally, we prove part~(3). We can suppose that $\overline{\mf{o}}_{\rm top}(\phi)$ is not sub-polynomial, otherwise \Cref{thm:tame} implies that all elements of $\mc{G}(\varphi)\cup\mf{g}(\phi)$ are sub-polynomial. Now, \Cref{thm:tame} also implies that $\overline{\mc{O}}_{\rm top}(\varphi)\sim\overline{\mf{o}}_{\rm top}(\phi)$, and hence $|\varphi^n(g)|\preceq\overline{\mf{o}}_{\rm top}(\phi)$ for all $g\in G$. The following characterises the elements for which this inequality is strict.
    
    \smallskip
    {\bf Claim.} \emph{If $g_*\in G$ is such that $|\varphi^n(g_*)|\not\sim\overline{\mf{o}}_{\rm top}(\phi)$, then there is a convex-cocompact subgroup $H\leq G$ with $\varphi(H)=H$, $g_*\in H$, and $H\in\mc{B}_{\rm top}(\phi)$.}

    \smallskip\noindent
    \emph{Proof of claim.}
    Since $|\varphi^n(g_*)|\preceq\overline{\mf{o}}_{\rm top}(\phi)$ and $|\varphi^n(g_*)|\not\sim\overline{\mf{o}}_{\rm top}(\phi)$, we can fix a non-principal ultrafilter $\om$ such that $g_*$ lies in the following subset of $G$:
    \[ \Om:=\{g\in G\mid |\varphi^n(g)|\prec_{\om}\mf{o}^{\om}_{\rm top}(\phi)\} .\]
    Note that $\Om$ is a subgroup, as we have $|\varphi^n(gh)|\leq |\varphi^n(g)|+|\varphi^n(h)|$ for any $g,h\in G$. Moreover, we have $\varphi(\Om)=\Om$, as $\varphi$ is bi-Lipschitz with respect to $|\cdot|$. Finally, note that, for each $h\in \Om$, we have $\|\phi^n(h)\|\leq |\varphi^n(h)|\prec_{\om}\mf{o}^{\om}_{\rm top}(\phi)$, and hence $\Om\in\mc{B}^{\om}_{\rm top}(\phi)=\mc{B}_{\rm top}(\phi)$. 

    Let $Z$ be the intersection of all elements of $\mc{Z}(G)$ containing $\Om$.
    Note that $Z$ is convex-cocompact and $\varphi(Z)=Z$. If $\overline{\mf{o}}_{\rm top}(\phi|_Z)\not\sim\overline{\mf{o}}_{\rm top}(\phi)$, then $Z\in\mc{B}_{\rm top}(\phi)$ and we are done. Otherwise, we have $\Om\in\mc{B}_{\rm top}(\phi|_Z)$. Let $B\leq Z$ be a maximal element of $\mc{B}_{\rm top}(\phi|_Z)$ containing $\Om$, which exists by \Cref{lem:B_om}(1). Since $\mc{B}_{\rm top}(\phi|_Z)$ is $\varphi$--invariant, the subgroup $\varphi(B)$ is also a maximal element of $\mc{B}_{\rm top}(\phi|_Z)$. If we had $\varphi(B)\neq B$, then \Cref{lem:B_om}(3) would imply that $\varphi(B)\cap B$ is contained in a centraliser $Z'\in\mc{Z}(Z)\sq\mc{Z}(G)$ with $Z'\neq Z$, and we would have $\Om\leq Z\cap Z'\lneq Z$, contradicting our choice of $Z$. In conclusion, we have $\varphi(B)=B$, $g_*\in\Om\leq B$ and $B\in\mc{B}_{\rm top}(\phi|_Z)\sq\mc{B}_{\rm top}(\phi)$. Finally, $B$ is convex-cocompact by \Cref{lem:B_om}(2), proving the claim.
    \hfill$\blacksquare$

    \smallskip
    Now, let $\mf{o}_1\prec\dots\prec\mf{o}_m$ be the list of pure growth rates in $\mf{g}(\phi)$, which is finite and contains all non-sub-polynomial elements of $\mf{g}(\phi)$ by hypothesis. If $|\varphi^n(g_*)|\not\sim\overline{\mf{o}}_{\rm top}(\phi)=\mf{o}_m$, the claim (together with \Cref{thm:tame}) implies that $|\varphi^n(g_*)|\preceq\mf{o}_{m-1}$. A repeated application of this argument then yields that, for each $g\in G$, either the growth rate $\big[\,|\varphi^n(g)|\,\big]$ is sub-polynomial or $|\varphi^n(g)|\sim\mf{o}_i$ for some index $1\leq i\leq m$.

    In conclusion, we have shown that each element of the difference $\mc{G}(\varphi)\setminus\mf{g}(\phi)$ is sub-polynomial. Conversely, the fact that the elements of $\mf{g}(\phi)\setminus\mc{G}(\varphi)$ are sub-polynomial is straightforward: for each pure growth rate $\mf{o}\in\mf{g}(\phi)$, we can pick a maximal element $B\in\mc{B}(\mf{o})$, which is convex-cocompact and preserved by a power of $\phi$ by part~(2), so we can apply \Cref{thm:tame} to the restriction to $B$ of a power of $\phi$ to conclude that $\mf{o}\in\mc{G}(\varphi)$. This proves part~(3), completing the proof of the lemma.
\end{proof}

\subsection{The core of the proof}

This subsection is devoted to the proof of \Cref{thm:cmp_main}, which mainly amounts to the following proposition.

\begin{prop}\label{prop:cmp_non-uniform}
    If $\phi\in\Out(G,m_{\G})$, then finitely many growth rates in $\mf{g}(\phi)$ are pure, and all other elements of $\mf{g}(\phi)$ are sub-polynomial.
\end{prop}
\begin{proof}
We prove the proposition by a double induction. The first induction is on the number of vertices of the finite graph $\G$. Thus, suppose that the proposition holds for all coarse-median preserving automorphisms of all convex-cocompact subgroups of proper parabolic subgroups of $A_{\G}$. We can also suppose that $G\not\in\mc{S}(G)$, otherwise it is straightforward to deduce the proposition from the fact that it holds for the virtual direct factors of $G$ (see \cite[Corollary~3.6]{Fio11a} for details). Up to raising $\phi$ to a power, we can further assume that each conjugacy class of subgroups in $\mc{S}(G)$ is $\phi$--invariant. We then define the set of growth rates
\[ \mf{g}_{\rm sing}(\phi):=\bigcup_{S\in\mc{S}(G)}\mf{g}_{\rm pure}(\phi|_S) ,\]
where $\mf{g}_{\rm pure}(\phi|_S)\sq\mf{g}(\phi|_S)$ denotes the subset of pure growth rates. Note that $\mf{g}_{\rm sing}(\phi)\sq\mf{G}$ is finite, since each set $\mf{g}_{\rm pure}(\phi|_S)$ is finite by the (first) inductive assumption and $\mc{S}(G)$ is finite up to conjugacy.

The second level of the induction is on the cardinality $\#\mf{g}_{\rm sing}(\phi)$. We have $\#\mf{g}_{\rm sing}(\phi)=0$ precisely when $\phi$ grows sub-polynomially on all singular subgroups of $G$. We will treat this base step simultaneously with the inductive step, as the argument is the same. Thus, suppose that the proposition also holds for all coarse-median preserving outer automorphisms $\psi$ of convex-cocompact subgroups of $A_{\G}$ 
such that $\#\mf{g}_{\rm sing}(\psi)<\#\mf{g}_{\rm sing}(\phi)$.

We proceed by distinguishing three cases. Throughout, we assume that $\overline{\mf{o}}_{\rm top}(\phi)$ is not sub-polynomial, otherwise the proposition is trivially true.

\smallskip\noindent
{\bf Case~1.} \emph{We have $\overline{\mf{o}}_{\rm top}(\phi)\sim\overline{\mf{o}}_{\rm sing}(\phi)$.}

\smallskip
In the base step, this means that $\overline{\mf{o}}_{\rm top}(\phi)$ is sub-polynomial, which we have already ruled out. For the inductive step, we argue as follows.

First, we claim that $\overline{\mf{o}}_{\rm top}(\phi)$ is pure and that $\overline{\mf{o}}_{\rm top}(\phi)\in\mf{g}_{\rm sing}(\phi)$. Indeed, each growth rate $\overline{\mf{o}}_{\rm top}(\phi|_S)$ with $S\in\mc{S}(G)$ is either pure or sub-polynomial by the (first) inductive hypothesis, recalling that $\overline{\mf{o}}_{\rm top}(\phi|_S)\in\mf{g}(\phi|_S)$ by \Cref{lem:parts3+4}(1). Thus, $\overline{\mf{o}}_{\rm sing}(\phi)$ is either pure or sub-polynomial, as it is defined as a finite sum of such growth rates (\Cref{defn:osing}). Since $\overline{\mf{o}}_{\rm top}(\phi)\sim\overline{\mf{o}}_{\rm sing}(\phi)$ by the hypothesis of Case~1, we conclude that this growth rate is pure and that it must thus coincide with $\overline{\mf{o}}_{\rm top}(\phi|_S)$ for some $S\in\mc{S}(G)$. Hence $\overline{\mf{o}}_{\rm top}(\phi)\in\mf{g}_{\rm sing}(\phi)$ as claimed.

Next, we claim that all growth rates in $\mf{g}(\phi)$ are either pure or sub-polynomial. For this, consider an element $g\in G$ and recall that $\|\phi^n(g)\|\preceq\overline{\mf{o}}_{\rm top}(\phi)$. If $\|\phi^n(g)\|\sim\overline{\mf{o}}_{\rm top}(\phi)$, we are done. Otherwise, there exists a non-principal ultrafilter $\om$ such that $\|\phi^n(g)\|\prec_{\om}\mf{o}^{\om}_{\rm top}(\phi)$, and hence we have $g\in B$ for a maximal element $B\in\mc{B}^{\om}_{\rm top}(\phi)$ by \Cref{lem:B_om}(1). The subgroup $B$ is convex-cocompact and, up to raising $\phi$ to a power, the conjugacy class of $B$ is $\phi$--invariant by \Cref{lem:B_om}(2). Since every singular subgroup of $B$ is contained in a singular subgroup of $G$, we have $\mf{g}_{\rm sing}(\phi|_B)\sq\mf{g}_{\rm sing}(\phi)$. Since $B\in\mc{B}^{\om}_{\rm top}(\phi)$, we have $\overline{\mf{o}}_{\rm top}(\phi)\not\in\mf{g}_{\rm sing}(\phi|_B)$. At the same time, we have seen above $\overline{\mf{o}}_{\rm top}(\phi)\in\mf{g}_{\rm sing}(\phi)$, so we conclude that $\#\mf{g}_{\rm sing}(\phi|_B)<\#\mf{g}_{\rm sing}(\phi)$ and the (second) inductive hypothesis implies that all elements of $\mf{g}(\phi|_B)$ are either pure or sub-polynomial. Since $g\in B$, this proves that the growth rate $\big[\,\|\phi^n(g)\|\,\big]$ is pure or sub-polynomial, as claimed.

Finally, we show that $\mf{g}(\phi)$ contains only finitely many pure growth rates. By the previous claim, we have $\mc{B}^{\om}_{\rm top}(\phi)=\mc{B}_{\rm top}(\phi)$ for all non-principal ultrafilters $\om$. Thus, \Cref{lem:B_om}(2) shows that $\mc{B}_{\rm top}(\phi)$ consists of finitely many conjugacy classes of convex-cocompact subgroups, each preserved by $\phi$ (after raising $\phi$ to some power). Now, each $\mf{o}\in\mf{g}(\phi)$ satisfies either $\mf{o}\sim\overline{\mf{o}}_{\rm top}(\phi)$ or $\mf{o}\in\mf{g}(\phi|_B)$ for some $B\in\mc{B}_{\rm top}(\phi)$, and each set $\mf{g}(\phi|_B)$ contains only finitely many pure growth rates by the (second) inductive assumption, as we have seen that $\#\mf{g}_{\rm sing}(\phi|_B)<\#\mf{g}_{\rm sing}(\phi)$. This proves the proposition in Case~1.

\smallskip\noindent
{\bf Case~2.} \emph{We have $\overline{\mf{o}}_{\rm sing}(\phi)\prec\overline{\mf{o}}_{\rm top}(\phi)$ and $G$ is $1$--ended.}

\smallskip
In the base step when $\overline{\mf{o}}_{\rm sing}(\phi)$ is sub-polynomial, \Cref{prop:pure_above_osing}(2) shows that $\mf{g}(\phi)$ contains only finitely many non-sub-polynomial growth rates and that these are all pure.

In general, recall that $\mc{K}_{\rm sing}(\phi)$ is the family of subgroups whose elements satisfy $\|\phi^n(k)\|\preceq\overline{\mf{o}}_{\rm sing}(\phi)$. By \Cref{cor:K_sing}, every subgroup in $\mc{K}_{\rm sing}(\phi)$ is contained in a maximal such subgroup, the latter are convex-cocompact, and they fall into finitely many conjugacy classes. Raising $\phi$ to a power, we can assume that each conjugacy class of maximal elements of $\mc{K}_{\rm sing}(\phi)$ is $\phi$--invariant. For each maximal element $K\in\mc{K}_{\rm sing}(\phi)$, we then have 
\begin{equation}\label{eq:sing_phi_K}
    \overline{\mf{o}}_{\rm sing}(\phi|_K)\preceq\overline{\mf{o}}_{\rm top}(\phi|_K)\preceq\overline{\mf{o}}_{\rm sing}(\phi) .
\end{equation}
Recall that $\overline{\mf{o}}_{\rm sing}(\phi)$ is either sub-polynomial or pure, by the (first) inductive assumption. If $\overline{\mf{o}}_{\rm sing}(\phi)$ is pure, it is realised on some $S\in\mc{S}(G)$ and so it lies in $\mf{g}_{\rm sing}(\phi)$. In conclusion, if $\overline{\mf{o}}_{\rm sing}(\phi|_K)$ is not sub-polynomial and if we do not have $\#\mf{g}_{\rm sing}(\phi|_K)<\#\mf{g}_{\rm sing}(\phi)$, then $\overline{\mf{o}}_{\rm sing}(\phi)\in\mf{g}_{\rm sing}(\phi|_K)$ and \Cref{eq:sing_phi_K} shows that $\overline{\mf{o}}_{\rm sing}(\phi|_K)\sim\overline{\mf{o}}_{\rm top}(\phi|_K)$. Now, either by the (second) inductive hypothesis or by Case~1, we know that $\phi|_K$ satisfies the proposition for every maximal element $K\in\mc{K}_{\rm sing}(\phi)$.

\Cref{cor:pure_above_osing} shows that, with the exception of finitely many pure growth rates, each $\mf{o}\in\mf{g}(\phi)$ that is not sub-polynomial satisfies $\mf{o}\preceq\overline{\mf{o}}_{\rm sing}(\phi)$. In turn, if $\mf{o}\preceq\overline{\mf{o}}_{\rm sing}(\phi)$, then $\mf{o}$ is realised on some maximal element $K\in\mc{K}_{\rm sing}(\phi)$ and we know that the proposition holds for $\phi|_K$. It follows that the proposition holds for $\phi$ as well, and this concludes Case~2.

\smallskip\noindent
{\bf Case~3.} \emph{The group $G$ is freely decomposable.}

\smallskip
Let $\mc{I}$ be the collection of freely indecomposable free factors of $G$. Up to raising $\phi$ to a power, we can assume that each conjugacy class in $\mc{I}$ is $\phi$--invariant. For each $I\in\mc{I}$, every singular subgroup of $I$ is contained in a singular subgroup of $G$, and thus we have $\mf{g}_{\rm sing}(\phi|_I)\sq\mf{g}_{\rm sing}(\phi)$. By Cases~1 and~2 above, it follows that the restriction $\phi|_I$ satisfies the proposition for each $I\in\mc{I}$. Letting $\mf{g}_{\rm ind}(\phi)\sq\mf{g}(\phi)$ be the set of all pure growth rates of the restrictions $\phi|_I$ as $I$ varies in $\mc{I}$, it follows that the set $\mf{g}_{\rm ind}(\phi)$ is finite.

Let $\Lambda\sq\R_{>1}$ be a finite set and $P\in\N$ an integer such that every growth rate in $\mf{g}_{\rm ind}(\phi)$ is of the form $[n^a\nu^n]$ for some $\nu\in\Lambda$ and some integer $a$ with $0\leq a\leq P$. Recall that $\Gr(G)$ denotes the Grushko rank of $G$. We complete Case~3 by proving the following.

\smallskip
{\bf Claim.} \emph{Each growth rate in the set $\mf{g}(\phi)$ is either sub-polynomial or of the form $[n^a\nu^n]$, for some integer $0\leq a\leq P+\Gr(G)$ and some $\nu\in\Lambda'$, where $\Lambda'$ is a set satisfying $\Lambda\sq\Lambda'\sq\R_{>1}$ and $\#(\Lambda'\setminus\Lambda)\leq\Gr(G)$.}

\smallskip\noindent
\emph{Proof of claim.}
We prove the claim via a third inductive procedure, on the value of the pair $(\#\mf{g}_{\rm ind}(\phi),\Gr(G))$, ordered lexicographically so that the first entry takes precedence over the Grushko rank $\Gr(G)$. The base step is trivial, so we only consider the inductive step.

Up to raising $\phi$ to a power, there is a (possibly sporadic) factor system $\mc{F}$ for $G$ such that $\phi$ is a fully irreducible element of $\Out(G,\mc{F})$. For each $F\in\mc{F}$, we have $\mf{g}_{\rm ind}(\phi|_F)\sq\mf{g}_{\rm ind}(\phi)$ and $\Gr(F)<\Gr(G)$, and thus the claim holds for the restriction $\phi|_F$ by the (third) inductive assumption. Let $\Lambda_0\sq\R_{>1}$ and $P_0\in\N$ be such that each growth rate in the union $\bigcup_{F\in\mc{F}}\mf{g}(\phi|_F)$ is either sub-polynomial or of the form $[n^a\nu^n]$ with $\nu\in\Lambda_0$ and $0\leq a\leq P_0$. 

If $\mc{F}$ is non-sporadic, we can consider a relative train-track map for $\phi$ and deduce that each growth rate in $\mf{g}(\phi)$ is either sub-polynomial, or of the form $[n^a\nu^n]$ with $\nu\in\Lambda_0$ and $0\leq a\leq P_0+1$, or of the form $[\lambda^n]$ for a single new number $\lambda>1$. See \cite[Proposition~5.2]{Fio11a}(3) for details.

Suppose instead that $\mc{F}$ is sporadic, so that there exists a $\phi$--invariant free splitting $G\acts T$ whose nontrivial vertex groups are the elements of $\mc{F}$. The growth rate $\overline{\mf{o}}_{\rm top}(\phi)$ then equals the fastest pure growth rate in the union $\bigcup_{F\in\mc{F}}\mf{g}(\phi|_F)$; see \cite[Proposition~5.2]{Fio11a}(3) for details. Each element of $\mf{g}(\phi)$ other than $\overline{\mf{o}}_{\rm top}(\phi)$ is realised on a maximal subgroup $B\in\mc{B}^{\om}_{\rm top}(\phi)$, for some non-principal ultrafilter $\om$. Moreover, $B$ is convex-cocompact and without loss of generality $\phi$--invariant (\Cref{lem:B_om}). By Kurosh's theorem, the freely indecomposable factors of $B$ are contained in the freely indecomposable factors of $G$, so we have $\mf{g}_{\rm ind}(\phi|_B)\sq\mf{g}_{\rm ind}(\phi)$. At the same time, we have $\overline{\mf{o}}_{\rm top}(\phi)\in\mf{g}_{\rm ind}(\phi)\setminus\mf{g}_{\rm ind}(\phi|_B)$ by construction, and so we can conclude by the (third) inductive assumption
\hfill$\blacksquare$

\smallskip
The claim concludes Case~3, completing the proof of \Cref{prop:cmp_non-uniform}.
\end{proof}

We can finally piece together the proof of \Cref{thm:cmp_main}. 

\begin{proof}[Proof of \Cref{thm:cmp_main}]
    Parts~(1)--(4) of the theorem mostly follow from the combination of \Cref{prop:cmp_non-uniform} and \Cref{lem:parts3+4}. The only statement not covered by this is the portion of part~(2) claiming that all sub-polynomial growth rates in $\mf{g}(\phi)$ are $\preceq[n^p]$, for an integer $p$ depending on $\phi$. To prove this missing statement, let $\mf{o}_{\min}$ be the slowest pure growth rate in $\mf{g}(\phi)$. Let $B_1,\dots,B_m$ be representatives of the finitely many $G$--conjugacy classes of maximal subgroups in $\mc{B}(\mf{o}_{\min})$. The $B_i$ are convex-cocompact and their conjugacy classes are $\phi$--invariant (after raising $\phi$ to a suitable power). Moreover, each sub-polynomial growth rate in $\mf{g}(\phi)$ is realised on one of the $B_i$, and it is therefore $\preceq\overline{\mf{o}}_{\rm top}(\phi|_{B_i})$. To conclude, it now suffices to observe that the rates $\overline{\mf{o}}_{\rm top}(\phi|_{B_i})$ are all sub-polynomial by \Cref{thm:tame}, as all elements of the $B_i$ grow sub-polynomially under $\phi$. Thus, $\overline{\mf{o}}_{\rm top}(\phi|_{B_i})\preceq n^{p_i}$ for some integers $p_i\in\N$ and we can take $p:=\max_ip_i$.

    We are left to prove part~(5) of the theorem, which amounts to realising each pure growth rate $\mf{o}\in\mf{g}(\phi)$ on a particular subgroup $H\leq G$. Up to replacing $G$ with a maximal $\mf{o}$--controlled subgroup, we can suppose that $\mf{o}\sim\overline{\mf{o}}_{\rm top}(\phi)$. The required subgroup is then provided by \Cref{thm:tame}(3), except that we need to avoid Case~(iii) of that result, where $H$ virtually splits as a direct product $H'\x\Z^m$ for some $m\geq 1$. Without loss of generality, $H'$ has trivial centre, otherwise it has a virtual $\Z$--factor, which can be incorporated into $\Z^m$. Up to raising $\phi$ to a power, $\phi$ has a representative $\varphi\in\Aut(G)$ that preserves $H'\x\Z^m$. By \Cref{rmk:no_skewing}, the fact that $\varphi$ is coarse-median preserving implies that $\varphi(H')=H'$ and, up to raising $\varphi$ to a further power, the restriction $\varphi|_{\Z^m}$ is the identity. This implies that $\overline{\mf{o}}_{\rm top}(\phi)\sim\overline{\mf{o}}_{\rm top}(\phi|_{H'})$. We can now replace $\phi$ with $\phi|_{H'}\in\Out(H')$ and repeat the argument, which eventually terminates because $\ar(H')\leq\ar(G)-m<\ar(G)$.
\end{proof}

\bibliography{./mybib}
\bibliographystyle{alpha}
\end{document}